\documentclass[english]{article}

\usepackage{preamble}
\usepackage{mymacros}

\begin{document}

\title{Characters and transfer maps via categorified traces}
\author{Shachar Carmeli, Bastiaan Cnossen, Maxime Ramzi and Lior Yanovski}
\maketitle

\begin{abstract}
We develop a theory of generalized characters of local systems in $\infty$-categories, which extends classical character theory for group representations and, in particular, the induced character formula. A key aspect of our approach is that we utilize the interaction between traces and their categorifications.
We apply this theory to reprove and refine various results on the composability of Becker-Gottlieb transfers, the Hochschild homology of Thom spectra, and the additivity of traces in stable $\infty$-categories.
\end{abstract}

\vspace{40pt}
\begin{figure}[H]
\centering{}\includegraphics[scale=0.7]{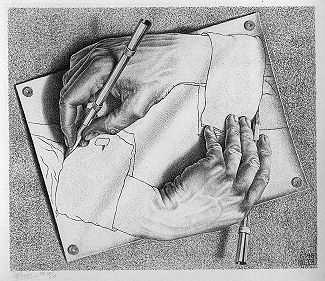}\caption*{\textit{Drawing Hands} by M. C. Esher. Taken from \textit{The Magic of M. C. Escher} by J. L. Locher, Harry N. Abrams, Inc. ISBN 0-8109-6720-0.}
\end{figure}

\newpage
\tableofcontents
\section{Introduction}
\subsection{Overview}
\subsubsection{Traces and characters}
In linear algebra, an important numerical invariant one can assign to an endomorphism $f\colon V \to V$ of a finite-dimensional $\field$-vector space $V$ is its \textit{trace} $\gentr{f}{V} \in \field$. For example, the trace of the identity map on $V$ is the dimension of $V$, regarded as a scalar in $\field$. By combining various traces together, one obtains the notion of a \textit{character} in representation theory: if $V$ is a finite-dimensional linear representation of a finite group $G$, its character $\chi_V\colon G\to \field$ is the conjugation-invariant $\field$-valued function on $G$ defined by $\chi_V(g) = \gentr{g}{V}$.
This character encodes useful information about $V$, especially when $|G|$ is invertible in $\field$. For instance, there is a simple formula for the dimension of the space of co-invariants $V_G$ of the representation in terms of its character:
\[
    \dim_\field(V_G) = 
    \frac{1}{|G|}\sum_{g\in G} \gentr{g}{V} =
    \sum_{[g]\in G/\mathrm{conj}} \frac{1}{|C(g)|} \chi_V([g]) \qin \field.
\]
This formula is a special case of a more general \emph{induced character formula}: given a (not-necessarily injective) group morphism $f\colon G \to H$, the induced character formula expresses the character $\chi_{\Ind_G^HV}$ of the induced representation $\Ind_G^HV$ in terms of the character $\chi_V$. The special case $H=\{1\}$ recovers the dimension formula above.

The goal of this article is to generalize this story in several directions. To begin with, note that a finite-dimensional vector space is a dualizable object in the symmetric monoidal category $\Vect_\field$ of vector spaces over $\field$. We will study more generally dimensions, traces and characters of dualizable objects in an arbitrary presentably symmetric monoidal $\infty$-category $\cC$, encompassing for example the notions of Euler characteristics and Lefschetz numbers in homotopy theory. Next, observe that a representation $V$ of the group $G$ can be encoded via a functor $BG\to \Vect_\field$ from the classifying space of $G$ to the category of vector spaces over $\field$. We will consider more generally functors $A\to \cC$ from an arbitrary space $A$, also known as $\cC$-valued \textit{local systems} over $A$; in the case $\Cc = \Sp$ these are known as \textit{parametrized spectra} over $A$. If $A$ is connected, a local system over $A$ can also be thought of as a $\Cc$-valued representation of the (not necessarily discrete) loop group $\Omega_a A$ with respect to some basepoint $a$.

Note that conjugacy classes of elements of $G$ can be identified with connected components of the free loop space $LBG$ of its classifying space $BG$. Accordingly, the conjugation-invariant character $\chi_V\colon G \to k$ of some $G$-representation $V$ may be identified with a locally constant function $LBG\to \field$. Generalizing this, we will assign to every local system of dualizable objects $V\in \cC^A$ a \emph{character}
\[
    \chi_V\colon LA\too \End(\one_\cC),
\] 
which is a map from the free loop space $LA = \Map(S^1,A)$ of $A$ to the endomorphism space of the unit object of $\cC$. This character encodes the traces of the ``monodromy actions'' of free loops in $A$ on the fibers of the local system $V$. Here, the endomorphism space $\End(\one_\cC)$ plays the role of the set of scalars $\field$.

The induced character formula, and in particular the formula above for the dimension of the coinvariants $V_G$, involves division by $|G|$, and hence is applicable only if $|G|$ is invertible in $\field$. This problem becomes more severe in the derived setting: if $V$ is a dualizable object in the derived $\infty$-category $D(\field)$ of $\field$ equipped with a $G$-action, its object of coinvariants $V_{hG}$ might not even be dualizable anymore if $|G|$ is not invertible in $\field$. To deal with this problem for general $\cC$, we introduce the notion of a \emph{$\cC$-adjointable} maps of spaces $f \colon A \to B$, see \Cref{def:CAdjointability}. This property in particular ensures that for a pointwise dualizable local system $V\colon A \to \cC$, its left Kan extension $f_!V \colon B \to \cC$ along $f$ (playing the role of the induced representation) is pointwise dualizable as well, so that we may form its character. In such a situation, in the spirit of \cite{PontoShulman}, we will provide an induced character formula which expresses the character of the left Kan extension $f_!V\colon B \to \cC$ in terms of the character of the original local system $V\colon A \to \cC$.

Finally, the theory of characters can be generalized in a third direction, which is needed in some of the applications. The definition of the trace can naturally be extended to \textit{generalized endomorphisms} of the form $f\colon X\to X\otimes Y$, which results in a \textit{generalized trace map} $\gentr{f}{X}\colon \one\to Y$. Given a space $A$ and an $A$-indexed family of generalized endomorphisms $f_a\colon X_a \to X_a\otimes Y$, we may assign to it a map $\chi_f\colon LA\to \Map(\one,Y)$, called the \textit{generalized character map} of this family. We will develop our induced character formula in this generality.

\subsubsection{Transfers}
One of the applications of our generalized character theory is to the study of \emph{transfer maps} in topology. Recall that a spectrum $E$ represents a homology theory $E_*(-)$ as well as a cohomology theory $E^*(-)$. In particular, we have for any map of spaces $f\colon A \to B$ induced maps on $E$-homology and $E$-cohomology
\[
    E_*(f) \colon E_*(A) \too 
    E_*(B), \qquad \qquad E^*(f) \colon E^*(B) \too E^*(A).
\]

If $f$ has suitably small fibers, we also have a \textit{wrong-way} or \textit{transfer} maps
\[
    f^! \colon E_*(B) \too E_*(A), \qquad \qquad \quad f_!\colon E^*(A) \too 
    E^*(B),
\]
going in the opposite direction. The map $f_!\colon E^*(A) \to E^*(B)$ on $E$-cohomology can be understood as a form of ``integration along the fibers''. These transfer maps can be canonically refined to maps of spectra $f^! \colon E\otimes B\to E\otimes A$ resp.\ $f_!\colon E^A \to E^B$. Transfer maps have a long history in algebraic topology, as described for example in \cite{history}. 
A trace-theoretic perspective on these transfer maps is developed by several authors, including \cite{DoldPuppe,PSLinearity,PSMultiplicativity,BenZviNadler}.

In this article, we study transfer maps in an arbitrary presentably symmetric monoidal $\infty$-category $\cC$. Recall that $\Cc$ admits a unique symmetric monoidal colimit-preserving functor $\unit_{\Cc}[-]\colon \Spc \to \Cc$, denoted $A \mapsto \unit_{\Cc}[A]$, which sends a space $A$ to the colimit over $A$ of the constant $A$-indexed diagram on $\unit_{\Cc}$. For example, for the $\infty$-category of spectra this is the unreduced suspension spectrum functor $\Sigma^{\infty}_+\colon \Spc \to \Sp$, while for the derived $\infty$-category of a ring $R$ this is the singular chain functor $C_{\bullet}(-;R)\colon \Spc \to D(R)$. Given a $\cC$-adjointable map of spaces $f\colon A\to B$, we will define a transfer map of the form $f^! \colon \one_\cC[B]\to \one_\cC[A]$. This transfer $f^!$ will by construction be a $B$-shaped colimit of traces of generalized endomorphisms, and hence can be studied using the induced character formula. This recovers a description of the $\Cc$-linear Becker-Gottlieb transfer $f^!$ in terms of the $\Cc$-linear free loop space transfer of $f$, given for $\cC = \Sp$ by \cite{LM}.

Using the interpretation of the transfer maps in terms of characters, we address the problem of \emph{composability} of transfer maps. We show by an example that, for two composable $\cC$-adjointable maps $f$ and $g$, the relation $f^!g^! \simeq (gf)^!$ does \textit{not} hold in full generality. We identify various conditions under which this relation \textit{does} hold, recovering and refining some of the results of \cite{LMS} and \cite{kleinMalkiewich2022Corrigendum}, and some of the related results of \cite{LM}. Despite various attempts to answer it, the question of composability of the transfers for maps of spaces with compact (sometimes called `finitely dominated') fibers in the case of $\cC=\Sp$ is still open. We hope that our methods and results could shed some new light on this problem.
For example, Klein, Malkiewich and the third author \cite{KMR} use it to prove functoriality of the Becker-Gottlieb transfers at the level of $\pi_0$.

\subsubsection{Categorified traces}

For an $\EE_\infty$-ring spectrum $E$, the $E$-cohomology $E^A$ of a space $A$ admits a natural categorification, namely the $\infty$-category $(\Mod_E)^A$ of local systems of $E$-modules on $A$. This categorification of cohomology also admits natural right and wrong-way maps: for a map of spaces $g\colon A \to B$, we have the restriction functor
\[
    g^* \colon (\Mod_E)^B \too (\Mod_E)^A,
\]
along with its right adjoint
\[
    g_*\colon (\Mod_E)^A \too (\Mod_E)^B,
\]
given by right Kan extension along $g$. Objects $M \in (\Mod_E)^A$ are local systems of $E$-modules on $A$, and we think of the right Kan extension $g_*M \in (\Mod_E)^B$ as fiberwise cohomology with local coefficients, i.e.\ as some sort of ``categorified integration along the fibers''.
When $B = \pt$, the object $g_*g^*E = E^A$ is precisely the $E$-cohomology of $A$. 
To summarize, cohomology both \emph{admits} wrong-way maps and can be \emph{seen} as a wrong-way map. 

In a similar way, the key tool in studying our generalized characters and associated induced character formulas is the interaction between traces and categorification. The generality of the formalism of symmetric monoidal traces allows us to also apply it to symmetric monoidal $(\infty,2)$-categories, such as the $(\infty,2)$-category of $E$-linear presentable $\infty$-categories. While in symmetric monoidal $(\infty,1)$-categories, traces can only be functorial in equivalences, in $(\infty,2)$-categories they are functorial more generally in \emph{left adjoints}. 
We will exploit this extra functoriality of categorified traces to understand the interaction of traces at a lower categorical level with colimits.

In the case of the symmetric monoidal $(\infty,2)$-category $\PrL_\st$ of presentable stable $\infty$-categories the dimension of dualizable objects recovers a well-known invariant: topological Hochschild homology, classically defined in terms of the cyclic bar construction.
The interpretation of topological Hochschild homology as a trace, when combined with the categorification of the induced character formula, allows us to compute the topological Hochschild homology of Thom spectra, recovering and generalizing the main result of \cite{BCS2010THHofThomSpectra}.

\subsection{Main results}
We shall now explain our approach to generalized character theory and its applications in more detail. For the remainder of this subsection, we fix a presentably symmetric monoidal $\infty$-category $\cC$. We use the following notational convention to distinguish traces in $\Cc$ from their categorifications:
\begin{itemize}
    \item Given an endomorphism $f$ of a dualizable object $X$ in $\Cc$, we denote its trace in $\Cc$ by $\tr(f) \in \End_{\Cc}(\unit_{\Cc})$, and we write $\dim(X) := \tr(\id_X)$ for its dimension;
    \item Given a $\Cc$-linear endomorphism $F\colon \cD \to \cD$ of a dualizable object $\cD$ in the symmetric monoidal $(\infty,2)$-category $\Mod_{\cC}$ of $\cC$-linear presentable $\infty$-categories, we denote its trace in $\Mod_{\Cc}$, called the \textit{$\Cc$-linear trace}, by $\Tr_{\Cc}(F) \in \Fun_{\Cc}(\Cc,\Cc) \simeq \Cc$. When $F = \id_{\cD}$ we will alternatively write $\Tr_{\cC}(\cD)$ for $\Tr_{\Cc}(\id_{\cD})$.
\end{itemize}

Consider a space $A$ and let $V\colon A \to \Cc$ be a pointwise dualizable local system, in the sense that $V_a \in \Cc$ is dualizable for all $a \in A$. Every free loop $\gamma \in LA := \Map(S^1,A)$ induces an automorphism $\gamma\colon V_{\gamma(0)} \to V_{\gamma(0)}$ in $\Cc$, and the traces of all these automorphisms assemble into the \textit{character} of $V$:
\[
    \chi_V\colon LA \to \End(\unit); \qquad \gamma \mapsto \gentr{\gamma}{V}.
\]
Our first main result provides a description of this map via categorified traces. The $\infty$-category of local systems $\cC^A$ is a dualizable object in $\Mod_{\cC}$, and its $\cC$-linear trace is given by 
\[
    \Tr_{\Cc}(\cC^A) \simeq \one[LA]
    \qin \cC,
\]
where $\unit[LA] = \colim_{LA} \unit$. The local system $V\colon A \to \cC$ determines a $\Cc$-linear map $\cC^A \to \cC$ by $\cC$-linear Yoneda extension, which by pointwise dualizability of $V$ is a left adjoint in the $(\infty,2)$-category $\Mod_{\Cc}$. Since the $\Cc$-linear trace is functorial in such left adjoints, this produces a map $\one[LA] \to \one$, or equivalently a map of spaces 
\(
    \chi_V \colon LA \to \End(\one).
\) 
We show that this map is precisely the character of $V$.

In fact, we prove a somewhat more general claim, which is required for some of the applications. A pointwise dualizable local system $V$ equipped with a generalized endomorphism
\[
    f\colon V\to V\otimes A^*Y
\]
determines a morphism $(\cC^A,\id_{\cC^A}) \to (\cC,Y)$ in the $\infty$-category of \emph{traceable endomorphisms} in $\Mod_{\cC}$, see \Cref{def:Dbl_Trl}. Upon applying the $\cC$-linear trace functor, this induces a map $\one[LA]\to Y$, or equivalently a map of spaces $\chi_f \colon LA\to \Map(\one,Y)$.

\begin{thmx}[\Cref{prop:Character_Formula_Coherent}]
Let $V\colon A\to \cC$ be a local system of dualizable objects, and $f\colon V\to V\otimes A^*Y$ a generalized endomorphism, where $Y\in\cC$. 
The map 
\[ 
    \chi_f\colon LA\too \Map(\one,Y)
\]
is the character of $f$, in the sense that $\chi_f(\gamma) \simeq \tr(f\circ \overline{\gamma}\mid V)$, naturally in the loop $\gamma$.\footnote{The appearance of the inverse loop $\overline\gamma$ as opposed to $\gamma$ is a by-product of the implicit identification $A\simeq A\catop$ for $\infty$-groupoids $A$. This can be resolved by working with the free $\cC$-linear $\infty$-category $\cC[A]$ instead of the $\infty$-category of local systems $\cC^A$, as is done in the main text.}
\end{thmx}

While the character $\chi_f$ encodes  $(\infty,1)$-categorical traces, the above description in terms of categorified traces allows us to study it using the calculus of $(\infty,2)$-categorical traces. In particular, it allows us to study the interaction of traces with colimits via a generalized induced character formula, as we will now explain.

Let $g\colon A \to B$ be a \emph{$\cC$-adjointable} map of spaces, in the sense that the restriction functor $g^*\colon \cC^B \to \cC^A$ is a left adjoint in the $(\infty,2)$-category $\Mod_{\cC}$. Given a pointwise dualizable local system $V\colon A \to \cC$ with character $\chi_V \colon LA \to \End(\one)$, the left Kan extension $g_!V\colon B \to \cC$ of $V$ is again pointwise dualizable (\Cref{prop:ColimIsDbl}), and hence also admits a character $\chi_{g_!V}\colon LB \to \End(\one)$. Our induced character formula then says that the character $\chi_{g_!V}$ can be expressed in terms of the character $\chi_V$ as the following composite:
\[
    \chi_{g_!V}\colon \one[LB] \oto{\TrMod_{\cC}(g^*)} 
    \one[LA] \oto{\quad\chi_V\quad} 
    \one.
\]
Here the map $\Tr_\cC(g^*)\colon \one[LB] \too \one[LA]$ is called the \textit{free loop transfer} of $g$, and is obtained by applying the functoriality of $\Cc$-linear traces to the left adjoint $g^*$ in $\Mod_{\Cc}$; see for example \cite{LM} for an alternative description of this transfer map in the case $\cC = \Sp$.

More generally, if the local system $V \in \cC^A$ comes equipped with a generalized endomorphism $f\colon V\to V\otimes A^*Y$, we can induce it along $g$ to get a generalized endomorphism 
\[
    \Ind_g(f)\colon g_!V\to g_!V\otimes B^*Y.
\] 
We obtain a similar expression of the induced character $\chi_{\Ind_g(f)}$ in terms of $\chi_f$:
\begin{thmx}[Induced character formula, \Cref{Ind_Char}]
Let $V\colon A\to \cC$ be a local system of dualizable objects, and let $f\colon V\to V\otimes A^*Y$ be a generalized endomorphism, where $Y\in\cC$. For a $\cC$-adjointable map of spaces $g\colon A\to B$, the induced character $\chi_{\Ind_g(f)}$ is given by the composition
\[
    \one[LB] \oto{\TrMod_{\cC}(g^*)} 
    \one[LA] \oto{\quad\chi_f\quad} 
    Y.
\]
\end{thmx}

An important special case of this theorem is the case where $B$ is a point and $Y$ is the monoidal unit: we may compute the trace of the induced map $A_!f \colon A_!V \to A_!V$ on colimits as the composite
\[
    \tr(A_!f)\colon   
    \unit \xrightarrow{\TrMod_{\cC}(A^*)} 
    \unit[LA] \xrightarrow{\quad\chi_f\quad} \unit.
\]
In particular, this gives a formula of the dimension of a colimit of dualizable objects over a $\cC$-adjointable space.

The induced character formula can itself be categorified. This categorification allows us to recover a result of Blumberg-Cohen-Schlichtkrull on the topological Hochschild homology of Thom spectra \cite{BCS2010THHofThomSpectra}, and extend it from $\Sp$ to an arbitrary presentably symmetric monoidal $\infty$-category $\cC$:
\begin{thmx}[\Cref{thm:BCS_THH_of_Thom}, cf.\ {\cite[Theorem~1]{BCS2010THHofThomSpectra}}]
    Let $G$ be an $\EE_1$-group in $\Spc$, and let $\xi\colon G\to \Pic(\cC)$ be an $\EE_1$-group map. Then the Hochschild homology of the Thom object $\Th_\cC(\xi)$ in $\Alg(\cC)$ is the Thom object of the following composite: 
    \[
        LBG \oto{LB\xi} 
        LB\Pic(\cC) \simeq 
        B\Pic(\cC)\times \Pic(\cC) \oto{\eta + \id} \Pic(\cC),
    \]
    where $\eta \colon B\Pic(\cC) \to \Pic(\cC)$ is the Hopf map.
\end{thmx}

The proof of this theorem uses the following result, which is of independent interest and has previously been sketched by Douglas \cite{Douglas} in the case of $\cC= \Sp$:

\begin{thmx}[{\Cref{thm:Thom_as_colim}, cf.\ \cite[Proposition 3.13]{Douglas}}]
 Let $A$ be a pointed connected space and let $\zeta\colon A \to \Mod_\cC$ be a pointed map. Then there is an equivalence
    \[
        \RMod_{\Th(\Omega \zeta)} \:\simeq\:
        \colim_A \zeta
        \qin \Mod_\cC,
    \]
    where $\Omega\zeta\colon \Omega A\to \Pic(\cC)$ is the $\EE_1$-group map induced by $\zeta$.
\end{thmx}

Another application of our methods is to study Becker-Gottlieb transfers. 
First, using our induced character formula we reprove a result of Lind and Malkiewich \cite{LM}, expressing the Becker-Gottlieb transfer in terms of the free loop transfer, and again extend it from $\Sp$ to arbitrary $\Cc$:

\begin{thmx}[{\Cref{thm:LM}, cf.\ \cite[Theorem 1.2]{LM}}]
For a $\cC$-adjointable map of spaces $f\colon A \to B$, the Becker-Gottlieb transfer $f^!\colon \unit[B] \to \unit[A]$ is given by the following composite:
\[
    \unit[B] \oto{c} \unit[LB] \oto{\Tr_{\cC}(f^*)} \unit[LA] \xrightarrow{e} \unit[A].
\]
Here $c\colon B \to LB$ is the inclusion of $B$ into $LB$ as the constant loops and $e\colon LA \to A$ is the evaluation at the basepoint of $S^1$.
\end{thmx}

Second, we address the problem of \emph{composability} of the Becker-Gottlieb transfer, namely the question of whether there is a homotopy $(gf)^!\simeq f^!g^!$ for two composable $\cC$-adjointable maps $f$ and $g$. While we observe that for a general symmetric monoidal $\infty$-category $\cC$ the answer is negative, we discuss a number of cases where the answer is positive, see \Cref{thm:compBG}. 

As a final application of the calculus of categorified traces, we reprove the additivity of traces in symmetric monoidal stable $\infty$-categories. 

\begin{thmx}[\Cref{thm:Additivity_Traces}, cf.\ \cite{may2001additivity,PSLinearity,ramzi2021additivity}]
    Let $\cC$ be a stable presentably symmetric monoidal $\infty$-category, and let $X_1 \xrightarrow{\varphi} X_2 \xrightarrow{\psi} X_3$ be a bifiber sequence of dualizable objects in $\cC$. Let $f_i\colon Z \tensor X_i \to X_i \tensor Y$ be morphisms fitting in a bifiber sequence in $\cC^{[1]}$ of the form\footnote{Note that formally this bifiber sequence is a cube $[1]^3 \to \Cc$, encoding compatibility with the null-homotopies.}
    \[
    \begin{tikzcd}
        Z \tensor X_1 \dar[swap]{f_1} \rar{1_Z \tensor \varphi} & Z \tensor X_2 \dar{f_2} \rar{1_Z \tensor \psi} & Z \tensor X_3 \dar{f_3} \\
        X_1 \tensor Y \rar{\varphi \tensor 1_Y} & X_2 \tensor Y \rar{\psi \tensor 1_Y} & X_3 \tensor Y,
    \end{tikzcd}
    \]
    where the top and bottom sequences are obtained from the original bifiber sequence by tensoring with $Z$ resp.\ $Y$. Then, there is a homotopy 
    \[
        \gentr{f_2}{X_2} \simeq \gentr{f_1}{X_1} + \gentr{f_3}{X_3} \qin \Map_{\cC}(Z,Y).
    \]
\end{thmx}

\subsection{Relation to other work}
Our treatment of higher categorical traces follows \cite{HSS}, which builds on foundations from \cite{JFS2017}.

Our induced character formula is largely inspired by the work of Ponto and Shulman \cite{PSLinearity}. Beyond the difference in implementation (derivators versus $\infty$-categories), the main conceptual difference between their work and ours is that we consider indexing diagrams which are $\infty$-groupoids, rather than $1$-categories. In particular, while \cite[Lemma 10.4]{PSLinearity} can be used to obtain an analogue of the above identification $\chi_f(\gamma) \simeq \gentr{f \circ \gamma}{V}$ for every individual free loop $\gamma \in LA$, we obtain this identification \emph{naturally} in $\gamma$. This naturality is necessary to study families of generalized endomorphisms indexed over spaces, which is where our main applications lie. Compared to \cite{PSLinearity}, a limitation of the current work is that our indexing diagrams are $\infty$-groupoids as opposed to categories, or even weights. An extension of this work to the common generalization of $\infty$-categorical weights is the subject of future work. 

Our work is also related to the work of Hoyois, Safronov, Scherotzke and Sibilla \cite{GRR}. Their result \cite[Corollary 5.3]{GRR} looks formally analogous to our induced character formula, though it does not seem to apply to our situation, as the relevant symmetric monoidal functors are typically not rigid (cf.\ \cite[Definition 2.15]{GRR}), nor the relevant adjunctions ambidextrous (cf.\ \cite[Definition 2.1]{GRR}). It would be interesting to find a common framework which unifies both results. 

\subsection{Organization}
This article is structured as follows. Sections \ref{sec:Traces}, \ref{sec:tracesInCorr} and \ref{sec:CLinearCategories} are mostly expository. In \Cref{sec:Traces} we recall the definition of generalized traces and of higher categorical traces following \cite{HSS}. In Section \ref{sec:tracesInCorr}, we specialize the discussion to the symmetric monoidal $(\infty,2)$-category $\Span(\cD)$ of spans and provide a detailed proof of the fact that dimensions in span categories can be \emph{naturally} identified with free loop spaces. In \Cref{sec:CLinearCategories}, we specialize to the symmetric monoidal $(\infty,2)$-category $\Mod_{\cC}(\PrL)$ of $\cC$-linear $\infty$-categories for a presentably symmetric monoidal $\infty$-category $\cC$. There, we study $\cC$-linear $\infty$-categories freely and cofreely generated by spaces, as well as their traces, and we introduce the key notion of $\cC$-adjointable spaces (resp.\ maps). We also recall the trace interpretation of Hochschild homology. 

Section \ref{sec:GeneralizedCharacters} forms the heart of this article. It is there that we introduce \emph{generalized characters}, prove their description in terms of categorified traces, and deduce the induced character formula. It is also in this section that we prove the additivity of generalized traces. 

Sections \ref{sec:BeckerGottlieb} and \ref{sec:ThomSpectra} contain two applications of the theory. In Section \ref{sec:BeckerGottlieb}, we discuss the Becker-Gottlieb transfer. We start by recalling its definition and basic properties, and we then express it in terms of the free loop transfer. We also address the question of its composability: although the equivalence $(gf)^!\simeq f^!g^!$ fails in full generality, we prove that it holds in a number of cases. Finally, in Section \ref{sec:ThomSpectra}, we study Thom objects and their module categories, and we recover and generalize the main result of \cite{BCS2010THHofThomSpectra}, computing the topological Hochschild homology of Thom spectra.

\subsection{Conventions}

We work in the context of $\infty$-categories, defined as quasicategories, and will generally follow Lurie's conventions from \cite{htt} and \cite{HA}. In particular, all categorical notions appearing here, such as (co)limits and adjoints, are meant to be interpreted in the $\infty$-categorical sense. For $(\infty,n)$-categories, we follow \cite{HSS} and work in the setting of Barwick's $n$-fold complete Segal spaces \cite{barwick2021unicity}.\footnote{For $n = 1$, recall that complete Segal spaces are a model for $\infty$-categories, see \cite{JoyalTierney,HebestreitSteinebrunner}.}
\begin{enumerate}
    \item We denote mapping spaces in an $\infty$-category $\cC$ by $\Map_{\cC}(-,-)$ or just by $\Map(-,-)$ if $\cC$ is clear from context. The notation $\Fun(-,-)$ is used for functor $\infty$-categories and $\FunL$ for the full subcategory of $\Fun$ spanned by colimit-preserving functors;
    \item We write $\Hom_{\cE}(-,-)$ for the mapping $\infty$-category in an $(\infty,2)$-category $\cE$;
    \item We let $\Spc$, $\Sp$, $\Cat_\infty$, $\widehat{\Cat}_\infty$ and $\PrL$ denote, respectively, the $\infty$-category of spaces, of spectra, of small $\infty$-categories, of possibly large $\infty$-categories and of presentable $\infty$-categories;
    \item We denote by $(-)^\simeq\colon \Cat_\infty\to \Spc$ the maximal subgroupoid functor. For an $(\infty,n)$-category $\cE$, we let $\iota_k\cE$ denote its maximal sub-$(\infty,k)$-category; in particular, $\iota_0$ and $(-)^\simeq$ agree when restricted to $(\infty,1)$-categories; 
    \item Restriction along a functor $f$ is denoted by $f^*$, and left and right Kan extensions along $f$ are denoted by $f_!$ and $f_*$ respectively;
    \item For a space $A$, we abuse notation and denote by $A$ also the unique map $A\to \pt$, so that for example the functor $A_!\colon \cC^A \to \cC$ denotes $\colim_A$;
    \item For an object $X\in \cC$, we will denote by $X[A]$ the object $A_!A^*X$, i.e.\ the $A$-indexed colimit of the constant diagram $X$; 
    \item Equivalences are denoted by the symbol $\simeq$;
    \item For an $\infty$-category $\Cc$ with finite products, a \textit{commutative monoid} in $\Cc$ is a functor $M\colon \Fin_* \to \Cc$ satisfying the Segal condition $M(n_+) \iso M(1_+)^n$.
    \item We define a \textit{symmetric monoidal $\infty$-category} as a commutative monoid in $\Cat_{\infty}$; by \cite[Remark~2.4.2.6]{HA} this agrees with Lurie's definition. More generally, we define a \textit{symmetric monoidal $(\infty,n)$-category} as a commutative monoid in $\Cat_{(\infty,n)}$.
    \item Given a symmetric monoidal $\infty$-category $\cC$, we denote its monoidal unit by $\one_\cC$, or by $\one$ if there is no ambiguity. We occasionally use $\one_\cD$ to denote the pointing of an $\EE_0$-monoidal category. We write $\pt$ for terminal objects when they exist;
    \item We typically denote tensor products in (symmetric) monoidal $\infty$-categories by $\otimes$ if there is no ambiguity; 
    \item We write $\Alg(\cC)$, $\CAlg(\cC)$ and $\Alg_n(\cC)$ for the $\infty$-categories of algebras, resp.\ commutative algebras, resp.\ $\EE_n$-algebras in a symmetric monoidal $\infty$-category $\cC$. Note that a commutative monoid in an $\infty$-category $\Cc$ with finite products is the same as a commutative algebra in the cartesian monoidal structure on $\Cc$, see \cite[Proposition~2.4.2.5]{HA}.
    \item We write $\RMod_R(\cC)$ and $\LMod_R(\cC)$ for the $\infty$-categories of right resp.\ left modules over an algebra object $R$ in $\cC$. When $R$ is a commutative algebra, we write $\Mod_R(\cC)$.
\end{enumerate}

\subsection{Acknowledgements}
We thank the anonymous referee for their valuable suggestions and corrections. We thank Merlin Christ for pointing out an oversight in our proof of \Cref{Cat_Grp_Alg_Mod}. We would like to thank Alice Hedenlund for helpful discussions related to this work, as well as Cary Malkiewich and John Klein for valuable discussions related to the composability of Becker-Gottlieb transfers. We also thank Andrea Bianchi and Shai Keidar for many helpful comments on an earlier version of this article. S.C. and M.R. are supported by the Danish National Research Foundation through the Copenhagen Centre for Geometry and Topology (DNRF151). At the time of writing, B.C. was supported by the Max Planck Institute for Mathematics (MPIM) in Bonn. L.Y. wishes to thank the MPIM for its hospitality; part of this work was conducted during his stay there.

\section{Higher categorical traces}
\label{sec:Traces}
Let $\cC = (\cC,\tensor,\unit)$ be a symmetric monoidal $\infty$-category. Recall that an object $X$ of $\cC$ is called \textit{dualizable} if there is another object $X^\vee$ and maps $\ev_X\colon X^\vee \tensor X \to \unit$ and $\coev_X\colon \unit \to X \tensor X^\vee$ that satisfy the triangle identities: the composites
\[
    X \xrightarrow{\coev_X \tensor \id_X} X \tensor X^\vee \tensor X \xrightarrow{\id_X \tensor \ev_X} X \qquad \text{ and } \qquad X^\vee \xrightarrow{\id_X \tensor \coev_X} X^\vee \tensor X \tensor X^\vee  \xrightarrow{\ev_X \tensor \id_X} X^\vee
\]
are homotopic to the respective identities. Given an endomorphism $f\colon X\to X$ of a dualizable object $X$, we define its \textit{trace} $\tr(f)\in \Map_{\cC}(\unit,\unit)$ to be the following composition:
\[
    \tr(f)\colon \one \oto{\coev_X}
    X\otimes X^\vee  \oto{\:f\otimes \id\:}
    X \otimes X^\vee\simeq X^\vee\otimes X \oto{\:\ev_X\:}
    \one.
\]
The trace construction admits a certain functoriality, which we will outline first and make precise below. Let us call an endomorphism $f\colon X \to X$ \textit{traceable} if the object $X$ is dualizable in $\cC$, and let $\cC^{\trl} \subseteq \Map(B\NN, \cC)$ denote the subspace of traceable endomorphisms. Then the assignment $f \mapsto \tr(f)$ can be enhanced to a map of spaces
\[
    \tr\colon \cC^{\trl} \to \Map_{\cC}(\unit,\unit) \qin \Spc;
\]
the details for this will be recalled in \Cref{subsec:FunctorialityTraces}. In the case where $\cC$ is a symmetric monoidal $(\infty,2)$-category, the trace construction admits additional functoriality. In this case, the mapping space $\Map_{\cC}(\unit,\unit)$ upgrades to an $\infty$-category, which will be denoted by $\Omega \cC$. Also the space $\cC^{\trl}$ can be upgraded to an $\infty$-category, again denoted by $\cC^{\trl}$, whose morphisms are suitable morphisms between traceable endomorphisms. To describe these, recall that a morphism $\varphi\colon X \to Y$ in $\cC$ is called a \textit{left adjoint in $\cC$} if there exists a morphism $\varphi^r\colon Y \to X$ together with 2-morphisms $\epsilon\colon \varphi \circ \varphi^r \Rightarrow \id_Y$ and $\eta\colon \id_X \Rightarrow \varphi^r \circ \varphi$ satisfying the triangle identities: the composites
\[
    \varphi \xRightarrow{\id \circ \eta} \varphi \circ \varphi^r \circ \varphi \xRightarrow{\epsilon \circ \id} \varphi \qquad \text{ and } \qquad \varphi^r \xRightarrow{\eta \circ \id} \varphi^r \circ \varphi \circ \varphi^r \xRightarrow{\id \circ \epsilon} \varphi^r
\]
are equivalent to the respective identity 2-morphisms. The morphisms in the $\infty$-category $\cC^{\trl}$ between two traceable endomorphisms $f\colon X \to X$ and $g\colon Y \to Y$ are pairs $(\varphi,\alpha)$ consisting of a left adjoint morphism $\varphi\colon X \to Y$ in $\cC$ together with a 2-morphism $\alpha\colon \varphi f \Rightarrow g \varphi$:
\begin{equation}
\label{eq:TraceableMorphism}
\begin{tikzcd}
    X \dar[swap]{f} \rar{\varphi} & Y \dar{g} \\
    X \rar[swap]{\varphi} \urar[Rightarrow, shorten <=4pt, shorten >=4pt]{\alpha} & Y.
\end{tikzcd}
\end{equation}
Having these two $\infty$-categories $\cC^{\trl}$ and $\Omega \cC$, the trace map can be upgraded to a functor of $\infty$-categories
\[
    \tr\colon \cC^{\trl} \to \Omega \cC \qin \Cat_{\infty}.
\]
On objects, it still sends a traceable endomorphism $f$ to its trace $\tr(f)$. On morphisms, it sends a morphism $(\varphi,\alpha)\colon (X,f) \to (Y,g)$ in $\cC^{\trl}$ to the morphism $\tr(\alpha)\colon \tr(f) \to \tr(g)$ in $\Omega \cC$ obtained by considering the total composite in the following diagram:
\begin{equation}
\label{eq:2FunctorialityTrace}
\begin{tikzcd}
	\unit && {X\otimes X^\vee} && {X\otimes X^\vee} \\
	\\
	{Y\otimes Y^\vee} && {X\otimes Y^\vee} && {Y\otimes X^\vee} && {X\otimes X^\vee} \\
	\\
	&& {Y\otimes Y^\vee} && {Y\otimes Y^\vee} && \unit.
	\arrow["{\coev_Y}"', from=1-1, to=3-1]
	\arrow["{\coev_X}", from=1-1, to=1-3]
	\arrow["{\varphi^r\otimes 1}"{description}, from=3-1, to=3-3]
	\arrow["{1\otimes (\varphi^r)^\vee}"{description}, from=1-3, to=3-3]
	\arrow[""{name=0, anchor=center, inner sep=0}, "{\varphi\otimes 1 }"{description}, from=3-3, to=5-3]
	\arrow[""{name=1, anchor=center, inner sep=0}, "{\id}"', curve={height=18pt}, from=3-1, to=5-3]
	\arrow["{\varphi^r\otimes 1}"{description}, from=3-5, to=3-7]
	\arrow["{\ev_X}", from=3-7, to=5-7]
	\arrow[""{name=2, anchor=center, inner sep=0}, "{\id}", curve={height=-18pt}, from=1-5, to=3-7]
	\arrow[""{name=3, anchor=center, inner sep=0}, "{\varphi\otimes 1}"{description}, from=1-5, to=3-5]
	\arrow[equal, shorten <=16pt, shorten >=16pt, from=1-3, to=3-1]
	\arrow[equal, shorten <=16pt, shorten >=16pt, from=5-5, to=3-7]
	\arrow["{f\otimes 1}", from=1-3, to=1-5]
	\arrow["{\ev_Y}"', from=5-5, to=5-7]
	\arrow["{1\otimes(\varphi^r)^\vee}"{description}, from=3-5, to=5-5]
	\arrow["g\otimes1"', from=5-3, to=5-5]
	\arrow["c\otimes1"{description}, shorten <=9pt, shorten >=9pt, Rightarrow, from=3-3, to=1]
	\arrow["u\otimes1"{description}, shorten <=9pt, shorten >=9pt, Rightarrow, from=2, to=3-5]
	\arrow["{\alpha\otimes(\varphi^r)^\vee}"{description}, shorten <=19pt, shorten >=19pt, Rightarrow, from=3, to=0]
\end{tikzcd}
\end{equation}
Here the morphism $(\varphi^r)^{\vee}\colon X^{\vee} \to Y^{\vee}$ is the dual of the morphism $\varphi^r\colon Y \to X$, making the top-left and bottom-right squares commute.

The precise construction of this functoriality was worked out in \cite[Section~2]{HSS}. In fact, they construct more generally for any $n \geq 1$ and any symmetric monoidal $(\infty,n)$-category $\cC$ an $(n-1)$-functor $\tr_{\cC}\colon \cC^{\trl} \to \Omega \cC$.\footnote{The $(\infty,n-1)$-category $\cC^{\trl}$ of traceable endomorphisms in $\cC$ is denoted $\End(\cC)$ in \cite{HSS}.}

The goal of this section is to give a detailed exposition of this $(n-1)$-functor, following \cite[Section~2]{HSS}. We will start in \Cref{subsec:Generalized_Traces} with a brief recollection of \textit{generalized} traces. In \Cref{subsec:FunctorialityTraces}, we will discuss the functoriality of traces in the $(\infty,1)$-categorical setting, which will be bootstrapped up to the $(\infty,n)$-categorical setting in \Cref{subsec:HigherCategoricalTraces}. In the bootstrapping process, we make use of the notion of an \textit{(op)lax natural transformation} between $n$-functors, which will be recalled in \Cref{subsec:LaxAndOplaxTransformations}.

\subsection{Generalized traces}
\label{subsec:Generalized_Traces}
Let $\cC$ be a symmetric monoidal $\infty$-category, and let $X \in \cC$ is a dualizable object, as defined in the introduction of this section. We will be interested in the \textit{generalized endomorphisms} of $X$, that is, morphisms in $\Cc$ of the form
\[
    Z\otimes X \oto{\: f\:} X\otimes Y,
\]
where $Y$ and $Z$ are objects of $\cC$.

\begin{defn}[Generalized trace]
\label{def:GeneralizedTrace}
Given a generalized endomorphism $f\colon Z \otimes X \to X \otimes Y$ in $\cC$, its \textit{generalized trace map}
    \[
        \gentr{f}{X} \colon Z \too Y \qin \cC
    \]
is defined as the following composition:
    \[
        Z \oto{\id\otimes \coev_X} 
        Z\otimes X \otimes X^\vee \oto{f\otimes \id} 
        X\otimes Y \otimes X^\vee \simeq
        X^\vee \otimes X \otimes Y \oto{\ev_X \otimes \id}
        Y.
    \]
If $Y$ and $Z$ are the monoidal unit $\unit$, the map $\tr(f) = \gentr{f}{X}\colon \unit \to \unit$ is called the \textit{trace} of the endomorphism $f\colon X \to X$. If $f$ is the identity on $X$, we will call its trace the \textit{dimension}\footnote{Also known as `Euler characteristic'.} of $X$:
    \[
        \dim(X) := \tr(\id_X) \qin \Map_{\cC}(\unit,\unit).
    \]
\end{defn}
 \begin{rmk}
    For simplicity, we will sometimes suppress the twist equivalence from the notation when writing down a (generalized) trace.
\end{rmk}
We now give several examples of dimensions and traces. For an extensive discussion of symmetric monoidal traces, including more examples, we refer the reader to \cite{PontoShulman}.

\begin{example}
\label{ex:Generalized_Trace_Unit}
    When $X=\one$, the trace of $f\colon Z \to Y$ is simply $f$ itself:
    \[
        \gentr{f}{\one} = f \qin \Map_{\cC}(Z,Y).
    \]
\end{example}

 \begin{example} 
 When $\cC = \cD(R)$ is the derived $\infty$-category of a commutative ring $R$, the notions of dimension and trace reduce to the notions of \textit{Euler characteristic} and \textit{Lefschetz number} of endomorphisms of perfect complexes. Similarly, when $\Cc = \Sp$ is the $\infty$-category of spectra and $X = \SS[A]$ is the suspension spectrum of a finite CW-complex $A$, one recovers the Euler characteristic of $A$ and the Lefschetz numbers of endomorphisms of $A$.
\end{example}
\begin{example}[Field trace]
    Consider a finite field extension $K \to L$. The generalized trace of the multiplication map $L \tensor_K L \to L$ is a $K$-linear map
    \[
        \tr_{L/K}\colon L \to K,
    \]
    known as the \textit{field trace} of $L$ over $K$. It sends an element $\alpha \in L$ to the trace of the endomorphism $m_{\alpha}\colon L \to L$ given by multiplication by $\alpha$.
    
    More generally, if $k$ is a commutative ring and $R$ is a dualizable $k$-algebra, applying the generalized trace construction to the multiplication map $R \tensor_k R \to R$ gives a $k$-linear trace map $\tr_{R/k}\colon R \to k$.
\end{example}

\begin{example}[Becker-Gottlieb transfer]
    Let $A$ be a compact space, so that its suspension spectrum $\SS[A] \in \Sp$ is dualizable. The generalized trace of the diagonal $\Delta\colon \SS[A] \to \SS[A \times A] \simeq \SS[A] \tensor \SS[A]$ is called the \textit{Becker-Gottlieb transfer of $A$}. We will study this map in more detail in \Cref{sec:BeckerGottlieb}. 
\end{example}

\begin{rmk}
\label{rmk:GeneralizedTraceFunctorial}
    The formation of generalized traces is functorial in $\cC$: if $F\colon \cC \to \cD$ is a symmetric monoidal functor, then $F$ preserves dualizable objects and it is easy to see that $\gentr{F(f)}{F(X)} \simeq F(\gentr{f}{X})$ for every map $f\colon Z \tensor X \to X \tensor Y$ in $\cC$.
\end{rmk}

\begin{rmk}
\label{Trace_Functorial}
The generalized trace is functorial in $Y$ and $Z$ in the following sense: if $a\colon Z' \to Z$ and $b\colon Y \to Y'$ are morphisms in $\cC$, then the generalized trace of the composite
\[
    Z' \otimes X \xrightarrow{a \otimes 1} Z \otimes X \xrightarrow{f} X \otimes Y \xrightarrow{1 \otimes b} X \otimes Y'
\]
is naturally equivalent to the composite
\[
    Z' \xrightarrow{a} Z \xrightarrow{\gentr{f}{X}} Y \xrightarrow{b} Y'.
\]
\end{rmk}

\begin{obs}[Trace is symmetric monoidal]
    Let $f\colon X \to X$ and $g\colon Y \to Y$ be two traceable endomorphisms in $\cC$, i.e.\ $X$ and $Y$ are dualizable objects of $\cC$. The tensor product $X \tensor Y$ is again dualizable, with evaluation and coevaluation given by the tensor product of those for $X$ and $Y$. A simple computation shows that the trace of the endomorphism $f \tensor g\colon X \tensor Y \to X \tensor Y$ is equivalent to the tensor product of the traces of $f$ and $g$:
    \[
    \tr(f\otimes g) \simeq \tr(f)\otimes \tr(g).
    \]
    In \Cref{cor:TraceSymmetricMonoidal}, we will see that these equivalences can be made fully coherent.
\end{obs}

\subsection{Functoriality of traces}
\label{subsec:FunctorialityTraces}
Let $\cC$ be a symmetric monoidal $\infty$-category. If $X$ and $X'$ are dualizable objects in $\cC$, then any equivalence $X \simeq X'$ in $\cC$ can be used to transfer the duality data from $X$ to duality data on $X'$, giving rise to a homotopy $\dim(X) \simeq \dim(X')$ in $\Map_{\cC}(\unit,\unit)$ between the dimensions of $X$ and $X'$. It is possible to choose these homotopies in a fully coherent fashion. For a precise formulation of this claim, we need the following definitions:

\begin{defn}
    Let $\Cc \in \CMon(\Cat_{\infty})$ be a symmetric monoidal $\infty$-category.
    \begin{enumerate}[(1)]
        \item We denote by $\Cc^{\rig} \subseteq \Cc$ the full subcategory spanned by the dualizable (a.k.a.\ `rigid') objects;
        \item We write $\Cc^{\dbl} \subseteq \Cc^{\simeq}$ for the full subgroupoid of dualizable objects and equivalences in $\Cc$, so that we have $\Cc^{\dbl} = (\Cc^{\rig})^{\simeq}$;
        \item We denote by $\Cc^{\trl} \subseteq \Fun(B\NN,\cC)^{\simeq}$ the full subgroupoid spanned by those endomorphisms $f\colon X \to X$ where $X$ is dualizable in $\cC$. We refer to such endomorphisms as the \textit{traceable} endomorphisms.
        \item We write $\Omega\cC := \Map_{\cC}(\unit,\unit)$ for the space of endomorphisms of the monoidal unit $\one$ of $\cC$.
    \end{enumerate}
\end{defn}

As we are about to make precise, the assignment $X \mapsto \dim(X)$ can be promoted to a map of spaces
\[
    \dim \colon \cC^\dbl \too \Omega \cC
    \qin \Spc,
\]
naturally in $\cC$. Similarly, one can promote the assignment $(X,f) \mapsto \gentr{f}{X}$ to a map of spaces
\[
    \tr \colon \cC^\trl \too \Omega \cC
    \qin \Spc,
\]
naturally in $\cC$. A convenient way of obtaining these two maps, as observed by \cite{TV2015caracteres}, is to use the fact that the functors $\cC \mapsto \cC^{\dbl}$ and $\cC \mapsto \cC^{\trl}$ from symmetric monoidal $\infty$-categories to spaces are corepresented by certain symmetric monoidal $\infty$-categories $\Fr^{\rig}(\pt)$ and $\Fr^{\rig}(B\NN)$. The main input for this is the following observation.

\begin{lemma}
The functor $(-)^{\rig}\colon \CAlg(\Cat_\infty) \too \Cat_{\infty}$ admits a left adjoint
\[
    \Fr^{\rig}\colon \Cat_{\infty} \too \CAlg(\Cat_\infty).
\]
\end{lemma}
\begin{proof}
We will show that the assignment $\Cc \mapsto \Cc^{\rig}$ preserves limits and filtered colimits, so that it admits a left adjoint by the adjoint functor theorem. Recall that both the forgetful functor $\CAlg(\Cat_{\infty}) \to \Cat_{\infty}$ and the groupoid core functor $(-)^{\simeq} \colon \Cat_{\infty} \to \Spc$ preserve limits and filtered colimits, and that fully faithful inclusions are closed under both limits and filtered colimits in $\Cat^{[1]}$. Since the inclusion $\cC^{\rig} \subseteq \cC$ is a full subcategory for every $\Cc$, we conclude that for every functor $I \to \Cat_{\infty}$ the canonical comparison maps $\lim_i \Cc^{\rig}_i \hookrightarrow \lim_i \Cc_i$ and, in case $I$ is filtered, $\colim_i \Cc_i^{\rig} \hookrightarrow \colim_i \Cc_i$ are fully faithful. All in all, we see that it remains to show that the functor $\CAlg(\Cat_\infty) \to \Spc$ sending $\Cc$ to $\Cc^{\dbl}$ preserves limits and filtered colimits.

By \cite[Lemma 4.6.1.10]{HA}, the $\infty$-groupoid $\cC^{\dbl}$ is equivalent to the $\infty$-category $\DDat(\cC)$ of duality data in $\cC$, cf.\ \cite[Notation 4.6.1.8]{HA}. But the assignment $\cC \mapsto \DDat(\cC)$ is easily seen to commute with limits and filtered colimits, finishing the proof.
\end{proof}

For an $\infty$-category $I$, the symmetric monoidal $\infty$-category $\Fr^{\rig}(I)$ comes equipped with a unit map $I \to \Fr^{\rig}(I)$ which lands in dualizable objects. Restriction along this map induces for every symmetric monoidal $\infty$-category $\cC$ an equivalence of spaces
\[
    \Map^\otimes(\Fr^\rig(I), \cC) \iso \Map(I,\cC^{\rig}).
\]
In particular, specializing to $I = \pt$ and $I = B\NN$ gives equivalences of spaces
\[
    \Map^\otimes(\Fr^\rig(\pt),\cC) \simeq \cC^\dbl \qquad \text{ and } \qquad \Map^\otimes(\Fr^\rig(B\NN),\cC) \simeq \cC^{\trl},
\]
showing that the functors $(-)^{\dbl}, (-)^{\trl}\colon \CAlg(\Cat_{\infty}) \to \Spc$ are corepresentable. 

In the case of $B\NN$, we will denote the unit map by $\gamma_{\univ}\colon B\NN \to \Fr^{\rig}(B\NN)$, and refer to it as the `universal traceable endomorphism'. Forming its trace produces an element $\tr(\gamma_{\univ}) \in \Omega \Fr^{\rig}(B\NN)$, thought of as the `universal trace'. By the Yoneda lemma, it determines for every $\cC \in \CAlg(\Cat_{\infty})$ a natural map of spaces $\tr\colon \cC^{\trl} \to \Omega \cC$. More explicitly:
\begin{defn}
\label{def:trace}
    Given a symmetric monoidal $\infty$-category $\cC$, we define the \textit{trace map} $\tr\colon \cC^{\trl} \to \Omega \cC$ as the composite
    \[
        \cC^\trl \simeq 
        \Map^\otimes(\Fr^\rig(B\NN),\cC) \oto{\:\:\Omega\:\:} 
        \Map(\Omega\Fr^\rig(B\NN),\Omega\cC) \oto{\ev_{\tr(\gamma_{\univ})}} 
        \Omega\cC.
    \]
\end{defn}
By \Cref{rmk:GeneralizedTraceFunctorial}, this functor agrees on objects with the formula for the trace given in the beginning of this section. 

The map $\dim\colon \cC^{\dbl} \to \Omega \cC$ may be produced in a similar way from the Yoneda lemma by using instead the `universal dimension' $\dim(\pt_{\univ}) \in \Omega \Fr^{\rig}(\pt)$, where $\pt_{\univ} \in \Fr^{\rig}(\pt)$ is the `universal dualizable object'. Alternatively, one can obtain the map $\dim\colon \cC^{\dbl} \to \Omega \cC$ from the map $\tr\colon \cC^{\trl} \to \Omega \cC$ by precomposing with the map $\cC^{\dbl} \to \cC^{\trl}\colon X \mapsto (X,\id)$.

\subsection{Lax and oplax transformations}
\label{subsec:LaxAndOplaxTransformations}
When $\cC \in \CMon(\Cat_{(\infty,n)})$ is a symmetric monoidal $(\infty,n)$-category, then the traceable endomorphisms $f\colon X \to X$ in $\cC$ naturally form an $(\infty,n-1)$-category $\cC^{\trl}$ rather than just a space. The construction of this $(\infty,n-1)$-category proceeds by realizing it as a subcategory of the $(\infty,n)$-category of functors $B\NN \to \cC$ and symmetric monoidal \textit{oplax} natural transformations between them, satisfying certain dualizability and adjointability constraints. In this subsection, we will recall the necessary background on lax and oplax transformations.

Following \cite{HSS}, we will model $(\infty, n)$-categories by Barwick's $n$-fold complete Segal spaces $\cC\colon (\Delta\catop)^{\times n} \to \Spc$ \cite[\S 14]{barwick2021unicity}. For an $(\infty,n)$-category $\cC$ and a vector $\vec{k} = (k_1,\dots,k_n)\in \NN^n$, we denote by $\cC_{\vec{k}}$ the value of $\cC$, considered as an $n$-fold complete Segal space, at $([k_1],\dots,[k_n])\in \Delta^n$.  
The collection of $(\infty,n)$-categories assembles into an $\infty$-category that we denote by $\Cat_{(\infty,n)}$. By a \textit{symmetric monoidal $(\infty,n)$-category} we mean a commutative monoid in the $\infty$-category $\Cat_{(\infty,n)}$.

For each vector $\vec{k}\in \NN^n$, the functor $\Cat_{(\infty,n)} \to \Spc$ given by $\cC\mapsto \cC_{\vec{k}}$ is corepresentable by the so called `walking $\vec{k}$-tuple' $\theta^{\vec{k}}\in \Cat_{(\infty,n)}$ (see \cite[Definition 5.1, Remark 5.4]{JFS2017}). Namely, we have a natural equivalence of spaces 
\[
\Map(\theta^{\vec{k}},\cC) \simeq \cC_{\vec{k}}.
\]
The $(\infty,n)$-categories $\theta^{\vec{k}}$ generate $\Cat_{(\infty,n)}$ under colimits. For $n = 1$, the $\infty$-category $\theta^{(k)}$ is just the $k$-simplex $[k] \in \Cat_{(\infty,1)}$. The following are some examples for $n = 2$:
\[
    \theta^{(0,0)} = \bullet, \qquad \qquad \qquad \qquad \theta^{(1,0)} = \begin{tikzcd} \bullet \rar & \bullet, \end{tikzcd} \qquad \qquad \theta^{(2,0)} = \begin{tikzcd} \bullet \rar & \bullet \rar & \bullet, \end{tikzcd}
\]
\[
    \theta^{(1,1)} =     \begin{tikzcd}
    	\bullet & \bullet,
    	\arrow[""{name=0, anchor=center, inner sep=0}, curve={height=-12pt}, from=1-1, to=1-2]
    	\arrow[""{name=1, anchor=center, inner sep=0}, curve={height=12pt}, from=1-1, to=1-2]
    	\arrow[shorten <=3pt, shorten >=3pt, Rightarrow, from=0, to=1]
    \end{tikzcd}
    \qquad \qquad
    \theta^{(1,2)} = 
    \begin{tikzcd}
    	\phantom{,} \bullet & \bullet,
    	\arrow[""{name=0, anchor=center, inner sep=0}, curve={height=-18pt}, from=1-1, to=1-2]
    	\arrow[""{name=1, anchor=center, inner sep=0}, from=1-1, to=1-2]
    	\arrow[""{name=2, anchor=center, inner sep=0}, curve={height=18pt}, from=1-1, to=1-2]
    	\arrow[shorten <=2pt, shorten >=2pt, Rightarrow, from=0, to=1]
    	\arrow[shorten <=2pt, shorten >=2pt, Rightarrow, from=1, to=2]
    \end{tikzcd}
    \qquad \qquad 
    \theta^{(2,1)} = 
\begin{tikzcd}
	\bullet & \bullet & \bullet.
	\arrow[""{name=0, anchor=center, inner sep=0}, curve={height=-12pt}, from=1-1, to=1-2]
	\arrow[""{name=1, anchor=center, inner sep=0}, curve={height=12pt}, from=1-1, to=1-2]
	\arrow[""{name=2, anchor=center, inner sep=0}, curve={height=-12pt}, from=1-2, to=1-3]
	\arrow[""{name=3, anchor=center, inner sep=0}, curve={height=12pt}, from=1-2, to=1-3]
	\arrow[shorten <=3pt, shorten >=3pt, Rightarrow, from=0, to=1]
	\arrow[shorten <=3pt, shorten >=3pt, Rightarrow, from=2, to=3]
\end{tikzcd}
\]

Given two $(\infty,n)$-categories $\cC$ and $\cD$, one can form two new $(\infty,n)$-categories
\[
    \Fun_{\lax}(\cC,\cD) \qquad \text{ and } \qquad 
    \Fun_{\oplax}(\cC,\cD),
\]
whose objects are the $n$-functors $\cC \to \cD$ and whose higher morphisms are given by certain \textit{(op)lax transformations} (see \cite[\S 2.2]{HSS}). We emphasize that there is no laxness in the objects of $\Fun_{\lax}(\cC,\cD)$ and $\Fun_{\oplax}(\cC,\cD)$, only in the higher morphisms. These two $(\infty,n)$-categories are most easily described in terms of a certain conjectural monoidal structure $\times^{\lax}$ on $\Cat_{(\infty,n)}$, called the \textit{Gray product}: for $(\infty,n)$-categories $\cC$ and $\cD$, we would like to define the $(\infty,n)$-categories $\Fun_{\lax}(\cC,\cD)$ and $\Fun_{\oplax}(\cC,\cD)$ via the  natural isomorphisms
\begin{align*}
\begin{split}
    \Map(\cE, \Fun_\lax(\cC,\cD)) &\simeq 
    \Map(\cE \times^{\lax} \cC, \cD) \\
    \Map(\cE, \Fun_\oplax(\cC,\cD)) &\simeq 
    \Map(\cC \times^{\lax} \cE, \cD)
\end{split}
\end{align*}
for $\cE \in \Cat_{(\infty,n)}$. Unfortunately, as far as we know the construction of the Gray product of $(\infty,n)$-categories with all of its expected properties is not yet fully furnished in the literature. 

To circumvent this technical difficulty, \cite{JFS2017} observed that in order to define $\Fun_\lax(-,-)$ and $\Fun_\oplax(-,-)$ one only needs the Gray products of the walking $\vec{k}$-tuples $\theta^{\vec{k}}$. In \cite[Definition 5.7]{JFS2017}, a combinatorial model is given for a Gray product $\theta^{\vec{k}}\times^\lax \theta^{\vec{\ell}}$, which we denote by $\theta^{\vec{k},\vec{\ell}}$. We have the following low-dimensional examples, cf.\ \cite[Example~4.2,4.3]{JFS2017}:
\[
\theta^{(1),(0)} = 
\begin{tikzcd}
	\bullet \rar & \bullet
\end{tikzcd},
\qquad \qquad
\theta^{(1),(1)} = 
\begin{tikzcd}
	\bullet & \bullet \\
	\bullet & \bullet
	\arrow[from=1-1, to=1-2]
	\arrow[from=2-1, to=2-2]
	\arrow[from=1-1, to=2-1]
	\arrow[from=1-2, to=2-2]
	\arrow[Rightarrow, shorten >=4, shorten <=4, from=2-1, to=1-2]
\end{tikzcd},
\qquad \qquad
\theta^{(1),(1,1)} = 
\begin{tikzcd}
	\bullet && \bullet \\
	\\
	\bullet && \bullet
	\arrow[""{name=0, anchor=center, inner sep=0}, curve={height=-12pt}, from=1-1, to=3-1]
	\arrow[""{name=1, anchor=center, inner sep=0}, curve={height=12pt}, from=1-1, to=3-1]
	\arrow[""{name=2, anchor=center, inner sep=0}, curve={height=-12pt}, from=1-3, to=3-3]
	\arrow[""{name=8, anchor=center, inner sep=0}, shift left=2, curve={height=-6pt}, shorten <=7pt, shorten >=13pt, Rightarrow, from=0, to=2]
	\arrow[""{name=3, anchor=center, inner sep=0}, curve={height=12pt}, from=1-3, to=3-3,crossing over]
	\arrow[""{name=9, anchor=center, inner sep=0}, shift right=2, curve={height=6pt}, shorten <=13pt, shorten >=7pt, Rightarrow, from=1, to=3,crossing over]
	\arrow[from=1-1, to=1-3]
	\arrow[from=3-1, to=3-3]
	\arrow[""{name=4, anchor=center, inner sep=0}, draw=none, from=3-1, to=1]
	\arrow[""{name=5, anchor=center, inner sep=0}, draw=none, from=1-3, to=2]
	\arrow[""{name=6, anchor=center, inner sep=0}, draw=none, from=3, to=3-3]
	\arrow[""{name=7, anchor=center, inner sep=0}, draw=none, from=1-1, to=0]
	\arrow[shorten <=4pt, shorten >=4pt, Rightarrow, from=5, to=6]
	\arrow[shorten <=2pt, shorten >=4pt, Rightarrow, from=7, to=4]
	\arrow[shorten <=5pt, shorten >=5pt, -, preaction={draw,double distance=3pt,>={Implies},->},from=9, to=8]
\end{tikzcd}.
\]
The construction of $\theta^{\vec{k},\vec{l}}$ is functorial in $\vec{k}, \vec{l} \in \Delta^{n}$, and satisfies $\theta^{\vec{k},\vec{0}} \simeq \theta^{\vec{k}} \simeq \theta^{\vec{0},\vec{k}}$. In particular, it comes equipped with a natural map $\theta^{\vec{k},\vec{l}} \to \theta^{\vec{k}} \times \theta^{\vec{l}}$ to the cartesian product. Using the $(\infty,n)$-categories $\theta^{\vec{k},\vec{l}}$, one now defines $\Fun_{\lax}(-,-)$ and $\Fun_{\oplax}(-,-)$ by the following two-step procedure:

\begin{defn}
\label{def:fun_lax_walking}
    For $\cC \in \Cat_{(\infty,n)}$ and $\vec{k}\in \NN^n$, we define the following $n$-fold simplicial spaces:
    \[
        \Fun_\lax(\theta^{\vec{k}},\cC)_{\vec{\ell}} := \Map(\theta^{\vec{k},\vec{\ell}},\cC) 
        \qquad
        \text{ and }
        \qquad
        \Fun_\oplax(\theta^{\vec{k}},\cC)_{\vec{\ell}} := \Map(\theta^{\vec{\ell},\vec{k}},\cC).
    \] 
    We then define $n$-fold simplicial spaces $\Fun_\lax(\cE,\cC)$ and $\Fun_{\oplax}(\cE,\cD)$ by:
    \begin{align*}
        \Fun_\lax(\cE,\cC)_{\vec{k}}&:= \Map(\cE,\Fun_\oplax(\theta^{\vec{k}},\cC)),\\
        \Fun_\oplax(\cE,\cC)_{\vec{k}}&:= \Map(\cE,\Fun_\lax(\theta^{\vec{k}},\cC)).
    \end{align*}
    These are $n$-fold complete Segal spaces by \cite[Corollary 5.19]{JFS2017} and these constructions assemble into functors 
    \[
        \Fun_\lax(-,-), \; \Fun_\oplax(-,-)\colon \Cat_{(\infty,n)}^\op \times \Cat_{(\infty,n)} \to \Cat_{(\infty,n)}.
    \]
    It is immediate from the construction that these functors preserve limits in both variables. Furthermore, they come equipped with a natural map from the ordinary functor category: by precomposition with the maps $\theta^{\vec{k},\vec{l}} \to \theta^{\vec{k}} \times \theta^{\vec{l}}$ one obtains maps
    \[
        \Fun_{\lax}(\cC,\cD) \longleftarrow \Fun(\cC,\cD) \longrightarrow \Fun_{\oplax}(\cC,\cD).
    \]
\end{defn}

\begin{example}
    The low-dimensional simplices in the $(\infty,n)$-category $\Fun_\lax([1], \cC)$ may be described as follows:
    \begin{enumerate}
        \item[(0)] Objects are morphisms $X_0 \to X_1$ in $\cC$;
        
        \item Morphisms are lax commuting squares
        \[\begin{tikzcd}
        	{X_0} & {X_1} \\
        	{Y_0} & {Y_1;}
        	\arrow[from=1-1, to=1-2]
        	\arrow[from=2-1, to=2-2]
        	\arrow[color={rgb,255:red,214;green,92;blue,92}, from=1-1, to=2-1]
        	\arrow[color={rgb,255:red,214;green,92;blue,92}, from=1-2, to=2-2]
        	\arrow[color={rgb,255:red,214;green,92;blue,92}, Rightarrow, shorten >=4, shorten <=4, from=2-1, to=1-2]
        \end{tikzcd}\]
            
        \item 2-Morphisms are $3$-dimensional diagrams of the following shape:
        \[
    \begin{tikzcd}
    	X_0 && X_1 \\
    	\\
    	Y_0 && Y_1.
    	\arrow[""{name=0, anchor=center, inner sep=0}, curve={height=-12pt}, from=1-1, to=3-1]
    	\arrow[""{name=1, anchor=center, inner sep=0}, curve={height=12pt}, from=1-1, to=3-1]
    	\arrow[""{name=2, anchor=center, inner sep=0}, curve={height=-12pt}, from=1-3, to=3-3]
    	\arrow[""{name=8, anchor=center, inner sep=0}, shift left=2, curve={height=-6pt}, shorten <=7pt, shorten >=13pt, Rightarrow, from=0, to=2]
    	\arrow[""{name=3, anchor=center, inner sep=0}, curve={height=12pt}, from=1-3, to=3-3,crossing over]
    	\arrow[""{name=9, anchor=center, inner sep=0}, shift right=2, curve={height=6pt}, shorten <=13pt, shorten >=7pt, Rightarrow, from=1, to=3,crossing over]
    	\arrow[from=1-1, to=1-3]
    	\arrow[from=3-1, to=3-3]
    	\arrow[""{name=4, anchor=center, inner sep=0}, draw=none, from=3-1, to=1]
    	\arrow[""{name=5, anchor=center, inner sep=0}, draw=none, from=1-3, to=2]
    	\arrow[""{name=6, anchor=center, inner sep=0}, draw=none, from=3, to=3-3]
    	\arrow[""{name=7, anchor=center, inner sep=0}, draw=none, from=1-1, to=0]
    	\arrow[draw={rgb,255:red,214;green,92;blue,92}, shorten <=4pt, shorten >=4pt, Rightarrow, from=5, to=6]
    	\arrow[draw={rgb,255:red,214;green,92;blue,92}, shorten <=2pt, shorten >=4pt, Rightarrow, from=7, to=4]
    	\arrow[draw={rgb,255:red,214;green,92;blue,92}, shorten <=5pt, shorten >=5pt, -, preaction={draw={rgb,255:red,214;green,92;blue,92},double distance=3pt,>={Implies},->},from=9, to=8]
    \end{tikzcd}
        \]
    \end{enumerate}
\end{example}

\begin{rem}
    In the case $n = 2$, a full definition of the Gray product is given by \cite{GHL2021Gray} in terms of scaled simplicial sets: by \cite[Corollary~2.15]{GHL2021Gray}, their construction determines a presentably (non-symmetric) monoidal structure
    \[
        -\times^\lax - \colon \Cat_{(\infty,2)} \times \Cat_{(\infty,2)} \too \Cat_{(\infty,2)}
    \]
    on the $\infty$-category $\Cat_{(\infty,2)}$. In \cite[Proposition~5.1.9]{HHLN2022}, it was computed that this lax product takes the expected values on pairs of simplices $[k] \in \Cat_{(\infty,1)}$ (as in the example $\theta^{(1),(1)}$ above). Consequently, for $\cC,\cD \in \Cat_{(\infty,1)}$ and $\cE \in \Cat_{(\infty,2)}$, we have
    \[
        \Map(\cC , \Fun_{\lax}(\cD,\cE)) \simeq 
        \Map(\cC \times^\lax \cD, \cE).
    \]
    In particular, the underlying $(\infty,1)$-category of $\Fun_\lax(\cD,\cE)$ can be described in terms of $\times^\lax$.
\end{rem}

Finally, we discuss the interaction of the constructions $\Fun_\lax(-,-)$ and $\Fun_{\oplax}(-,-)$ with symmetric monoidal structures. 

\begin{obs}
\label{obs:TrickSymmetricMonoidality}
    If $\cD$ is a symmetric monoidal $(\infty,n)$-category, i.e.\ a commutative monoid in $\Cat_{(\infty,n)}$, then $\Fun_{\lax}(\cC,\cD)$ and $\Fun_{\oplax}(\cC,\cD)$ inherit canonical structures of symmetric monoidal $(\infty,n)$-categories as well. Indeed, since the functors $\Fun_{\lax}(\cC,-)$ and $\Fun_{\oplax}(\cC,-)$ preserve limits, and therefore in particular finite products, they preserve commutative monoids in $\Cat_{(\infty,n)}$.
\end{obs}

If both $\cC$ and $\cD$ have symmetric monoidal structures, one can define the following variants of $\Fun_{\lax}(\cC,\cD)$ and $\Fun_{\oplax}(\cC,\cD)$ in which the functors and (op)lax natural transformations are symmetric monoidal:

\begin{defn}
    Assume $\cC$ and $\cD$ are symmetric monoidal $(\infty,n)$-categories. We define $(\infty,n)$-categories $\Fun_{\lax}^{\tensor}(\cC,\cD)$ and $\Fun_{\oplax}^{\tensor}(\cC,\cD)$ by the following representability conditions: for $\cE \in \Cat_{(\infty,n)}$ we have natural equivalences
    \begin{align*}
          \Map(\cE, \Fun_\lax^{\tensor}(\cC,\cD)) &\simeq 
        \Map^{\tensor}(\cC, \Fun_{\oplax}(\cE,\cD)) \\
        \Map(\cE, \Fun_\oplax^{\tensor}(\cC,\cD)) &\simeq 
        \Map^{\tensor}(\cC, \Fun_{\lax}(\cE,\cD)).
    \end{align*}
\end{defn}

\subsection{Higher categorical traces}
\label{subsec:HigherCategoricalTraces}
In this section, we will recall from \cite{HSS} the definition of the symmetric monoidal $(n-1)$-functor $\tr\colon \cC^{\trl} \to \Omega \cC$ for a symmetric monoidal $(\infty,n)$-category $\cC$. We will start by formally defining its source and target.

\begin{defn}\label{def:Dbl_Trl}
    Let $\cC \in \CAlg(\Cat_{(\infty,n)})$ be a symmetric monoidal $(\infty,n)$-category. We define the \textit{$(\infty,n-1)$-category $\cC^{\dbl}$ of dualizable objects in $\cC$} as
    \[
        \cC^{\dbl} := \iota_{n-1}\Fun^{\tensor}_{\oplax}(\Fr^{\rig}(\pt),\cC).
    \]
    Similarly, we define the \textit{$(\infty,n-1)$-category $\cC^{\trl}$ of traceable endomorphisms in $\cC$} as
    \[
        \cC^{\trl} := \iota_{n-1}\Fun^{\tensor}_{\oplax}(\Fr^{\rig}(B\NN),\cC).
    \]
    Precomposition with the symmetric monoidal functor $\Fr^\rig(B\NN) \to \Fr^\rig(\pt)$ induced by the map $B\NN \to \pt$ gives an $(n-1)$-functor $\cC^{\dbl} \to \cC^{\trl}$.
\end{defn}

\begin{rem}
    The application of the underlying $(\infty,n-1)$-category functor $\iota_{n-1}$ in the definition of $\cC^\dbl$ and $\cC^\trl$ is actually redundant as the $(\infty,n)$-categories to which it is applied are already $(\infty,n-1)$-categories, cf.\ \cite[Remark~2.3]{HSS}.
\end{rem}

For $n = 1$, there is an equivalence of $\infty$-groupoids $\iota_0 \Fun^{\tensor}_{\oplax}(-,-) \simeq \Map^{\tensor}(-,-)$, and thus by corepresentability the definitions of the $\infty$-groupoids $\cC^{\dbl}$ and $\cC^{\trl}$ agree with the ones from the previous subsection. For larger $n$, the constructions of $\cC^{\dbl}$ and $\cC^{\trl}$ are compatible with the inclusions $\Cat_{(\infty,n)} \hookrightarrow \Cat_{(\infty,n+1)}$, in the sense that the following squares commute:
\[\begin{tikzcd}
	{\CAlg(\Cat_{(\infty,n)})} && {\Cat_{(\infty,n-1)}} \\
	{\CAlg(\Cat_{(\infty,n+1)})} && {\Cat_{(\infty,n)},}
	\arrow["{(-)^{\dbl}}", from=1-1, to=1-3]
	\arrow[hook, from=1-1, to=2-1]
	\arrow["{(-)^{\dbl}}", from=2-1, to=2-3]
	\arrow[hook, from=1-3, to=2-3]
\end{tikzcd}
\qquad \qquad
\begin{tikzcd}
	{\CAlg(\Cat_{(\infty,n)})} && {\Cat_{(\infty,n-1)}} \\
	{\CAlg(\Cat_{(\infty,n+1)})} && {\Cat_{(\infty,n)}.}
	\arrow["{(-)^{\trl}}", from=1-1, to=1-3]
	\arrow[hook, from=1-1, to=2-1]
	\arrow["{(-)^{\trl}}", from=2-1, to=2-3]
	\arrow[hook, from=1-3, to=2-3]
\end{tikzcd}
\]
Moreover, these squares are vertically right adjointable: for $1 \leq k \leq n$ there are equivalences
\[
    (\iota_k \cC)^{\dbl} \simeq \iota_{k-1}\cC^{\dbl} \qquad \qquad \text{ and } \qquad \qquad (\iota_k \cC)^{\trl} \simeq \iota_{k-1}\cC^{\trl}.
\]

It follows that the space of objects of the $(\infty,n-1)$-category $\cC^{\dbl}$ is the space of dualizable objects in $\cC$, while the space of objects of $\cC^{\trl}$ is the space of traceable endomorphisms in $\cC$.

We will next show that also the morphisms in $\cC^{\dbl}$ and $\cC^{\trl}$ are as claimed in the introduction of this section.
\begin{lem}
\label{lem:DblFun}
    Let $\cC$ be a symmetric monoidal $(\infty,n)$-category. For any $(\infty,n-1)$-category $\cE$ there are equivalences of spaces
    \begin{align*}
        \Map(\cE,\cC^{\dbl}) &\simeq (\iota_1\Fun_{\lax}(\cE,\cC))^{\dbl},\\
        \Map(\cE,\cC^{\trl}) &\simeq (\iota_1\Fun_{\lax}(\cE,\cC))^{\trl}.
    \end{align*}
\end{lem}
\begin{proof}
    For the second equation, there are equivalences
    \begin{align*}
        \Map(\cE,\cC^{\trl}) &\simeq \Map^{\tensor}(\Fr^{\rig}(B\NN),\Fun_{\lax}(\cE,\cC)) \\
        &\simeq \Map^{\tensor}(\Fr^{\rig}(B\NN),\iota_1\Fun_{\lax}(\cE,\cC))\\
        &\simeq  (\iota_1\Fun_{\lax}(\cE,\cC))^{\trl},
    \end{align*}
    where the first equivalence is immediate from the definitions, the second equivalence holds as $\Fr^{\rig}(B\NN)$ is an $(\infty,1)$-category, and the third is the defining property of $\Fr^{\rig}(B\NN)$. The computation for the first equation is analogous, using $\Fr^{\rig}(\pt)$ instead.
\end{proof}

\begin{cor}[cf.\ {\cite[Lemma~2.4]{HSS}}]
\label{cor:Morphisms_In_Dbl}
Let $\cC$ be an $(\infty,n)$-category. There are natural monomorphisms of $(\infty,n)$-categories
\[
    \cC^{\dbl} \hookrightarrow \cC \qquad \text{ and } \qquad \cC^{\trl} \hookrightarrow \Fun_{\oplax}(B\NN,\cC).
\]
An object $X \in \cC$ is in $\cC^{\dbl}$ if and only if it is dualizable, and a morphism $f\colon X \to Y$ between dualizable objects is in $\cC^{\dbl}$ if and only if it is a left adjoint morphism. Similarly, an endomorphism $f\colon X \to X$ is an object in $\cC^{\trl}$ if and only if $X$ is dualizable, and a lax square
\[\begin{tikzcd}
    X \dar[swap]{f} \rar{\varphi} & Y \dar{g} \\
    X \rar[swap]{\varphi} \urar[Rightarrow, shorten <=4pt, shorten >=4pt]{\alpha} & Y
\end{tikzcd}\]
is a morphism in $\cC^{\trl}$ if and only if $X$ and $Y$ are dualizable and $\varphi\colon X \to Y$ is a left adjoint in $\cC$.
\end{cor}
\begin{proof}
We start by producing the two monomorphisms. We will do this for $\cC^{\trl}$, and leave the case for $\cC^{\dbl}$ to the reader. For every $(\infty,n-1)$-category $\cE$, we consider the following natural composite map of spaces:
\[
    \Map(\cE,\cC^{\trl}) \overset{\ref{lem:DblFun}}{\simeq} (\iota_1 \Fun_{\lax}(\cE,\cC))^{\trl} \hookrightarrow \Map(B\NN,\Fun_{\lax}(\cE,\cC)) \simeq \Map(\cE,\Fun_{\oplax}(B\NN,\cC)).
\]
Note that the middle map is an inclusion of path components by definition of $(-)^{\trl}$. By the Yoneda lemma, we obtain the desired monomorphism $\cC^{\trl} \hookrightarrow \Fun_{\oplax}(B\NN,\cC)$.

We now prove the descriptions of the objects and morphisms of $\cC^{\dbl}$ and $\cC^{\trl}$. As mentioned earlier, the statement on objects is clear from the equivalences $\iota_0\cC^{\dbl} \simeq (\iota_1\cC)^{\dbl}$ and $\iota_0\cC^{\trl} \simeq (\iota_1\cC)^{\trl}$. For the statement on morphisms, we apply \Cref{lem:DblFun} to $\cE = [1]$ to reduce to a statement about the dualizable objects of $\Fun_{\lax}([1],\cC)$. This then becomes an instance of Lemma 2.4 in \cite{HSS}, which says that an object $(\varphi\colon X \to Y)$ of $\Fun_{\lax}([1],\cC)$ is dualizable if and only if $X$ and $Y$ are dualizable and $\varphi\colon X \to Y$ is a left adjoint in $\cC$.
\end{proof}

The target $\Omega \cC$ of the higher trace functor is the $(\infty,n-1)$-category of endomorphisms of the monoidal unit, defined formally as follows:

\begin{defn}
\label{def:Omega}
The $(\infty,n-1)$-category $\Omega \cC$ is defined as the following pullback\footnote{Here we think of the $(\infty,n)$-category $\cC$ as a complete Segal $(\infty,n-1)$-category $\cC\colon \Delta\catop \to \Cat_{(\infty,n-1)}$ whose $(\infty,n-1)$-category of objects $\cC_0$ is an $\infty$-groupoid.} in $\Cat_{(\infty,n-1)}$:
\[\begin{tikzcd}
    \Omega \cC \dar \rar \drar[pullback] & \cC_1 \dar{(d_1,d_0)} \\
    \{(\unit,\unit)\} \rar[hookrightarrow] & \cC_0 \times \cC_0.
\end{tikzcd}\]
\end{defn}

That is, the objects of $\Omega \cC$ are given by 1-morphisms $f\colon \unit \to \unit$ in $\cC$, the morphisms are given by 2-morphisms 
\[\begin{tikzcd} 
\unit && \unit \arrow[""{name=0, anchor=center, inner sep=0}, "f", curve={height=-12pt}, from=1-1, to=1-3]
    	\arrow[""{name=1, anchor=center, inner sep=0}, "g"', curve={height=12pt}, from=1-1, to=1-3]
    	\arrow["\alpha", shorten <=3pt, shorten >=3pt, Rightarrow, from=0, to=1]
\end{tikzcd},\]
and so forth. If $\cC$ is symmetric monoidal, then so is $\Omega\cC$, as all the functors in the diagram of which $\Omega \cC$ is defined to be the limit are manifestly symmetric monoidal. 

\begin{prop}\label{Omega_Fun_Lax}
    For every $(\infty,n-1)$-category $\cE$ the canonical map 
    \[
        \Fun_{\lax}(\cE,\Omega \cC) \to \Omega \Fun_{\lax}(\cE,\cC)
    \]
    is an equivalence. In particular, there is an equivalence
    \[
        \Omega (\iota_1 \Fun_{\lax}(\cE,\cC)) \simeq \Map(\cE,\Omega \cC).
    \]
\end{prop}
\begin{proof}
    In the special case $\cE = \theta^{\vec{k}}$, the first statement is \cite[Proposition~2.6]{HSS}. The general case follows from the fact that the $(\infty,n-1)$-categories $\theta^{\vec{k}}$ generate $\Cat_{(\infty,n-1)}$ under colimits and both sides take colimits in the variable $\cE$ to limits.
    
    For the second statement, we observe that $\Omega \iota_{k} \cC \simeq \iota_{k-1} \Omega \cC$ for $1 \leq k \leq n-1$, so that 
    \[
        \Omega \iota_1 \Fun_{\lax}(\cE,\cC) \simeq \iota_0 \Omega \Fun_{\lax}(\cE,\cC) \simeq \iota_0 \Fun_{\lax}(\cE,\Omega \cC) \simeq \Map(\cE,\Omega \cC),
    \]
    proving the claim.
\end{proof}

We can now bootstrap the $(\infty,1)$-categorical trace map from \Cref{def:trace} to the  $(\infty,n)$-categorical setting. 
\begin{defn}
\label{def:Higher_Categorical_Trace_Functor}
    Let $\cC$ be a symmetric monoidal $(\infty,n)$-category.  We define the \textit{trace functor} 
    \[
    \tr\colon \cC^{\trl} \to \Omega \cC
    \]
    as the $(n-1)$-functor inducing the following natural map of spaces for every $\cE \in \Cat_{(\infty,n-1)}$: 
    \[
        \Map(\cE,\cC^{\trl}) \simeq (\iota_1\Fun_{\lax}(\cE,\cC))^{\trl} \xrightarrow{\tr} \Omega(\iota_1 \Fun_{\lax}(\cE,\cC)) \simeq \Map(\cE,\Omega \cC).
    \]
    Here the first and last equivalences are \Cref{lem:DblFun} and \Cref{Omega_Fun_Lax} respectively and the middle map is the trace map of spaces for the symmetric monoidal $(\infty,1)$-category $\iota_1 \Fun_{\lax}(\cE,\cC)$, as defined in \Cref{def:trace}.

    We define the \textit{dimension functor} $\dim\colon \cC^{\dbl} \to \Omega \cC$ as the $(n-1)$-functor obtained by precomposing $\tr\colon \cC^{\trl} \to \Omega \cC$ with the map $\cC^{\dbl} \to \cC^{\trl}$ from \Cref{def:Dbl_Trl}.
\end{defn}

By definition, the $(n-1)$-functor $\tr\colon \cC^{\trl} \to \Omega \cC$ is on groupoid cores simply given by the trace map of the $(\infty,1)$-category $\iota_1 \cC$, and thus it has the correct behavior on objects. The following lemma shows that also its behavior on morphisms is as described in the introduction of this section:

\begin{lem}[cf.\ {\cite[Lemma~2.4]{HSS}}]
    Let $f\colon X \to X$ and $g\colon Y \to Y$ be traceable endomorphisms in $\cC$, and let $(\varphi,\alpha)\colon (X,f) \to (Y,g)$ be a morphism in $\cC^{\trl}$ as in \eqref{eq:TraceableMorphism}, cf.\ \Cref{cor:Morphisms_In_Dbl}. Then applying $\tr\colon \cC^{\trl} \to \Omega \cC$ gives the morphism $\tr(\alpha)\colon \tr(f) \to \tr(g)$ in $\Omega \cC$ given by \eqref{eq:2FunctorialityTrace}.
\end{lem}
\begin{proof}
    Under the identification $\Map([1],\cC^{\trl}) \simeq (\iota_1\Fun_{\lax}([1],\cC))^{\trl}$, the morphism $(\varphi,\alpha)\colon (X,f) \to (Y,g)$ in $\cC^{\trl}$ corresponds to an endomorphism $(f,g,\alpha)$ of the dualizable object $(\varphi\colon X \to Y)$ in the symmetric monoidal $\infty$-category $\Fun_{\lax}([1],\cC)$, and the map $\tr(\alpha)$ is by definition the trace of this endomorphism:
    \[
        \id_{\unit} \xrightarrow{\coev_{\varphi}} \varphi \tensor \varphi^{\vee} \xrightarrow{\alpha \tensor \id} \varphi \tensor \varphi^\vee \xrightarrow{\ev_{\varphi}} \id_{\unit}.
    \]
    In the proof of \cite[Lemma~2.4]{HSS}, the authors write down explicit duality data of $\varphi$ as an object of $\Fun_{\lax}([1],\cC)$, and plugging this in gives the explicit description of the map $\tr(\alpha)\colon \tr(f) \to \tr(g)$ as given in \eqref{eq:2FunctorialityTrace}.
\end{proof}

We finish the section by showing that the trace functor $\tr\colon \cC^{\trl} \to \Omega \cC$ admits a canonical enhancement to a symmetric monoidal $(n-1)$-functor.

\begin{lem}
\label{lem:OmegaLimitPreserving}
    The functors
    \begin{align*}
        (-)^{\trl}\colon \CAlg(\Cat_{(\infty,n)}) &\to \Cat_{(\infty,n-1)} \\
        \Omega\colon \CAlg(\Cat_{(\infty,n)}) &\to \Cat_{(\infty,n-1)}
    \end{align*}
    preserve limits.
\end{lem}
\begin{proof}
    The functor $\Omega(-)$ is the pullback of the limit-preserving functors sending $\cC$ to $\{(\one,\one)\}$, $\cC_0 \times \cC_0$ and $\cC_1$, and thus also $\Omega$ preserves limits. For $(-)^{\trl}$, this follows from the fact that for every $\cE \in \Cat_{(\infty,n-1)}$ and $\cC \in \CAlg(\Cat_{(\infty,n)})$ there is an equivalence
    \[
        \Map(\cE,\cC^{\trl}) \simeq \Map^{\tensor}(\Fr^{\rig}(B\NN),\Fun_{\lax}(\cE,\cC)),
    \]
    where the right-hand side preserves limits in the variable $\cC \in \CAlg(\Cat_{(\infty,n)})$.
\end{proof}

Since the functors $(-)^{\trl}$ and $\Omega(-)$ in particular preserve finite products, they induce functors on commutative algebra objects. Employing the equivalence $\CAlg(\Cat_{(\infty,n)}) \simeq \CAlg(\CAlg(\Cat_{(\infty,n)}))$, we may regard every symmetric monoidal $(\infty,n)$-category as a commutative algebra in the $\infty$-category $\CAlg(\Cat_{(\infty,n)})$. This leads to a symmetric monoidal enhancement of the trace functor.

\begin{cor}[{cf.\ \cite[Definition~2.11]{HSS}}]
\label{cor:TraceSymmetricMonoidal}
    For every symmetric monoidal $(\infty,n)$-category $\cC$, the $(\infty,n-1)$-categories $\cC^{\trl}$ and $\Omega \cC$ admit canonical enhancements to symmetric monoidal $(n-1)$-categories, and the trace functor $\tr\colon \cC^{\trl} \to \Omega \cC$ admits a canonical enhancement to a symmetric monoidal $(n-1)$-functor.
\end{cor}

\section{Traces in the \texorpdfstring{$(\infty,2)$}{2}-category of spans}
\label{sec:tracesInCorr}

Consider an $\infty$-category $\cD$ which admits finite limits. In this section, we will compute dimensions and traces in a certain $(\infty,2)$-category\footnote{The notation $\Span(\cD)$ is usually used for the $(\infty,1)$-category of spans. Since we will only ever use its $(\infty,2)$-categorical enhancement, we give it the same name to avoid cluttering of notation.} $\Span(\cD)$ of \textit{spans} (or \textit{correspondences}) in $\cD$, to be recalled below. The objects of $\Span(\cD)$ are given by the objects the objects of $\cD$, and the $\infty$-category of morphisms from $X$ to $Y$ is equivalent to the $\infty$-category $\cD_{/(X \times Y)}$ of \textit{spans} $X \leftarrow W \rightarrow Y$. Composition is informally given by forming the pullback:
\[
    (Y \leftarrow V \rightarrow Z) \circ (X \leftarrow W \rightarrow Y) = (X \leftarrow W \times_Y V \rightarrow Z).
\]
The $(\infty,2)$-category $\Span(\cD)$ inherits a symmetric monoidal structure from the cartesian product on $\cD$, and applying the constructions of \Cref{subsec:HigherCategoricalTraces} thus gives a trace functor
\[
    \tr_{\Span(\cD)}\colon \Span(\cD)^\trl \to \Omega \Span(\cD) \simeq \cD.
\]
The primary reason we are interested in traces in $\Span(\cD)$ is the fact that, as a symmetric monoidal $(\infty,2)$-category, $\Span(\cD)$ admits a universal property: it is the target of the universal symmetric monoidal \emph{bivariant theory} $\cD\to\Span(\cD)$, in the sense of \Cref{defn:symmonbiv}. In particular, any other symmetric monoidal bivariant theory $F\colon \cD \to \cE$ factors uniquely through a symmetric monoidal 2-functor $\Span(\cD) \to \cE$, which allows us to reduce computations of traces in $\cE$ to computations of traces in $\Span(\cD)$.

A direct computation of traces in $\Span(\cD)$ can be obtained from the description of composition in $\Span(\cD)$ given above. Observe first that every object in $\Span(\cD)$ is dualizable: the evaluation and coevaluation maps for an object $X \in \cD$ are given by the spans
\[
    X \times X \xleftarrow{\;\Delta\;} X \longrightarrow \pt \qquad\qquad \text{ and } \qquad\qquad \pt \longleftarrow X \xrightarrow{\;\Delta\;} X \times X
\]
where $\Delta$ denotes the diagonal of $X$. It follows that the dimension of $X$ in $\Span(\cD)$ is equivalent to the free loop space $LX$, defined as the cotensoring $X^{S^1}$ of $X$ by the circle $S^1$. Furthermore, if we consider an endomorphism of $X$ in $\Span(\cD)$ given by a span $X \xleftarrow{f} Z \xrightarrow{g} X$, then the trace of this endomorphism in $\Span(\cD)$ is the \textit{equalizer} $\Eq(f,g)$ of $f$ and $g$ in $\cD$, see \Cref{def:Equalizer} and \Cref{lem:TracesInCorr}. 

The goal of this section is to give a careful account of (the functoriality of) these computations. In \Cref{subsec:DefinitionCorr}, we will recall the definition of the $(\infty,2)$-category $\Span(\cD)$, its universal property, and some of its features. In \Cref{subsec:DimensionsInCorr}, we will show that the trace of the span $X \xleftarrow{f} Z \xrightarrow{g} X$, viewed as an endomorphism in $\Span(\cD)$, is the equalizer $\Eq(f,g)$ of $f$ and $g$, and in particular that there is an equivalence $\dim_{\Span(\cD)}(X) \simeq LX$. Moreover, we will make the latter identification natural in $X \in \cD$.

\subsection{The \texorpdfstring{$(\infty,2)$}{2}-category of spans}
\label{subsec:DefinitionCorr}
We will start by introducing the $(\infty,2)$-category $\Span(\cD)$ of spans associated with an $\infty$-category $\cD$ which admits finite limits. There are various ways to define $\Span(\cD)$: explicitly as a 2-fold complete Segal space \cite{HaugsengSpans, GR} generalizing Barwick's $(\infty,1)$-categorical construction \cite{barwick2013QConstruction,barwick2017spectral}, via bivariant fibrations \cite{stefanich2020higher}, or via bivariant theories \cite{Macpherson}. All these $(\infty,2)$-categories are equivalent because they satisfy a universal property: they come equipped with a functor $\cD \to \Span(\cD)$ which is universal among bivariant theories from $\cD$ into some $(\infty,2)$-category $\cE$, cf.\ \Cref{def:bivariantTheory}. For the purposes of this article, it is convenient to simply \textit{define} $\Span(\cD)$ via its universal property.

\begin{defn}
An $\infty$-category is called \textit{left exact} if it has finite limits, and a functor is called left exact if it preserves finite limits. We let $\Cat^{\lex}_{\infty}$ denote the $\infty$-category of (small) left exact $\infty$-categories and left exact functors.
\end{defn}

Recall that a commutative square
\[
\begin{tikzcd}
    X' \dar[swap]{b} \rar{a} & X \dar{c} \\
    Y' \rar{d} & Y
\end{tikzcd}
\]
in an $(\infty,2)$-category $\cE$ is called \textit{vertically right adjointable} if the morphisms $b$ and $c$ admit right adjoints $b^r$ and $c^r$ in $\cE$, and the Beck-Chevalley map
\[
    ab^r \xrightarrow{\eta_c} c^rcab^r \simeq c^rdbb^r \xrightarrow{\epsilon_b} c^rd
\]
is an equivalence in $\cE$.

\begin{defn}[{\cite[Section~3.2]{Macpherson}}]
\label{def:bivariantTheory}
Let $\cD$ be a left exact $\infty$-category and let $\cE$ be an $(\infty,2)$-category. A functor $F\colon \cD \to \cE$ is called a \textit{bivariant theory} if the following two conditions are satisfied:
\begin{enumerate}[(1)]
    \item For every morphism $f\colon A \to B$ in $\cD$, the morphism $F(f)\colon F(A) \to F(B)$ admits a right adjoint $F(f)^r$ in $\cE$;
    \item For every pullback square
    \[
    \begin{tikzcd}
        A' \dar[swap]{f'} \rar{g} \drar[pullback] & A \dar{f} \\
        B' \rar{h} & B
    \end{tikzcd}
    \]
    in $\cD$, the induced square
    \[
    \begin{tikzcd}
        F(B) \dar[swap]{h^*} \rar{f^*} & F(A) \dar{g^*} \\
        F(B') \rar{{f'}^*} & F(A')
    \end{tikzcd}
    \]
    in $\cE$ is verticall right adjointable.
\end{enumerate}

Given another bivariant theory $F'\colon \cC \to \cE$, a transformation $\alpha\colon F \Rightarrow F'$ is called \textit{bivariant} if it commutes with the right adjoints of $F(f)$:
for every morphism $f\colon A \to B$ in $\cD$, the resulting naturality square
\[
\begin{tikzcd}
    F(A) \rar{\alpha_A} \dar[swap]{F(f)} & F'(A) \dar{F'(f)} \\
    F(B) \rar{\alpha_B} & F'(B)
\end{tikzcd}
\]
is vertically right adjointable.

We let $\Biv(\cD, \cE) \subseteq \Fun(\cD,\cE)$ denote the $(\infty,2)$-category of bivariant theories $F\colon \cD \to \cE$, bivariant transformations and arbitrary 3-transformations.
\end{defn}

We may now introduce $\Span(\cD)$ as the $(\infty,2)$-category equipped with the universal bivariant theory $h_{\cD}\colon \cD \to \Span(\cD)$.

\begin{prop}[{\cite[Theorem~4.2.6]{Macpherson}}] 
\label{prop:UniversalPropertyCorr}
Let $\cD$ be a left exact $\infty$-category. There exists an $(\infty,2)$-category $\Span(\cD)$ equipped with a bivariant theory $h_{\cD}\colon \cD \to \Span(\cD)$ satisfying the following property: for any other $(\infty,2)$-category $\cE$, composition with $h_{\cD}$ induces an equivalence of $(\infty,2)$-categories
\begin{align*}
    \Fun(\Span(\cD),\cE) \iso \Biv(\cD,\cE).
\end{align*}
The construction $\cD \mapsto \Span(\cD)$ and the maps $h_{\cD}\colon \cD \to \Span(\cD)$ uniquely assemble into a functor $\Span\colon \Cat^{\lex}_{\infty} \to \Cat_{(\infty,2)}$ equipped with a natural transformation $h$ from the forgetful functor $\Cat^{\lex}_{\infty} \to \Cat_{\infty} \hookrightarrow \Cat_{(\infty,2)}$.
\end{prop}

\begin{rem}
    MacPherson used the notation $\Corr(\cD)$ instead of $\Span(\cD)$ and referred to it as the $(\infty,2)$-category of \textit{correspondences} in $\cD$.
\end{rem}

The above universal property of $\Span(\cD)$ admits a symmetric monoidal variant:

\begin{defn}\label{defn:symmonbiv}
Let $\cD$ be a left exact $\infty$-category, equipped with the cartesian monoidal structure, and let $\cE^{\tensor}$ be a symmetric monoidal $(\infty,2)$-category. A symmetric monoidal functor $F\colon \cD^{\times} \to \cE^{\tensor}$ is called a \textit{symmetric monoidal bivariant theory} if its underlying functor $\cD \to \cE$ is a bivariant theory.

We let $\Biv^{\tensor}(\cD, \cE) \subseteq \Fun^{\tensor}(\cD,\cE)$ denote the $(\infty,2)$-category of symmetric monoidal bivariant theories $F\colon \cD \to \cE$, symmetric monoidal bivariant transformations and arbitrary 3-transformations.
\end{defn}

\begin{prop}[{\cite[4.4.6 Theorem]{Macpherson}}]
Let $\cD$ be a left exact $\infty$-category. There is a unique enhancement of $\Span(\cD)$ to a symmetric monoidal $(\infty,2)$-category together with an enhancement of the bivariant theory $h_{\cD}\colon \cD \to \Span(\cD)$ to a symmetric monoidal bivariant theory. It satisfies the following property: for any other symmetric monoidal $(\infty,2)$-category $\cE$, composition with $h_{\cD}$ induces an equivalence of $(\infty,2)$-categories
\begin{align*}
    \Fun^{\tensor}(\Span(\cD),\cE) \iso \Biv^{\tensor}(\cD,\cE).
\end{align*}
The construction $\cD \mapsto \Span(\cD)$ and the maps $h_{\cD}\colon \cD \to \Span(\cD)$ uniquely assemble into a functor $\Span\colon \Cat^{\lex}_{\infty} \to \CAlg(\Cat_{(\infty,2)})$ equipped with a natural transformation $h$ from the forgetful functor $\Cat^{\lex}_{\infty} \to \CAlg(\Cat_{\infty}) \hookrightarrow \CAlg(\Cat_{(\infty,2)})$.
\end{prop}

In \cite[Section~3.1]{stefanich2020higher}, Stefanich introduces for every left exact $\infty$-category $\cD$ an explicit 2-fold complete Segal space $2\Corr(\cD)$ equipped with a functor $\cD \to 2\Corr(\cD)$ satisfying the universal property of $h_{\cD}\colon \cD \to \Span(\cD)$, cf.\ \cite[Theorem 3.4.18]{stefanich2020higher}. In \cite[Remark 3.1.10]{stefanich2020higher}, the mapping spaces in $2\Corr(\cD)$ are computed to be
\[
    \Map_{2\Corr(\cD)}(d,d') \simeq \cD_{/d \times d'}
\]
for $d,d'\in \cD$. One observes from his computation (cf.\ \cite[Notation 3.1.4, Notation 3.1.9, Remark 3.1.8]{stefanich2020higher}) that this identification is natural in the triple $(\cD,d,d')$. In particular, letting both $d$ and $d'$ be the final object of $\cD$, we obtain the following result:

\begin{prop}[{\cite{stefanich2020higher}}]
\label{prop:OmegaCorr}
Let $\cD$ be a left exact $\infty$-category. Then there is a functorial equivalence
\[
    i_{\cD}\colon \cD \iso \Omega \Span(\cD).
\]
\end{prop}

We finish this subsection by recalling from \cite{HaugsengSpans,stefanich2020higher} that all objects in $\Span(\cD)$ are dualizable and that every morphism $f$ in $\cD$ gives rise to an adjunction in $\Span(\cD)$ between the associated left- and right-pointing morphisms.

\begin{lem}[{\cite[Theorem~1.4]{HaugsengSpans}, \cite[Proposition~3.3.3]{stefanich2020higher}}]
\label{lem:DblCorr}
For a left exact $\infty$-category $\cD$ and an object $X \in \cD$, the span
\[
    X \times X \xleftarrow{\;\Delta\;} X \xrightarrow{\;r\;} \pt
\]
is part of a duality datum exhibiting $X$ as self-dual in $\Span(\cD)$. The coevaluation is given by the span $\pt \xleftarrow{\;r\;} X \xrightarrow{\;\Delta\;} X \times X$. In particular, every object in $\Span(\cD)$ is dualizable.
\end{lem}

\begin{lemma}[{\cite[Proposition~3.3.1]{stefanich2020higher}, see also \cite[Lemma~12.3]{HaugsengSpans}}]
\label{lem:AdjointInCorr}
For a morphism $f\colon X \to Y$ in $\cD$, the morphisms of spans
\[
\begin{tikzcd}
    & X \dar["f"{description}] \ar[bend right=20, swap]{ddl}{f} \ar[bend left=20]{ddr}{f} \\
    & Y \dlar[equal] \drar[equal] \\
    Y && Y
\end{tikzcd}
\qquad \text{ and } \qquad
\begin{tikzcd}
    & X  \dar["{\Delta_f}"{description}] 
    \ar[bend right=20, equal]{ddl} \ar[bend left=20, equal]{ddr} \\
    & X \times_Y X \dlar{\pr_1} \drar[swap]{\pr_2} \\
    X && X.
\end{tikzcd}
\]
are the counit resp.\ unit of an adjunction in $\Span(\cD)$ between the span $h_{\cD}(f) = (\!\!\!\begin{tikzcd}
X & X \lar[swap]{\id} \rar{f} & Y \end{tikzcd}\!\!\!)$ and the span $f^r = (\!\!\!\begin{tikzcd}
Y  & X \lar[swap]{f} \rar{\id} & X\end{tikzcd}\!\!\!)$.
\end{lemma}

\subsection{Traces and dimensions in the \texorpdfstring{$(\infty,2)$}{2}-category of spans}
\label{subsec:DimensionsInCorr}

In this subsection, we recall the folklore fact that the trace of an endomorphism $X \xleftarrow{f} Z \xrightarrow{g} X$ in $\Span(\cD)$ is given by the equalizer $\Eq(f,g)$ of $f$ and $g$. An instance of this appears for example in \cite[Section~4]{BenZviNadler}. For completeness, we provide a proof of this fact. Furthermore, we will show that there is a natural identification $\dim_{\Span(\cD)}(X) \simeq LX$ for $X \in \cD$.

\begin{defn}
\label{def:Equalizer}
Let $f,g\colon Z \to X$ be two maps in $\cD$. The \textit{equalizer} $\Eq(f,g)$ of $f$ and $g$ is defined via the pullback square
\[
\begin{tikzcd}
    \Eq(f,g) \dar \rar \drar[pullback] & X \dar{\Delta} \\
    Z \rar{(f,g)} & X \times X.
\end{tikzcd}
\]
When $Z = X$ and $f = g = \id_X$ is the identity on $X$, we write the equalizer as
\[
    LX := \Eq(\id_X,\id_X)
\]
and call it the \textit{free loop space}, or the free loop object of $X$. Observe that it is equivalent to the cotensor $X^{S^1}$ of $X$ by $S^1$.
\end{defn}

The proof that traces in $\Span(\cD)$ are given by equalizers is a straightforward computation.
\begin{lem}
\label{lem:TracesInCorr}
Consider an endomorphism in $\Span(\cD)$ given by a span $X \xleftarrow{f} Z \xrightarrow{g} X$. 
Under the equivalence $\cD \simeq \Omega \Span(\cD)$ of \Cref{prop:OmegaCorr}, the trace of this endomorphism is equivalent to the equalizer $\Eq(f,g) \in \cD$. In particular, the dimension of $X$ in $\Span(\cD)$ is equivalent to its free loop space $LX \in \cD$.
\end{lem}
\begin{proof}
Spelling out the definition of trace and plugging in the explicit duality data from \Cref{lem:DblCorr}, the trace is given by the following composite of spans:
\[\begin{tikzcd}
	& X && {Z \times X} && X \\
	{\;\;\;\pt\;\;\;} && {X \times X} && {X \times X} && {\;\;\;\pt.\;\;\;}
	\arrow["{f \times \id}"', from=1-4, to=2-3]
	\arrow["{g \times \id}", from=1-4, to=2-5]
	\arrow["r"', from=1-2, to=2-1]
	\arrow["\Delta", from=1-2, to=2-3]
	\arrow["\Delta"', from=1-6, to=2-5]
	\arrow["r", from=1-6, to=2-7]
\end{tikzcd}\]
Observe that we have the following pullback diagram:
\[\begin{tikzcd}
	{\Eq(f,g)} \drar[pullback] & Z\drar[pullback] & X \\
	Z \drar[pullback] & {Z \times X} & {X \times X} \\
	X & {X \times X.}
	\arrow["{f \times \id}"', from=2-2, to=3-2]
	\arrow["{g \times \id}", from=2-2, to=2-3]
	\arrow["\Delta", from=3-1, to=3-2]
	\arrow["\Delta"', from=1-3, to=2-3]
	\arrow["f"', from=2-1, to=3-1]
	\arrow["g", from=1-2, to=1-3]
	\arrow["{(\id,f)}", from=2-1, to=2-2]
	\arrow["{(\id,g)}"', from=1-2, to=2-2]
	\arrow[from=1-1, to=2-1]
	\arrow[from=1-1, to=1-2]
\end{tikzcd}\]
It follows that the above composite span is equivalent to the span $\pt \xleftarrow{\;\;\;} \Eq(f,g) \xrightarrow{\;\;\;} \pt$, as desired.
\end{proof}

The above computation of dimensions in $\Span(\cD)$ gives rise to explicit descriptions of dimensions in $(\infty,2)$-categories $\cE$ which receive a bivariant theory from $\cD$:

\begin{cor}
\label{cor:BivariantTheoryGivesDualizability}
Let $F\colon \cD \to \cE$ be a symmetric monoidal bivariant theory, with $\cD$ and $\cE$ as in \Cref{def:bivariantTheory}.
\begin{enumerate}[(a)]
    \item For every object $X \in \cD$, the composite
    \[
        F(X) \tensor F(X) \simeq F(X \times X) \xrightarrow{F(\Delta)^r} F(X) \xrightarrow{F(r)} F(\pt) \simeq \unit
    \]
    exhibits the object $F(X) \in \cE$ as self-dual in $\cE$. In particular, objects in the image of $F$ are dualizable. 
    \item For every object $X \in \cD$, the dimension of $F(X) \in \cE$ is given by the composite
    \[
        \unit_{\cE} \simeq F(\pt) \xrightarrow{F(r)^r} F(LX) \xrightarrow{F(r)} F(\pt) \simeq \unit_{\cE}.
    \]
\end{enumerate}
\end{cor}
\begin{proof}
By the universal property of the $(\infty,2)$-category $\Span(\cD)$, the functor $F$ uniquely extends to a symmetric monoidal 2-functor $F'\colon \Span(\cD) \to \cE$. As $F'$ preserves duality data and adjunction data, statement a) thus follows from \Cref{lem:DblCorr}, while statement b) follows from \Cref{lem:TracesInCorr}.
\end{proof}

\subsubsection{Coherence}
\label{subsec:DimensionsInCorrCoherence}

We have seen above that for an object $X$ of a left exact $\infty$-category $\cD$, the dimension of $X$ in $\Span(\cD)$ is given by its free loop space: $\dim_{\Span(\cD)}(X) \simeq LX$. The goal of the remainder of subsection is to make this calculation functorial in $X$, see \Cref{thm:DimensionInCorr} below. This is more subtle than it looks: to get at the functoriality of the trace functor, we need to produce all the higher duality data for the objects in $\Span(\cD)$ in a coherent fashion. 

We will work around this problem by observing that the computation of dimensions in $\Span(\cD)$ can be done uniformly in $\cD \in \Cat_{\infty}^{\lex}$, allowing us to reduce to the `universal' left exact $\infty$-category. More precisely, we will consider the left exact $\infty$-category $\cD = (\Spc^{\fin})\catop$, the opposite of the $\infty$-category of finite spaces. By using the fact that $(\Spc^{\fin})\catop$ is freely generated under finite limits by the point $\pt \in (\Spc^{\fin})\catop$, it is possible to reduce the coherence problem to a non-coherent statement, which was solved in \Cref{lem:TracesInCorr}. In the remainder of this subsection, we will fill in the details of this proof strategy.

\begin{lem}
The functor $h_{\cD}\colon \cD \to \Span(\cD)$ factors through the subcategory $\Span(\cD)^{\dbl}$.
\end{lem}
\begin{proof}
Since $\Span(\cD)^{\dbl} \hookrightarrow \Span(\cD)$ is a non-full subcategory by \Cref{cor:Morphisms_In_Dbl}, it suffices to check this on the level of objects and morphisms. By \Cref{lem:DblCorr}, $h_{\cD}$ carries every object in $\cD$ to a dualizable object of $\Span(\cD)$, and since $h_{\cD}$ is a bivariant theory, it carries morphisms in $\cD$ to left adjoint morphisms in $\Span(\cD)$. The statement thus follows from the description of objects and morphisms of $\Span(\cD)^{\dbl}$ given in \Cref{cor:Morphisms_In_Dbl}.
\end{proof}

\begin{rem}
The functor $h_{\cD}\colon \cD \to \Span(\cD)^{\dbl}$ is in fact an equivalence of $\infty$-categories. As we will not need this statement, we leave the proof to the reader.
\end{rem}

It is relatively straightforward to show that the equivalence $\dim_{\Span(\cD)}(X) \simeq LX$ is natural with respect to \textit{equivalences} in the $\infty$-category $\cD$.

\begin{lem}
\label{lem:DimensionInCorrPrelim}
For every left exact $\infty$-category $\cD$, the following diagram naturally commutes:
\[
\begin{tikzcd}
    \cD^{\simeq} \rar{h_{\cD}} \dar[swap]{L} & \Span(\cD)^{\dbl} \dar{\dim_{\Span(\cD)}} \\
    \cD \rar["i_{\cD}", "\simeq"'] & \Omega \Span(\cD).
\end{tikzcd}
\]
\end{lem}
\begin{proof}
The four corners of the square are natural in $\cD \in \Cat^{\lex}_{\infty}$ and the four edges form natural transformations. The assignment $\cD \mapsto \cD^{\simeq}$ is corepresented by the left exact $\infty$-category $(\Spc^{\fin})\catop$. Indeed, the $\infty$-category $\Spc^{\fin}$ is freely generated under finite colimits by the point, see (the proof of) \cite[Proposition~5.3.6.2]{htt}. It thus suffices, by the Yoneda lemma, to show that the two composites agree for $\cD = (\Spc^{\fin})\catop$ when evaluated at the point $\pt \in (\Spc^{\fin})\catop$. This is an instance of \Cref{lem:TracesInCorr}.
\end{proof}

In order to enhance the above commutative diagram to a diagram defined on all of $\cD$ rather than just its groupoid core $\cD^{\simeq}$, we will need some understanding of the interaction between span-categories and functor categories.

\begin{lem} Let $\cE$ and $\cD$ be $\infty$-categories such that $\cD$ is left exact. The inclusion $\Fun(\cE,\cD) \hookrightarrow \Fun(\cE,\Span(\cD)) \hookrightarrow \Fun_{\lax}(\cE,\Span(\cD))$ is a symmetric monoidal bivariant theory, i.e.\ it is symmetric monoidal, sends all morphisms to left adjoint morphisms and sends pullback squares to right adjointable squares. 
\end{lem}
\begin{proof}
The inclusion $\Fun(\cE,\cD) \hookrightarrow \Fun_{\lax}(\cE,\Span(\cD))$ inherits a symmetric monoidal structure from applying the finite-product-preserving functor $\Fun(\cE,-)$ to the symmetric monoidal functor $\cD \to \Span(\cD)$.

Next, given a morphism $\alpha\colon F \to G$ in $\Fun(\cE,\cD)$, we have to show its image in $\Fun_{\lax}(\cE,\Span(\cD))$ is a left adjoint. By \cite[Theorem~4.6]{Haugseng2021Lax}, it suffices to show that this image in $\Fun_{\lax}(\cE,\Span(\cD))$ comes from a morphism in the functor category $\Fun(\cE,\Span(\cD))$, and that for every $e \in \cE$ the map $\alpha(e)\colon F(e) \to G(e)$ is a left adjoint morphism in $\Span(\cD)$. But the former is automatic, and the latter follows from the fact that $h_{\cD}\colon \cD \to \Span(\cD)$ sends morphisms to left adjoints.

Finally, given a pullback square 
\[\begin{tikzcd}
	F \drar[pullback] & G \\
	H & K
	\arrow["\alpha"', from=1-1, to=2-1]
	\arrow["\beta", from=1-1, to=1-2]
	\arrow["\delta"', from=2-1, to=2-2]
	\arrow["\gamma", from=1-2, to=2-2]
\end{tikzcd}\]
in $\Fun(\cE,\cD)$, we have to show that the resulting Beck-Chevalley map (see e.g.\ \cite[Section 2.2]{TeleAmbi})
\[\begin{tikzcd}
	F & G \\
	H & K
	\arrow["\beta", from=1-1, to=1-2]
	\arrow["\delta"', from=2-1, to=2-2]
	\arrow["{\alpha^*}", from=2-1, to=1-1]
	\arrow["{\gamma^*}"', from=2-2, to=1-2]
	\arrow[shorten <=4pt, shorten >=4pt, Rightarrow, from=1-1, to=2-2]
\end{tikzcd}\] 
in $\Fun_{\lax}(\cE,\Span(\cD))$ is an equivalence. For every $e \in \cE$, evaluation of the pullback square at $e \in \cE$ gives a pullback square in $\cD$. Hence, since $h_{\cD}\colon \cD \to \Span(\cD)$ is a bivariant theory, the above Beck-Chevalley map is an equivalence in $\Span(\cD)$ after evaluating at $e$:
\[
\begin{tikzcd}
	{F(e)} & {G(e)} \\
	{H(e)} & {K(e).}
	\arrow["{\beta_e}", from=1-1, to=1-2]
	\arrow["{\delta_e}"', from=2-1, to=2-2]
	\arrow["{\alpha^*_e}", from=2-1, to=1-1]
	\arrow["{\gamma^*_e}"', from=2-2, to=1-2]
	\arrow["\simeq"{description}, shorten <=4pt, shorten >=4pt, Rightarrow, from=1-1, to=2-2]
\end{tikzcd}
\]
Since the 2-functors $\ev_e\colon \Fun_{\lax}(\cE,\Span(\cD)) \to \Span(\cD)$ for $e \in \cE$ are jointly conservative on morphism $\infty$-categories, this gives the claim. 
\end{proof}

By the previous lemma, the symmetric monoidal inclusion $\Fun(\cE,\cD) \hookrightarrow \Fun_{\lax}(\cE,\Span(\cD))$ universally extends to a symmetric monoidal functor
\[
    \Phi_{\cE,\cD}\colon \Span(\Fun(\cE,\cD)) \to \Fun_{\lax}(\cE,\Span(\cD)).
\]
The following lemma shows that this functor interacts well with the equivalence $i_{\cD}\colon \cD \iso \Omega \Span(\cD)$:

\begin{lem}
\label{lem:OmegaCorrVSPhi}
For every $\infty$-category $\cE$, there is a natural homotopy making the following diagram of $\infty$-groupoids commute:
\begin{equation}
\label{eq:OmegaCorrVSPhi}
\begin{tikzcd}
    \Fun(\cE,\cD)^{\simeq} \ar[equal]{rr} \dar[swap]{i_{\Fun(\cE,\cD)}} && \Fun(\cE,\cD)^{\simeq} \dar{\Fun(\cE,i_{\cD})} \\
    \Omega \Span(\Fun(\cE,\cD))^{\simeq} \ar{r}{\Omega\Phi_{\cE,\cD}} &  \Omega\Fun_{\lax}(\cE, \Span(\cD))^{\simeq} \rar{\simeq} &  \Fun(\cE,\Omega \Span(\cD))^{\simeq}.
\end{tikzcd}
\end{equation}
\end{lem}
\begin{proof}
Under the identification $\Fun(\cE,\cE')^{\simeq} \simeq \Map_{\Cat_{\infty}}(\cE,\cE')$, the composite along the left and bottom of the diagram (\ref{eq:OmegaCorrVSPhi}) constitutes a map from $\Map_{\Cat_{\infty}}(\cE,\cD)$ to $\Map_{\Cat_{\infty}}(\cE,\Omega \Span(\cD))$ which is natural in $\cE \in \Cat_{\infty}$. By the Yoneda lemma, it is thus induced by a functor $j_{\cD}\colon \cD \to \Omega \Span(\cD)$. From the definition it is immediate that the functors $i_{\cD}$ and $j_{\cD}$ agree on groupoid cores and that the diagram (\ref{eq:OmegaCorrVSPhi}) commutes if we replace $i_{\cD}$ by $j_{\cD}$. It thus remains to show that the maps $i_{\cD}$ and $j_{\cD}$ are equivalent as functors $\cD \to \Omega \Span(\cD)$.

Note that the construction of $j_{\cD}$ is natural in $\cD \in \Cat^{\lex}_{\infty}$. We may thus consider the composite $k_{\cD}:= i^{-1}_{\cD}j_{\cD}\colon \cD \to \cD$ and regard it as an endomorphism $k\colon U \to U$ of the forgetful functor $U\colon \Cat^{\lex}_{\infty} \to \Cat_{\infty}$. We need to show that $k$ is the identity of $U$. Since $i_{\cD}$ and $j_{\cD}$ agree on groupoid cores, the induced map $k_{\cD}\colon \cD^{\simeq} \to \cD^{\simeq}$ on groupoid cores is naturally equivalent to the identity. It thus remains to prove the following result, \Cref{lem:Rigidity_Left_Exact_Categories}, which we record separately because of its independent interest.
\end{proof}

\begin{lemma}
\label{lem:Rigidity_Left_Exact_Categories}
    Let $k\colon U \to U$ be an endomorphism of the forgetful functor $U\colon \Cat^{\lex}_{\infty} \to \Cat_{\infty}$. Assume that the induced endomorphism $k^{\simeq}\colon U^{\simeq} \to U^{\simeq}$ of the functor $U^{\simeq}\colon \Cat^{\lex}_{\infty} \to \Spc$ is equivalent to the identity of $U^{\simeq}$. Then $k$ is equivalent to the identity of $U$.
\end{lemma}
\begin{proof}
The forgetful functor $U$ admits a left adjoint $F\colon \Cat_{\infty} \to \Cat_{\infty}^{\lex}$, sending an $\infty$-category $\cE$ to the subcategory of the presheaf category $\PSh(\cE\catop)\catop$ generated under finite limits by the representable objects. In particular, the transformation $k\colon U \to U$ corresponds to a transformation $k'\colon \id \to U \circ F$, and the problem translates to showing that $k'$ is equivalent to the unit transformation $u\colon \id \to U \circ F$.

First we will show that $k'$ factors through $u$ via some transformation $k''\colon \id \to \id$. To see this, observe that for an $\infty$-category $\cE$ the unit map $u_{\cE}\colon \cE \to U(F(\cE))$ is fully faithful, both sides being full subcategories of $\PSh(\cE\catop)\catop$. It follows in particular that $u\colon \id \to U \circ F$ is a monomorphism in $\Fun(\Cat_{\infty},\Cat_{\infty})$, and thus it is a \textit{property} for $k'$ to factor through $u$. Moreover, as the property of factoring through a monomorphism in a functor category can be checked pointwise, it suffices to show that for each $\cE \in \Cat_{\infty}$ the functor $k'_{\cE}\colon \cE \to U(F(\cE))$ factors through $u_{\cE}\colon \cE \to U(F(\cE))$ via some map $k''_{\cE}\colon \cE \to \cE$. By construction, $k'_{\cE}\colon \cE \to U(F(\cE))$ is given by the composite
\[
    \cE \xrightarrow{u_{\cE}} U(F(\cE)) \xrightarrow{k_{F(\cE)}} U(F(\cE)).
\]
By assumption on $k$, the second map induces the identity on groupoid cores, proving that $k'_{\cE}$ factors through $u_{\cE}$ on objects. By fully faithfulness of $u_{\cE}\colon \cE \to U(F(\cE))$ this means $k'_{\cE}$ factors through $u_{\cE}$ as desired.

We are thus left with showing that the resulting endomorphism $k''\colon \id_{\Cat_{\infty}} \to \id_{\Cat_{\infty}}$ of the identity on $\Cat_{\infty}$ is equivalent to the identity. This is immediate: the space $\End(\id_{\Cat_\infty})$ of endomorphisms of $\id_{\Cat_{\infty}}$ is contractible. This follows from a theorem by To\"en, which says that the full subcategory of $\Fun(\Cat_{\infty},\Cat_{\infty})$ spanned by the equivalences is equivalent to the discrete category $\{\id, (-)^\op\}$, see \cite[Theorem~4.4.1]{lurie2009Goodwillie}.
\end{proof}

We are now ready for the main theorem of this section.

\begin{thm}
\label{thm:DimensionInCorr}
For every left exact $\infty$-category $\cD$, the following diagram naturally commutes:
\[
\begin{tikzcd}
    \cD \rar{h_{\cD}} \dar[swap]{L} & \Span(\cD)^{\dbl} \dar{\dim_{\Span(\cD)}} \\
    \cD \rar["i_{\cD}", "\simeq"'] & \Omega \Span(\cD).
\end{tikzcd}
\]
\end{thm}
\begin{proof}
By the Yoneda lemma, it suffices to prove that the diagram commutes after applying the functor $\Map_{\Cat_{\infty}}(\cE,-) = \Fun(\cE,-)^{\simeq}$ for every $\infty$-category $\cE$, naturally in $\cE$. Here we may use the following naturally commutative diagram:
\[
\begin{tikzcd}[cramped, column sep = 20]
    \Fun(\cE,\cD)^{\simeq} \ar[swap]{ddd}{(-)^{S^1}} \drar["h_{\Fun(\cE,\cD)}"{description}] \ar{rrr}{\Fun(\cE,h_{\cD})^{\simeq}} &&& \Fun(\cE,\Span(\cD)^{\dbl})^{\simeq} \ar[ddd,"{\Fun(\cE,\dim_{\Span(\cD)})^{\simeq}}"{description}]\\
    & \Span(\Fun(\cE,\cD))^{\dbl, \simeq} \dar["{\dim_{\Span(\Fun(\cE,\cD))}}"{description}] \rar{\Phi_{\cE,\cD}^{\dbl}} & \Fun_{\lax}(\cE,\Span(\cD))^{\dbl, \simeq} \dar["{\dim_{\Fun_{\lax}(\cE,\Span(\cD))}}"{description}] \urar{\simeq} \\
    & \Omega \Span(\Fun(\cE,\cD))^{\simeq} \rar{\Omega \Phi_{\cE,\cD}} & \Omega \Fun_{\lax}(\cE,\Span(\cD))^{\simeq} \drar{\simeq} \\
    \Fun(\cE,\cD)^{\simeq} \urar["i_{\Fun(\cE,\cD)}"{description}] \ar{rrr}{\Fun(\cE,i_{\cD})^{\simeq}} &&& \Fun(\cE,\Omega \Span(\cD))^{\simeq}.
\end{tikzcd}
\]
The left square commutes by \Cref{lem:DimensionInCorrPrelim} applied to $\Fun(\cE,\cD)$. The top square commutes by definition of $\Phi_{\cE,\cD}$. The right square commutes by definition of the higher categorical dimension functor (\Cref{def:Higher_Categorical_Trace_Functor}). The bottom square commutes by \Cref{lem:OmegaCorrVSPhi}. Finally, the middle square commutes by naturality of the dimension functor applied to the symmetric monoidal 2-functor $\Phi_{\cE,\cD}$.
\end{proof}

\section{Traces of \texorpdfstring{$\cC$}{C}-linear \texorpdfstring{$\infty$}{oo}-categories}
\label{sec:CLinearCategories}
Let $\cC \in \CAlg(\PrL)$ be a presentably symmetric monoidal $\infty$-category, taken to be fixed. In this case, we can form the symmetric monoidal $(\infty,2)$-category $\Mod_{\cC} := \Mod_{\cC}(\PrL)$ of $\cC$-modules in $\PrL$, with tensor product denoted by $\otimes_{\cC}$. We will refer to the objects of $\Mod_{\cC}$ as \textit{$\cC$-linear $\infty$-categories}, to its morphisms as \textit{$\cC$-linear functors} and to its 2-morphisms as \textit{$\cC$-linear transformations}. The mapping $\infty$-categories in $\Mod_{\cC}$ are $\Fun_{\cC}(\cD,\cE)$, the $\infty$-categories of $\cC$-linear colimit preserving functors from $\cD$ to $\cE$; they come equipped with natural enhancements to $\cC$-linear $\infty$-categories, making them into internal mapping objects in $\Mod_{\cC}$. We refer to \cite[Section~4.4]{HSS} for a more precise discussion. 

The goal of this section is to study traces in $\Mod_\cC$, which provide a natural categorification of traces in $\cC$. As $\cC$ is the monoidal unit of $\Mod_{\cC}$, the $\infty$-category $\Omega \Mod_{\cC}$ of \Cref{def:Omega} is given by the $\infty$-category $\Fun_{\cC}(\cC,\cC)$ of $\cC$-linear endofunctors of $\cC$. Observe that evaluation at the monoidal unit $\unit \in \cC$ induces an equivalence $\Fun_{\cC}(\cC,\cC) \iso \cC$, and thus postcomposing the trace functor of $\Mod_{\cC}$ with this identification yields a $\cC$-valued trace functor. Since this functor will play an important role in the remainder of this paper, we give it its own notation and terminology.

\begin{defn}
The \textit{$\cC$-linear trace functor} $\TrMod_{\cC}\colon \Mod_{\cC}^{\trl} \to \cC$ is defined as the following composite:
\[
    \TrMod_{\cC}\colon \Mod_{\cC}^{\trl} \xrightarrow{\tr_{\Mod_{\Cc}}} \Omega \Mod_{\cC} \simeq \Fun_{\cC}(\cC,\cC) \iso \cC.
\]
For a dualizable $\cC$-linear $\infty$-category $\cD$ equipped with a $\cC$-linear endofunctor $F\colon \cD \to \cD$, we call the object $\TrMod_{\Cc}(\cD,F) \in \cC$ the \textit{$\cC$-linear trace} of the pair $(\cD,F)$. When $F$ is the identity endofunctor of $\cD$, we will denote this object by $\TrMod_{\Cc}(\cD)$ and call it the \textit{$\Cc$-linear trace of $\cD$.}\footnote{An alternative name would be `$\cC$-linear dimension'.}
\end{defn}

\begin{example}
\label{ex:CLinear_Trace_Scalar}
Assume that $\cD = \Cc$, so that $F\colon \Cc \to \Cc$ is given by tensoring with an object $X \in \Cc$. Then by \Cref{ex:Generalized_Trace_Unit} the $\Cc$-linear trace is simply given by $X$ itself:
\[
    \TrMod_{\Cc}(\Cc,X) = X \qin \Cc.
\]
\end{example}

\begin{example}
    For a space $A$, recall that the functor category $\Cc^A = \Fun(A,\Cc)$ is dualizable in $\Mod_{\Cc}$ with dual given by $\Cc^A$ itself; see \Cref{cor:FreeCLinearCategoryIsDualizable} for details. An explicit computation shows that the $\Cc$-linear trace of $\Cc^A$ is given as
    \[
        \TrMod_{\Cc}(\Cc^A) = \unit_{\Cc}[LA] \qin \Cc,
    \]
    where $LA$ is the free loop space of $A$. This example will be treated in detail in \Cref{subsection:tracesfree} below.
\end{example}

\begin{example}
    For an algebra object $R \in \Alg(\Cc)$, Lurie showed in \cite[Proposition~4.6.3.12]{HA} that the $\Cc$-linear $\infty$-category $\RMod_R(\Cc)$ of right-modules over $R$ is dualizable in $\Mod_{\Cc}$, with dual given by the $\infty$-category $\LMod_R(\Cc)$ of left-modules. It is a folklore fact, reviewed for completeness in \Cref{subsection:HHtrace}, that the $\Cc$-linear trace of $\RMod_R(\Cc)$ is given by the $\Cc$-linear \textit{(topological) Hochshild homology} of $R$:
    \[
        \TrMod_{\Cc}(\RMod_R(\Cc)) = \THH_{\Cc}(R) := R \otimes_{R\otimes R\catop} R \qin \Cc.
    \]
    More generally, for an $R$-bimodule $M$ the $\Cc$-linear trace of the functor $- \otimes_R M\colon \RMod_R(\Cc) \to \RMod_R(\Cc)$ is $\THH_{\Cc}(R;M) := R \otimes_{R\otimes R\catop} M$.
\end{example}

\begin{example}
    Let $X$ be a locally compact Hausdorff space. Using proper base change, one can show that if $\Cc$ is stable, then the $\Cc$-linear $\infty$-category $\Shv(X;\Cc)$ of $\Cc$-valued sheaves on $X$ is dualizable in $\Mod_{\Cc}$ with dual given by $\Shv(X;\Cc)$ itself. The evaluation and coevaluation maps can be written down explicitly in terms of proper pushforwards, and the resulting $\Cc$-linear trace recovers the compactly supported cohomology of $X$ with coefficients in $\Cc$ (see for example \cite[Theorem~1.1]{KiuShendeZhang}):
    \[
        \TrMod_{\Cc}(\Shv(X;\Cc)) \simeq \Gamma_c(X;\unit_{\Cc}).
    \]
\end{example}

\begin{rem}
\label{rem:CLinearTraceSymmetricMonoidal}
By \Cref{cor:TraceSymmetricMonoidal}, the functor $\TrMod_{\cC}$ admits a canonical refinement to a symmetric monoidal functor of $\infty$-categories. 
In particular, given two $\cC$-linear endofunctors $(\cD,F), (\cD',F') \in \Mod^{\trl}_{\cC}$, there is a natural equivalence
\[
    \TrMod_{\Cc}(\cD \tensor_{\cC} \cD', F \tensor_\cC F') \simeq \TrMod_{\Cc}(\cD,F) \tensor \TrMod_{\Cc}(\cD',F') \qin \cC
\]
\end{rem}

Let us spell out the behavior of the $\cC$-linear trace functor on objects and morphisms. The $\infty$-category $\Mod_{\cC}^{\trl}$ is a (non-full) $(\infty,1)$-subcategory of the oplax functor category $\Fun_{\oplax}(B\NN,\Mod_{\cC})$, see \Cref{cor:Morphisms_In_Dbl}. Its objects are given by pairs $(\cD,F)$, where $\cD \in \Mod_{\cC}^{\dbl}$ is a dualizable $\cC$-linear $\infty$-category and $F\colon \cD \to \cD$ is a $\cC$-linear endofunctor. Its morphisms $(\cD,F) \to (\cE,G)$ are pairs $(H,\alpha)$, where $H\colon \cD \to \cE$ is a left adjoint morphism in $\Mod_{\cC}$ and where $\alpha\colon H \circ F \Rightarrow G \circ H$ is a $\cC$-linear transformation:
\begin{equation*}
\begin{tikzcd}
    \cD \dar[swap]{F} \rar{H} & \cE \dar{G} \\
    \cD \rar{H} & \cE.
    \arrow[shorten <=10pt, shorten >=10pt, Rightarrow, from=2-1, to=1-2, "\alpha"]
\end{tikzcd}
\end{equation*}
For an object $(\cD,F)$ of $\Mod^{\trl}_{\cC}$, its $\cC$-linear trace $\TrMod_{\cC}(\cD,F) \in \cC$ is the image of the monoidal unit $\one \in \cC$ under the following composite:
\begin{align*}
    \cC \xrightarrow{\coev} \cD \otimes_{\cC} \cD^\vee \xrightarrow{F \otimes \id} \cD \otimes_{\cC} \cD^\vee \overset{\tau}{\simeq} \cD^\vee \otimes_{\cC} \cD \xrightarrow{\ev} \cC.
\end{align*}
For a morphism $(H,\alpha)\colon (\cD,F) \to (\cE,G)$ in $\Mod^{\trl}_{\cC}$, the map $\TrMod_{\cC}(H,\alpha)\colon \TrMod_{\cC}(\cD,F) \to \TrMod_{\cC}(\cE,G)$ is given by evaluating the following composite transformation at the monoidal unit $\one \in \cC$:

\begin{equation}
\label{eq:FunctorialityTHH}
\begin{tikzcd}
	\cC && {\cD\otimes_{\cC} \cD^\vee} && {\cD\otimes_{\cC} \cD^\vee} \\
	\\
	{\cE\otimes_{\cC} \cE^\vee} && {\cD\otimes_{\cC} \cE^\vee} && {\cE\otimes_{\cC} \cD^\vee} && {\cD\otimes_{\cC} \cD^\vee} \\
	\\
	&& {\cE\otimes_{\cC} \cE^\vee} && {\cE\otimes_{\cC} \cE^\vee} && \cC.
	\arrow["{\coev_{\cE}}"', from=1-1, to=3-1]
	\arrow["{\coev_{\cD}}", from=1-1, to=1-3]
	\arrow["{H^r\otimes 1}"{description}, from=3-1, to=3-3]
	\arrow["{1\otimes (H^r)^\vee}"{description}, from=1-3, to=3-3]
	\arrow[""{name=0, anchor=center, inner sep=0}, "{H\otimes 1 }"{description}, from=3-3, to=5-3]
	\arrow[""{name=1, anchor=center, inner sep=0}, equal, curve={height=18pt}, from=3-1, to=5-3]
	\arrow["{H^r\otimes 1}"{description}, from=3-5, to=3-7]
	\arrow["{\ev_{\cD}}", from=3-7, to=5-7]
	\arrow[""{name=2, anchor=center, inner sep=0}, equal, curve={height=-18pt}, from=1-5, to=3-7]
	\arrow[""{name=3, anchor=center, inner sep=0}, "{H\otimes 1}"{description}, from=1-5, to=3-5]
	\arrow["{F\otimes 1}", from=1-3, to=1-5]
	\arrow["{\ev_{\cE}}"{description}, from=5-5, to=5-7]
	\arrow["{1\otimes(H^r)^\vee}"{description}, from=3-5, to=5-5]
	\arrow["G\otimes1"', from=5-3, to=5-5]
	\arrow[shorten <=23pt, shorten >=23pt, Rightarrow, no head, from=3-1, to=1-3]
	\arrow[shorten <=23pt, shorten >=23pt, Rightarrow, no head, from=5-5, to=3-7]
	\arrow["c\otimes1"{description}, shorten <=11pt, shorten >=5pt, Rightarrow, from=3-3, to=1]
	\arrow["{\alpha\otimes(H^r)^\vee}"{description}, shorten <=20pt, shorten >=20pt, Rightarrow, from=3, to=0]
	\arrow["u\otimes1"{description}, shorten <=11pt, shorten >=5pt, Rightarrow, from=2, to=3-5]
\end{tikzcd}
\end{equation}
This section is organized as follows: in \Cref{subsection:freecofree}, we define and review basic properties of $\cC$-linear $\infty$-categories freely and cofreely generated from spaces; in \Cref{subsec:Adjointability} we discuss maps of spaces $f$ for which the pullback functor $f^*$ admits a $\cC$-linear right adjoint, called $\cC$-adjointable maps; in \Cref{subsection:tracesfree} we compute the $\cC$-linear trace of $\cC$-linear $\infty$-categories freely generated by spaces and the morphisms between them induced by maps of spaces, and in the case of $\cC$-adjointable maps, we introduce a ``wrong-way'' morphism called the free loop transfer; in \Cref{subsection:HHtrace}, we prove the folklore identification of Hochschild homology as a trace in the case of module categories; and finally in \Cref{subsec:Free_CLinear_Cats_As_Module_Cats}, we discuss the identification of $\cC$-linear $\infty$-categories of local systems with module $\infty$-categories over the correspnding loop spaces. 
\subsection{Free and cofree \texorpdfstring{$\cC$}{C}-linear \texorpdfstring{$\infty$}{oo}-categories}\label{subsection:freecofree}
There are two classes of $\cC$-linear $\infty$-categories which will play an important role throughout this article: for every space $A$ we have the \textit{free} $\cC$-linear $\infty$-category $\cC[A]$ and the \textit{cofree} $\cC$-linear $\infty$-category $\cC^A$. The goal of this subsection is to recall the basic properties these $\infty$-categories have.

\begin{defn}
Let $A$ be a space and let $\cD$ be a $\cC$-linear $\infty$-category. We define the $\cC$-linear $\infty$-categories $\cD[A]$ and $\cD^A$ by
\[
    \cD[A] := \colim_A \cD \qin \Mod_{\cC} \qquad\qquad \text{ and } \qquad\qquad \cD^A := \lim_A \cD \qin \Mod_{\cC}.
\]
For a map $f\colon A \to B$ of spaces, we denote by
\[
    f_!\colon \cD[A] \to \cD[B] \qquad\qquad \text{ and } \qquad\qquad f^*\colon \cD^B \to \cD^A
\]
the induced $\cC$-linear functors. Their right adjoints in $\widehat{\Cat}_{\infty}$ are denoted by
\[
    f^*\colon \cD[B] \to \cD[A] \qquad\qquad \text{ and } \qquad\qquad f_*\colon \cD^A \to \cD^B.
\]
We will show in \Cref{cor:Free_VS_Cofree_CLinear_Category} that there is an equivalence $\cD[A] \simeq \cD^A$. Nevertheless, we choose to distinguish them in the notation to emphasize the different roles they play. Under this identification, the two functors called $f^*$ get identified, justifying the notation.
\end{defn}

\begin{rem}
Since the forgetful functors $\Mod_{\cC} \to \PrL$ and $\PrL \to \Cat_{\infty}$ preserve limits, the underlying $\infty$-category of $\cD^A$, and therefore also of $\cD[A]$, is the functor category $\Fun(A,\cD)$.
\end{rem}

The $\cC$-linear $\infty$-category $\cC[A]$ is called the \textit{free} $\cC$-linear $\infty$-category on $A$, as the $\cC$-linear functors out of it into some $\cC$-linear $\infty$-category $\cD$ correspond to functors $A \to \cD$. In fact, we have the following stronger statement:

\begin{lem}
\label{lem:MappingOutOfFreeCModule}
Let $\cD$ be a $\cC$-linear $\infty$-category and let $A$ be a space. 
There are natural equivalences of $\cC$-linear $\infty$-categories
\[
    \cD[A] \iso \cD \otimes_\cC \cC[A] \qquad\qquad \text{ and } \qquad\qquad \Fun_{\cC}(\cC[A],\cD) \iso \cD^A.
\]
\end{lem}
\begin{proof}
As the functor $\cD\otimes_\cC -\colon \Mod_{\cC} \to \Mod_{\cC}$ preserves colimits, the first equivalence follows from the equivalence $\cD \simeq \cD \tensor_{\cC} \cC$. Similarly, as the functor $\Fun_{\cC}(-,\cD)\colon \Mod_{\cC}\catop \to \Mod_{\cC}$ turns colimits in $\Mod_{\cC}$ into limits, the second equivalence similarly follows from the fact that evaluation at $\unit \in \cC$ induces an equivalence $\Fun_{\cC}(\cC,\cD) \iso \cD$.
\end{proof}

\begin{nota}
\label{nota:inducedmap}
We will frequently abuse notation and denote the $\cC$-linear functor associated with an object $X \in \cD^A$ simply by $X\colon \cC[A] \to \cD$, or sometimes by $X_{\cC}\colon \cC[A] \to \cD$ when we need to distinguish it from $X$ itself. For an object $Y \in \cC$ we will write $Y\colon \cC \to \cC$ for the functor $- \tensor Y\colon \cC \to \cC$.
\end{nota}

The assignments $A \mapsto \cC[A]$ and $A \mapsto \cD^A$ satisfy various adjointability properties, making them bivariant theories in the sense of \Cref{def:bivariantTheory}.

\begin{prop}
\label{prop:FreeCofreeCLinearCategories}
\begin{enumerate}[(1)]
    \item \label{it:FreeCLinearCategoryBivariant} The functor $\cC[-]\colon \Spc \to \Mod_{\cC}$ is a symmetric monoidal bivariant theory.
    \item \label{it:CofreeCLinearCategoryBivariant} For every $\cC$-linear $\infty$-category $\cD$, the functors $\cD[-]\colon \Spc \to \Mod_{\cC}$ and $\cD^{(-)}\colon \Spc \to \Mod_{\cC}\catop$ are bivariant theories.
    \item \label{it:RestrictionInduction}
    For every $\cC$-linear $\infty$-category $\cD$, there are equivalences of bivariant theories
    \[ 
    \cD\otimes_\cC\cC[-] \simeq \cD[-] \qquad\qquad \text{ and } \qquad\qquad \Fun_{\cC}(\cC[-],\cD) \simeq \cD^{(-)}.
    \]
\end{enumerate}
\end{prop}
\begin{proof}
For part (\ref{it:FreeCLinearCategoryBivariant}), notice that the functor $A \mapsto \cC[A]$ factors through the symmetric monoidal 2-functor $\cC \tensor_{\Spc} -\colon \PrL \to \Mod_{\cC}$. It thus suffices to show the claim when $\cC = \Spc$. More concretely, this means we have to show that the functor $\Spc[-]\colon \Spc \to \PrL$ is symmetric monoidal, sends maps of spaces to internal left adjoints in $\PrL$ and sends pullback squares of spaces to right adjointable squares in $\PrL$. Symmetric monoidality of $\Spc[-]\colon \Spc \to \PrL$ is automatic from the observation that it is the unique colimit-preserving symmetric monoidal functor $\Spc \to \PrL$. For the other two conditions, we use that for every space $A$ there is an equivalence between $\Spc[A]$ and the $\infty$-category $\PSh(A) := \Fun(A\catop,\Spc)$ of presheaves on $A$, see \cite[Theorem~5.1.5.6]{htt}. For a morphism of spaces $f\colon A \to B$, the induced functor $f_!\colon \PSh(A) \to \PSh(B)$ is given by left Kan extension. Consequently, the right adjoint $f^*\colon \PSh(B) \to \PSh(A)$ is given by precomposing with $f$, and hence preserves colimits. We deduce that $f^*$ is an internal right adjoint of $f_!$ in $\PrL$. It remains to show that $\Spc[-]$ sends pullback squares of spaces to adjointable squares in $\PrL$, which is an instance of \cite[Proposition~4.3.3]{HLAmbiKn}. This finishes the proof of part (\ref{it:FreeCLinearCategoryBivariant}).

Parts (\ref{it:CofreeCLinearCategoryBivariant}) and (\ref{it:RestrictionInduction}) follow immediately from part (\ref{it:FreeCLinearCategoryBivariant}). Indeed, observe that the post-composition of a bivariant theory $\cD \to \cE$ with a 2-functor $\cE \to \cE'$ is again a bivariant theory. Since $\cD\otimes_\cC -$ and $\Fun_{\cC}(-,\cD)\colon \Mod_{\cC} \to \Mod_{\cC}\catop$ are 2-functors, it thus follows from part (\ref{it:FreeCLinearCategoryBivariant}) that the functors $\cD\otimes_\cC \cC[-]$ and  $\Fun_{\cC}(\cC[-],\cD)$ are bivariant theories. By the natural equivalences $\cD\otimes_\cC \cC[-]\simeq \cD[-]$ and $\Fun_{\cC}(\cC[-],\cD) \simeq \cD^{(-)}$ of \Cref{lem:MappingOutOfFreeCModule}, parts (\ref{it:CofreeCLinearCategoryBivariant}) and (\ref{it:RestrictionInduction}) follow.
\end{proof}

Let us spell out some consequences of \Cref{prop:FreeCofreeCLinearCategories} in a more concrete form. There is for every pair of spaces $A$ and $B$ an equivalence
\[
\cC[A] \tensor_{\cC} \cC[B] \simeq \cC[A \times B].
\]
For every map $f\colon A \to B$ of spaces, the $\cC$-linear functor $f_!\colon \cC[A] \to \cC[B]$ admits a $\cC$-linear right adjoint $f^*\colon \cC[B] \to \cC[A]$, and similarly the $\cC$-linear functor $f^*\colon \cD^B \to \cD^A$ admits a $\cC$-linear left adjoint $f_!\colon \cD^A \to \cD^B$. Furthermore, for every pullback square of spaces
\[
\begin{tikzcd}
        A \rar{f} \dar[swap]{g} \drar[pullback] & B \dar{g'} \\
        C \rar{f'} & D,
    \end{tikzcd}
\]
if we consider the two associated commutative squares in $\Mod_{\cC}$
\[
\begin{tikzcd}
    \cC[A] \rar{f_!} \dar[swap]{g_!} & \cC[B] \dar{g'_!} \\
    \cC[C] \rar{f'_!} & \cC[D]
\end{tikzcd}
\qquad \qquad \text{ and } \qquad \qquad
\begin{tikzcd}
    \cD^A & \cD^B \lar[swap]{f^*} \\
    \cD^C \uar{g^*} & \cD^D \uar[swap]{{g'}^*} \lar{{f'}^*},
\end{tikzcd}
\]
then the left square is right adjointable in $\Mod_{\cC}$, while the right square is left adjointable.

\begin{cor}
\label{cor:RestrictionInduction}
Let $\cD$ be a $\cC$-linear $\infty$-category. For every map $f\colon A \to B$ of spaces, there are naturally commutative squares of $\cC$-linear functors
\[
\begin{tikzcd}
    \Fun_{\cC}(\cC[B],\cD) \dar[swap]{- \circ f_!} \rar{\simeq} & \cD^B \dar{f^*} \\
    \Fun_{\cC}(\cC[A],\cD) \rar{\simeq} & \cD^A
\end{tikzcd}
\qquad\qquad \text{and} \qquad\qquad
\begin{tikzcd}
    \Fun_{\cC}(\cC[A],\cD) \dar[swap]{- \circ f^*} \rar{\simeq} & \cD^A \dar{f_!} \\
    \Fun_{\cC}(\cC[B],\cD) \rar{\simeq} & \cD^B.
\end{tikzcd}
\]
\end{cor}
\begin{proof}
The left square is an instance of \Cref{lem:MappingOutOfFreeCModule}. The right square is obtained from the left square by passing to left adjoints.
\end{proof}

As a corollary of \Cref{prop:FreeCofreeCLinearCategories}, we obtain that the free $\cC$-linear $\infty$-categories $\cC[A]$ are dualizable in $\Mod_{\cC}$.

\begin{cor}
\label{cor:FreeCLinearCategoryIsDualizable}
For every space $A$, the $\cC$-linear pairing
\[
    \cC[A] \tensor_{\cC} \cC[A] \simeq \cC[A\times A] \xrightarrow{\Delta^*} \cC[A] \xrightarrow{A_!} \cC
\]
is non-degenerate. In particular the $\cC$-linear $\infty$-category $\cC[A]$ is self-dual in $\Mod_{\cC}$.
\end{cor}
\begin{proof}
By part (\ref{it:FreeCLinearCategoryBivariant}) of \Cref{prop:FreeCofreeCLinearCategories}, this is an instance of \Cref{cor:BivariantTheoryGivesDualizability}. The coevaluation is given by $\cC \oto{A^*} \cC[A] \oto{\Delta_!} \cC[A \times A] \simeq \cC[A] \otimes_{\cC} \cC[A]$.
\end{proof}

Under the self-duality of the free $\cC$-linear $\infty$-categories $\cC[A]$, the functors $f_!$ and $f^*$ are dual to each other, in the sense of the following lemma:

\begin{lem}
\label{lem:ResDualToInd}
Let $f\colon A \to B$ be a map of spaces. Then the following diagrams commute:
\[
\begin{tikzcd}
    \cC[B] \dar[swap]{\simeq} \ar{rr}{f^*} && \cC[A] \dar{\simeq} \\
    \cC[B]^{\vee} \ar{rr}{(f_!)^{\vee}} && \cC[A]^{\vee}
\end{tikzcd}
\qquad\qquad
\text{ and }
\qquad\qquad
\begin{tikzcd}
    \cC[A] \dar[swap]{\simeq} \ar{rr}{f_!} && \cC[B] \dar{\simeq} \\
    \cC[A]^{\vee} \ar{rr}{(f^*)^{\vee}} && \cC[B]^{\vee}.
\end{tikzcd}
\]
\end{lem}
\begin{proof}
We will prove the commutativity of the left diagram. The proof for the right diagram is analogous and is left to the reader. Expanding the definition of $(f_!)^{\vee}$ by plugging in the explicit evaluation and coevaluation maps from \Cref{cor:FreeCLinearCategoryIsDualizable}, we see it is given by the composite
\[
    \cC[B] \xrightarrow{(\pr_B)^*} \cC[B \times A] \xrightarrow{(\id \times (f,\id))_!} \cC[B \times B \times A] \xrightarrow{(\Delta \times \id)^*} \cC[B \times A] \xrightarrow{(\pr_A)_!} \cC[A].
\]
Observe that the maps $\id \times (f,\id)$ and $\Delta \times \id$ fit into a pullback square
\[\begin{tikzcd}
	A \ar[pullback]{drr} && {B \times A} \\
	{B \times A} && {B \times B \times A,}
	\arrow["{\id \times (f,\id)}", from=1-3, to=2-3]
	\arrow["{\Delta \times 1}", from=2-1, to=2-3]
	\arrow["{(f,\id)}"', from=1-1, to=2-1]
	\arrow["{(f,\id)}", from=1-1, to=1-3]
\end{tikzcd}\]
and thus it follows from part (\ref{it:FreeCLinearCategoryBivariant}) of \Cref{prop:FreeCofreeCLinearCategories} that the above composite is homotopic to
\[
    \cC[B] \xrightarrow{(\pr_B)^*} \cC[B \times A] \xrightarrow{(f,\id)^*} \cC[A] \xrightarrow{(f,\id)_!} \cC[B \times A] \xrightarrow{(\pr_A)_!} \cC[A].
\]
But this composite is the functor $f^*\colon \cC[B] \to \cC[A]$, as desired.
\end{proof}

We deduce from the previous lemma that the free and cofree $\cC$-linear $\infty$-categories $\cC[A]$ and $\cC^A$ are equivalent, and that under this equivalence the functor $f^*\colon \cC[B] \to \cC[A]$ agrees with $f^*\colon \cC^B \to \cC^A$ and  $f_!\colon \cC[A] \to \cC[B]$ agrees with $f_!\colon \cC^A \to \cC^B$.

\begin{cor}
\label{cor:Free_VS_Cofree_CLinear_Category}
For every space $A$ and every $\cC$-linear $\infty$-category $\cD$, there is an equivalence
\[
    \cD[A] \iso \cD^A
\]
of $\cC$-linear $\infty$-categories, natural in $A \in \Spc^{\simeq}$.  Furthermore, for every map $f\colon A \to B$ of spaces, the following diagrams commute:
\[
\begin{tikzcd}
    \cD[B] \rar{\simeq} \dar[swap]{f^*} & \cD^B \dar{f^*} \\
    \cD[A] \rar{\simeq} & \cD^A
\end{tikzcd}
\qquad
\text{ and }
\qquad
\begin{tikzcd}
    \cD[A] \rar{\simeq} \dar[swap]{f_!} & \cD^A \dar{f_!} \\
    \cD[B] \rar{\simeq} & \cD^B.
\end{tikzcd}
\]
\end{cor}
\begin{proof}
We will start by proving the equivalence $\cD[A] \simeq \cD^A$. By self-duality of $\cC[A]$, the evaluation map $\cC[A] \tensor_{\cC} \cC[A] \tensor_{\cC} \cD \to \cD$ adjoints over to an equivalence $\cC[A] \tensor_{\cC} \cD \iso \Fun_{\cC}(\cC[A],\cD)$, which is natural in $A \in \Spc^{\simeq}$. By \Cref{lem:MappingOutOfFreeCModule}, the left-hand side of this equivalence is equivalent to $\cD[A]$ and the right-hand side is equivalent to $\cD^A$. Together, this gives the desired equivalence $\cD[A] \simeq \cD^A$, naturally in $A \in \Spc^{\simeq}$.

Now we show that the two diagrams commute. The right diagram is obtained from the left one by passing to left adjoints, so it suffices to prove that the left diagram commutes. This follows from the following commutative diagram:
\[\begin{tikzcd}
	{\cD[B]} & {\cD \tensor_{\cC} \cC[B]} & {\Fun_{\cC}(\cC[B],\cD)} & {\cD^B} \\
	{\cD[A]} & {\cD \tensor_{\cC} \cC[A]} & {\Fun_{\cC}(\cC[A],\cD)} & {\cD^A.}
	\arrow["\sim", from=1-1, to=1-2]
	\arrow["\sim", from=2-1, to=2-2]
	\arrow["{f^*}"', from=1-1, to=2-1]
	\arrow["{\id \tensor_{\cC} f^*}", from=1-2, to=2-2]
	\arrow["{- \circ f_!}", from=1-3, to=2-3]
	\arrow["{f^*}", from=1-4, to=2-4]
	\arrow["\sim", from=1-3, to=1-4]
	\arrow["\sim", from=2-3, to=2-4]
	\arrow["\sim", from=1-2, to=1-3]
	\arrow["\sim", from=2-2, to=2-3]
\end{tikzcd}\]
The left and right squares commutes by functoriality in $A$ of the equivalences $\cD[A] \simeq \cD \tensor_{\cC} \cC[A]$ and $\Fun_{\Cc}(\Cc[A],\cD) \simeq \cD^A$ from \Cref{lem:MappingOutOfFreeCModule}. The middle square commutes by \Cref{lem:ResDualToInd}.
\end{proof}

\begin{rmk}
The assignment $A \mapsto \cD[A]$ can be turned into a \textit{contravariant} functor via the right adjoints $f^*\colon \cD[B] \to \cD[A]$. With this functoriality, the assignments $A \mapsto \cD[A]$ and $A \mapsto \cD^A$ are equivalent as functors $\Spc\catop \to \Mod_{\cC}$. Indeed, this follows from the observation that both send colimits of spaces to limits in $\Mod_{\cC}$ and agree on the point.
\end{rmk}

\begin{cor}
\label{cor:PointsJointlyConservative}
Let $A$ be a space. Then the functors $a^*\colon \cC[A] \to \cC$ for $a \in A$ are jointly conservative.
\end{cor}
\begin{proof}
Since this is true for the functors $a^*\colon \cC^A \to \cC$, this follows directly from \Cref{cor:Free_VS_Cofree_CLinear_Category}.
\end{proof}

\subsection{Adjointability in \texorpdfstring{$\Mod_{\cC}$}{ModC}}
\label{subsec:Adjointability}

Since $\Mod_{\cC}$ is an $(\infty,2)$-category, there is a notion of a left adjoint morphism in $\Mod_{\cC}$: a $\cC$-linear functor $H\colon \cD \to \cE$ which admits a $\cC$-linear right adjoint $H^r\colon \cE \to \cD$ with $\cC$-linear unit and counit. For emphasis, we will refer to the left adjoint morphisms in $\Mod_{\cC}$ as the \textit{internal left adjoints}. By \cite[Remark~7.3.2.9]{HA}, a $\cC$-linear functor $H\colon \cD \to \cE$ is an internal left adjoint precisely when its right adjoint $G\colon \cE \to \cD$ in $\widehat{\Cat}_{\infty}$ preserves colimits and satisfies the right projection formula: the canonical map
\[
    X \tensor G(Y) \to G(X \tensor Y)
\]
is an equivalence in $\cD$ for all $X \in \cC$ and $Y \in \cE$. 

In fact, it frequently happens that the projection formula comes for free:

\begin{lem}
\label{lem:CLinearityAutomatic}
Let $H\colon \cD \to \cE$ be a $\cC$-linear functor whose right adjoint $G\colon \cE \to \cD$ in $\widehat{\Cat}_{\infty}$ preserves colimits. Assume that one of the following two conditions hold:
\begin{enumerate}[(1)]
    \item The presentable $\infty$-category $\cC$ is an idempotent algebra in $\PrL$;
    \item The dualizable objects in $\cC$ generate $\cC$ under colimits.
\end{enumerate}
Then $H$ is an internal left adjoint in $\Mod_{\cC}$.
\end{lem}
\begin{proof}
In case (1), $\Mod_{\cC}$ is a full subcategory of $\PrL$, so any right adjoint of $H$ in $\PrL$ is automatically a right adjoint of $H$ in $\Mod_{\cC}$. In case (2), we have to show that the map $X \tensor G(Y) \to G(X \tensor Y)$ is an equivalence in $\cD$ for all $X \in \cC$ and $Y \in \cE$. The assumption on $G$ implies that both sides preserve colimits in $X$, and thus by the assumption on $\cC$ it will suffice to prove this whenever $X$ is dualizable in $\Cc$. But in this case, if we let $X^{\vee} \in \cC$ denote a dual of $X$, we observe that the transformation $X \tensor G(-) \to G(X \tensor -)$ is the total mate of the specified equivalence $X^{\vee} \tensor H(-) \xrightarrow{\simeq} H(X^{\vee} \tensor -)$, and thus is itself an equivalence.
\end{proof}

Testing whether a $\cC$-linear functor out of a free $\cC$-linear $\infty$-category $\cC[A]$ is an internal left adjoint can be done pointwise.
\begin{lem}
\label{lem:Criterion_Internal_Left_Adjoint}
A $\cC$-linear functor $F\colon \cC[A] \to \cD$ is an internal left adjoint if and only if the composition $F \circ a_!\colon \cC \to \cD$ is an internal left adjoint for every $a\colon \pt \to A$.
\end{lem}
\begin{proof}
Since internal left adjoints are closed under composition, one direction follows from the fact that the functors $a_!\colon \cC \to \cC[A]$ are internal left adjoints for every $a \in A$. For the other direction, assume that $F \circ a_!$ is an internal left adjoint for every $a \in A$. We need to show that the right adjoint $G\colon \cD \to \cC[A]$ of $F$ preserves colimits and satisfies the projection formula. For the projection formula, consider objects $X \in \cC$ and $Y \in \cD$. We need to show that the canonical map $X \tensor G(Y) \to G(X \tensor Y)$ is an equivalence. As the functors $a^*\colon \cC[A] \to \cC$ are jointly conservative by \Cref{cor:PointsJointlyConservative}, this may be tested after applying $a^*$ for every $a$. But since $a^*$ is $\Cc$-linear, this map becomes the exchange map $X \tensor (a^* \circ G)(Y) \to (a^* \circ G)(X \tensor Y)$ for the composite $a^* \circ G$, which is an equivalence by the assumption that $F \circ a_!$ is an internal left adjoint. We conclude that $G$ satisfies the projection formula. A similar argument shows that $G$ preserves colimits.
\end{proof}

\subsubsection{Adjointability for maps of spaces}
We will be particularly interested in the internal left adjoints which are induced by a map of spaces $g\colon A \to B$. In the covariant direction, we have the $\cC$-linear functor $g_!\colon \cC[A] \to \cC[B]$, which by \Cref{prop:FreeCofreeCLinearCategories}(1) is always an internal left adjoint. In the contravariant direction, we have the $\cC$-linear functor $g^*\colon \cC[B] \to \cC[A]$, which need not be an internal left adjoint in general. The class of morphisms for which this happens to be the case will play a mayor role throughout this article.

\begin{defn}
\label{def:CAdjointability}
     We say that a map of spaces $g\colon A \to B$ is \textit{$\cC$-adjointable} if the $\cC$-linear functor $g^*\colon \cC[B] \to \cC[A]$ is an internal left adjoint in $\Mod_{\cC}$, i.e.\ its right adjoint $g_*\colon \cC[A] \to \cC[B]$ preserves colimits and satisfies the projection formula. We say that a space $A$ is \textit{$\cC$-adjointable} if the map $A \to \pt$ is $\cC$-adjointable.
\end{defn}

It is immediate that the collection of $\cC$-adjointable maps is closed under composition and cartesian products.

\begin{lem}
\label{lem:FiberwiseAdjointableImpliesAdjointable}
Let $g\colon A \to B$ be a map of spaces. If all of the fibers of $g$ are $\cC$-adjointable, then $g$ is $\cC$-adjointable.
\end{lem}
\begin{proof}
By \Cref{lem:Criterion_Internal_Left_Adjoint}, the $\Cc$-linear functor $g^*\colon \cC[B] \to \cC[A]$ is an internal left adjoint as soon as for every point $b\colon \pt \to B$, the composite $g^* \circ b_!\colon \cC \to \cC[A]$ is an internal left adjoint. Letting $\iota_b\colon A_b \to A$ denote the inclusion of the fiber of $g$ at $b$, this composite is equivalent to the $\cC$-linear functor $(\iota_b)_! \circ (A_b)^*\colon \cC \to \cC[A]$. As $(\iota_b)_!$ is always an internal left adjoint and $(A_b)^*$ is an internal left adjoint by assumption, this finishes the proof.
\end{proof}

\begin{war}
    The converse of the previous lemma is not true.\footnote{This was mistakenly claimed in a previous version of this document, in which we also misleadingly referred to the $\Cc$-adjointable maps as \textit{twisted $\Cc$-ambidextrous}.} For a counterexample, consider the category $\cC= \Vect_{k}$ of vector spaces over some field $k$ of characteristic $p$ and let $G$ be the cyclic group of order $p$. Then the map $f\colon \pt\to B^2G$ is $\cC$-adjointable, since the functor $f^*\colon \Fun(B^2G,\Cc) \to \Cc$ is an equivalence and thus in particular an internal left adjoint. However, the fiber $BG$ of $f$ is not $\cC$-adjointable: applying the $G$-fixed point functor to the exact sequence $k[G] \to k \to 0$ gives the non-exact sequence $k \xrightarrow{0} k \to 0$, showing that $(BG)^*\colon \Cc \to \Fun(BG,\Cc)$ does not preserve colimits.
\end{war}

\begin{rmk}
The concept of $\cC$-adjointability is closely related to other notions appearing in the literature:
\begin{enumerate}
    \item In \cite[Definition~2.15]{BCSY2021Fourier}, a map of spaces $g\colon A \to B$ is called \textit{$\cC$-semiaffine} if the restriction functor $g^*\colon \cC^B \to \cC^A$ is an internal left adjoint in $\Mod_{\Cc^B}$. It is clear that any $\cC$-semiaffine map is $\Cc$-adjointable.  Conversely, a variant of the proof of \Cref{lem:FiberwiseAdjointableImpliesAdjointable} shows that if all the \textit{fibers} of $g$ are $\Cc$-adjointable, then $g$ is $\Cc$-semiaffine. Indeed, the main observation is that for a point $b \in B$, the $\Cc^B$-actions on $\Cc$ and $\Cc^{A_b}$ are obtained from the $\Cc$-action via the evaluation functor $b^*\colon \Cc^B \to \Cc$, so that the functor $A_b^*\colon \Cc \to \Cc^{A_b}$ is a $\Cc^B$-linear internal left adjoint if and only if it is a $\Cc$-linear internal left adjoint.
    \item For a space $A$, there exists a `dualizing object' $D_A \in \cC[A]$ and a `twisted norm map' $\Nm_A \colon A_!(- \otimes D_A)\to A_*$ which exhibits $A_!(- \otimes D_A)$ as the terminal $\cC$-linear approximation to the lax $\cC$-linear functor $A_*$. It follows that $A$ is $\cC$-adjointable if and only if it is \textit{twisted $\Cc$-ambidextrous}, i.e.\ if the twisted norm map $\Nm_A$ is an equivalence.
    
    When $\cC$ is the $\infty$-category of spectra, the parameterized spectrum $D_A \in \Sp^A$ was introduced and studied by John Klein \cite{klein2001dualizing} and is called the \textit{dualizing spectrum} of $A$.\footnote{For a topological group $G$, Klein denotes by $D_G$ what here is denoted by $D_{BG}$.} The universal property of the twisted norm map in this setting was proved by \cite[Theorem~I.4.1(v)]{NS}. 
    For general $\cC$, the construction of the twisted norm map and a proof of its universal property appears in the article \cite{Cnossen_Twisted_Ambi} by the second author, where in fact a more general version in the setting of parameterized homotopy theory is introduced.
\end{enumerate}
\end{rmk}

\begin{lem}[cf.\ {\cite[Lemma~21.1.2.14]{SAG}}]
\label{lem:TwistedAmbiClosure}
The collection of $\cC$-adjointable maps is closed under retracts.
\end{lem}
\begin{proof}
Consider a retract diagram
\[
\begin{tikzcd}
    A \dar{g} \rar{i} & A' \dar{g'} \rar{r} & A \dar{g} \\
    B \rar{j} & B' \rar{s} & B,
\end{tikzcd}
\]
and assume that $g'$ is $\cC$-adjointable. We have to show that $g^*\colon \cC[B]\to \cC[A]$ is an internal left adjoint. This follows by writing $g^*$ as a retract of the composite $r_!(g')^*j_!\colon \cC[B]\to \cC[A]$: indeed, the right adjoint $g_*$ is then also a retract of the composite $j^*(g')_*r^*$, compatibly with the projection formula maps. The fact that $j^*(g')_*r^*$ satisfies the projection formula therefore implies that so does $g_*$. Similarly, the fact that $j^*(g')_*r^*$ preserves colimits implies that so does $g_*$.
\end{proof}

We next discuss some examples of $\cC$-adjointable spaces under various assumptions on $\cC$. One source of such examples is the theory of ambidexterity developed by Hopkins and Lurie \cite{HLAmbiKn}:

\begin{example}
\label{ex:AdjointableFromAmbi}
    If a space $A$ is $\cC$-ambidextrous in the sense of \cite[Construction~4.1.8]{HLAmbiKn}, then it is in particular $\cC$-adjointable: colimit-preservation of $A_*$ follows from the fact that it is a left adjoint of $A^*$ \cite[Proposition~4.3.9]{HLAmbiKn}, while the projection formula follows from \cite[Proposition~3.3.1]{TeleAmbi}. In particular, if $\cC$ is $m$-semiadditive, then by definition every $m$-finite space is $\Cc$-ambidextrous and hence $\cC$-adjointable. This yields the following special cases:
    \begin{enumerate}
        \item[(-2)] For every $\cC$ the point $A = \pt$ is $\cC$-adjointable.
        \item[(-1)] For every pointed $\cC$, the empty space $A=\es$ is  $\cC$-adjointable.
        \item[(0)] For every semiadditive $\cC$, every finite discrete set $A$ is $\cC$-adjointable.
        \item[(1)] For every $1$-semiadditive $\cC$, and for every finite group $G$, the classifying space $A=BG$ is $\cC$-adjointable. 
        \item[($\infty$)] For every $\infty$-semiadditive $\cC$, every $\pi$-finite space $A$ is $\cC$-adjointable.
    \end{enumerate}
    For such spaces $A$, it follows that also all iterated loop spaces $\Omega^k(A,a)$ for all $k \in \NN$ and all basepoints $a \in A$ are $\cC$-adjointable, as they are also $\cC$-ambidextrous. Conversely, it follows directly from \cite[Proposition~4.3.9]{HLAmbiKn} that if $A$ is an $n$-truncated space such that both $A$ and all its iterated loop spaces $\Omega(A,a), \dots, \Omega^n(A,a)$ are $\cC$-adjointable, then $A$ is $\cC$-ambidextrous.
\end{example}

A space can be $\cC$-adjointable without being $\cC$-ambidextrous. For example, a space can be $\cC$-adjointable for trivial reasons:

\begin{example}
\label{ex:TwistedAmbiIn1Category}
    If $\cC$ is an $n$-category (i.e., its mapping spaces are homotopically $(n-1)$-truncated), then every $n$-connected space $A$ is $\cC$-adjointable, because the functor $A^*\colon \cC \to \cC^A$ is an equivalence, and hence in particular an internal left adjoint. In particular, if $\cC$ happens to be an ordinary category, then every simply-connected space is $\cC$-adjointable. Similarly, if $\cC$ happens to be a poset, then any connected space is $\cC$-adjointable.
\end{example}

There are however also non-trivial examples of $\cC$-adjointable spaces which are not $\cC$-ambidextrous. 

\begin{example}\label{ex:CompactSpacesAdjointable}
    If $\cC$ is stable, then every compact space is $\cC$-adjointable. To see this, we may reduce to the case $\cC = \Sp$ by base changing along the unique map $\Sp \to \cC$ in $\CAlg(\PrL)$, and to the case of a finite space $A$ by \Cref{lem:TwistedAmbiClosure}. In this case, the functor $A_*\colon \Sp[A] \to \Sp$ is a finite limit and thus commutes with colimits by stability. It follows from \Cref{lem:CLinearityAutomatic} that $A$ is $\cC$-adjointable.
\end{example}

More generally, the $\cC$-adjointable spaces are closed under pushouts when $\cC$ is stable.

\begin{lem}
\label{lem:Stable_Ambi_Pushouts}
Assume $\Cc$ is stable. Then for every space $B$, the collection of $\cC$-adjointable maps $A \to B$ is closed under finite colimits in $\Spc_{/B}$.
\end{lem}
\begin{proof}
As $\cC$ is pointed, the map $\emptyset \to B$ is $\cC$-adjointable. It will thus suffice to show that the $\cC$-adjointable morphisms are closed under pushouts in $\Spc_{/B}$. To this end, consider a pushout square of spaces
\[
\begin{tikzcd}
    A_0 \rar{\alpha} \dar[swap]{\beta} \drar[pushout] & A_1 \dar{j_1} \\
    A_2 \rar{j_2} & A
\end{tikzcd}
\]
Let $f\colon A \to B$ be a map of spaces, and define $f_i = f \circ j_i\colon A_i \to B$ for $j=0,1,2$, where $j_0 = j_2 \circ \beta \simeq j_1 \circ \alpha$. Assume that $f_1$, $f_2$ and $f_0$ are $\cC$-adjointable. We want to prove that the functor $f^*\colon\cC[B]\to\cC[A]$ is an internal left adjoint, or equivalently, by \Cref{cor:Free_VS_Cofree_CLinear_Category}, that the functor $f^*\colon \Fun(B,\cC) \to \Fun(A,\cC)$ is an internal left adjoint.

Let $I= \Lambda^2_0$ denote the walking cospan. The cospan $A_2\xleftarrow{\beta} A_0\xrightarrow{\alpha} A_1$ defines a functor $I \to \Spc_{/B} \simeq \Fun(B,\Spc)$. We let $p\colon E \to I\catop \times B$ denote the unstraightening of the associated functor $I \times B \to \Spc \subseteq \Cat_{\infty}$; note that the pullback of $p$ along the inclusion $\{i\} \times B\to I\catop\times B$ is equivalent to the map $f\colon A_i \to \{i\}\times B$ for $i=0,1,2$. The above pushout diagram determines a functor $g \colon E\to A$ which exhibits $A$ as the groupoidification of $E$ (see \cite[Corollary 3.3.4.6]{htt}). As $g\colon E\to A$ is a localization, the restriction functor $g^*\colon \Fun(A,\cC)\to \Fun(E,\cC)$ is fully faithful, so that the counit $g_!g^*\to \id_{\cC^A}$ is an equivalence. Letting $f' := f \circ g\colon E \to B$, it follows that the functor $f^*\colon \Fun(B,\cC)\to \Fun(A,\cC)$ is $\cC$-linearly equivalent to the following composite:
\[
    \Fun(B,\cC) \xrightarrow{{f'}^*} \Fun(E,\cC)\xrightarrow{g_!} \Fun(A,\cC).
\]
As the functor $g_!$ is an internal left adjoint, it will suffice to prove the same for ${f'}^*$. This functor can in turn be $\cC$-linearly decomposed as the following composite:
\[
    \Fun(B,\cC) \xrightarrow{\pr_B^*} \Fun(I\catop \times B, \Cc) \xrightarrow{p^*} \Fun(E,\cC),
\]
where the first functor is given by precomposition with the projection $\pr_B\colon I\catop \times B \to B$. 

The functor $\pr_B^*$ is obtained by tensoring the functor $(I\catop)^*\colon \Sp \to \Fun(I\catop,\Sp)$ with $\Fun(B,\Cc)$ in $\PrL$. Since finite limits in $\Sp$ commute with colimits, the functor $(I\catop)^*\colon \Sp \to \Fun(I\catop,\Sp)$ is an internal left adjoint in $\PrL_{\st}$, and thus $\pr_B^*\colon \Fun(B,\cC) \to \Fun(I\catop \times B, \Cc)$ is an internal left adjoint in $\Mod_{\Cc}$. 

It thus remains to show that $p^*$ is an internal left adjoint, i.e.\ that its right adjoint $p_*\colon \Fun(E,\cC) \to \Fun(I\catop,\cC)$ in $\widehat{\Cat}_{\infty}$ preserves colimits and satisfies the projection formula. The proof is similar to that of \Cref{lem:FiberwiseAdjointableImpliesAdjointable}.
The three evaluation functors $\ev_i\colon\Fun(I\catop \times B,\cC)\to \Fun(B,\cC)$ for $i \in I\catop$ are jointly conservative, and since projection formulas are compatible with $\cC$-linear functors (cf.\ \cite[Lemma 2.2.4]{TeleAmbi}) it suffices to check that each of the three composites
\[
    \Fun(E,\cC)\xrightarrow{p_*}\Fun(I\catop \times B,\cC)\xrightarrow{\ev_i}\Fun(B,\cC)
\]
preserves colimits and satisfies the projection formula. In other words, letting $i \colon B \hookrightarrow I\catop \times B$ denote the inclusion of $B\simeq B \times \{i\}$ into $B \times I^\op$, it suffices to check that for each $i\in I\catop$ the composite $p^*\circ i_!\colon \Fun(B,\cC)\to \Fun(E,\cC)$ is an internal left adjoint. To this end, observe that for each $i \in I\catop$ there is a pullback diagram
\[
\begin{tikzcd}
    A_i \rar{j_i} \dar[swap]{f_i} \drar[pullback] & E \dar{p} \\
    B \rar[hookrightarrow]{i} & I\catop \times B.
\end{tikzcd}
\]
As $p$ is a cartesian fibration, it follows from \cite[Proposition~4.1.2.11, Proposition~4.1.2.15]{htt}\footnote{Beware that in \cite[Proposition~4.1.2.11]{htt} Lurie works with the contravariant model structure, so that what he denotes by $Lf_!$ would be $(f\catop)_!$ in our notation. Since $p\catop$ is a cocartesian fibration, \cite[Proposition~4.1.2.15]{htt} does indeed apply to our situation.}
that the Beck-Chevalley map ${j_i}_! \circ f_i^*  \to p^*\circ i_!$ in $\FunL(\Spc,\Spc^{E})$ is an equivalence. By tensoring with $\cC$, we obtain a similar equivalence of functors $\cC \to \cC^E$. By assumption on $f_i$, the functor $f_i^*\colon \cC^B \to \cC^{A_i}$ is an internal left adjoint, and the same is always true for $(j_i)_!\colon \cC^{A_i} \to \cC^A$. It thus follows that their composite $p^* \circ i_!\colon \cC^B \to \Cc^E$ is an internal left adjoint as well, finishing the proof.
\end{proof}

As it turns out, $\cC$-adjointability becomes ubiquitous when we move one categorical level up. 

\begin{example}\label{ex:ambiPrL}
    \textit{Every} space is $\PrL$-adjointable. For $n$-truncated spaces, this is an instance of \Cref{ex:AdjointableFromAmbi}, using \cite[Example~4.3.11]{HLAmbiKn}. For a general space $A$, the colimit of any diagram $F\colon A \to \PrL$ is computed as the limit of the associated diagram of right adjoints $F^r \colon A^\op \to \PrR$ (see \cite{htt}). Since $A$ is a space, we have a natural isomorphism $F \simeq F^r$ and hence a natural isomorphism of functors $A_! \simeq A_*$. Since $A_!$ preserves colimits, so does $A_*$. Since dualizable presentable $\infty$-categories generate $\PrL$ under colimits, cf.\ \cite[Lemma 7.14]{RagimovSchlank}, the claim follows from \cref{lem:CLinearityAutomatic}.
    
    It follows that, similarly, every space is $\Mod_{\cC}$-adjointable for every $\cC \in \CAlg(\PrL)$.
\end{example}
\begin{rem}\label{rem:kappaambi}
For a fixed space $A$ and for large enough cardinals $\kappa$ (depending on $A$), the inclusion $\PrL_\kappa\into\PrL$ preserves $A$-shaped limits. Therefore, for such $\kappa$, the above example also works in $\PrL_\kappa$.
\end{rem}

\subsubsection{Adjointability and dualizability}
In the remainder of this subsection, we will discuss the close relationship between internal left adjoints in $\Mod_{\cC}$ and dualizable objects in $\cC$. Recall from \Cref{nota:inducedmap} that every object $X \in \Cc^A$ gives rise to a $\Cc$-linear functor that we abusively also denote by $X\colon \Cc[A] \to \Cc$.

\begin{obs}
\label{obs:Int_Left_Adj_Vs_Dbl}
For an object $Y\in\cC$, the morphism $Y\colon \cC \to \cC$ in $\Mod_{\cC}$ is an internal left adjoint if and only if $Y$ is dualizable in $\cC$. Indeed, the internal right adjoint of $Y\colon \cC \to \cC$ is of the form $Y^{\vee}\colon \cC \to \cC$ for some object $Y^{\vee} \in \cC$, and the unit and counit of the adjunction correspond to the evaluation and coevaluation maps forming the duality data. The triangle identities for the adjunction translate into the triangle identities for the duality.
\end{obs}

\begin{cor}
\label{cor:Criterion_Pointwise_Dualizable}
For $X \in \cC^A$, the functor $\cC[A] \oto{X} \cC$ is an internal left adjoint if and only if $X$ is pointwise dualizable.
\end{cor}
\begin{proof}
The composite of the functor $X\colon \cC[A] \to \cC$ with $a_!\colon \cC \to \cC[A]$ corresponds to $a^*X \in \cC$ by \Cref{cor:RestrictionInduction}. The claim thus follows from \Cref{obs:Int_Left_Adj_Vs_Dbl} and \Cref{lem:Criterion_Internal_Left_Adjoint}.
\end{proof}

In what follows, we equip the $\infty$-category $\cC^A$ with the pointwise symmetric monoidal structure. In particular, an object of $\cC^A$ is dualizable if and only if it is pointwise dualizable.

\begin{prop} [{{cf.\ \cite[Proposition 2.5]{CSY2021Cyclotomic}}}]
\label{prop:ColimIsDbl}
Let $g\colon A \to B$ be a map of spaces and assume that $g$ is $\cC$-adjointable. Then the left Kan extension functor $g_!\colon \cC^A \to \cC^B$ preserves dualizable objects. For a dualizable object $X \in \cC^A$, the dual of $g_!(X) \in \cC^B$ is given by $g_*(X^\vee)$.
In particular, if $A$ is a $\cC$-adjointable space and $X \in \cC^A$ is dualizable, then $A_!X$ is dualizable with dual $A_*X^{\vee}$.
\end{prop}
\begin{proof}
Assume that $X \in \cC^A$ is dualizable. By \Cref{cor:Criterion_Pointwise_Dualizable}, the associated $\cC$-linear functor $X\colon \cC[A] \to \cC$ is an internal left adjoint. Since $g$ is $\cC$-adjointable, the $\cC$-linear functor $g^*\colon \cC[B] \to \cC[A]$ is also an internal left adjoint, and thus so is the composite
\[
    \cC[B] \xrightarrow{g^*} \cC[A] \xrightarrow{X} \cC.
\]
By \Cref{cor:RestrictionInduction}, this composite classifies the object $g_!X \in \cC^B$, and thus it follows from another application of \Cref{cor:Criterion_Pointwise_Dualizable} that $g_!X$ is dualizable. The dual of $g_!X$ may now be computed by $(g_!X)^\vee \simeq \hom_{\cC^B}(g_!X, \one) \simeq g_*\hom_{\cC^A}(X,\one) \simeq g_*(X^\vee)$, where $\hom_\cD$ denotes the internal hom in a symmetric monoidal $\infty$-category $\cD$. 
\end{proof}

\subsection{Traces of free \texorpdfstring{$\cC$}{C}-linear \texorpdfstring{$\infty$}{oo}-categories and free loop transfers}\label{subsection:tracesfree}
Let $A$ be a space and consider the free $\cC$-linear $\infty$-category $\cC[A]$. By combining \Cref{cor:BivariantTheoryGivesDualizability} and \Cref{prop:FreeCofreeCLinearCategories}(\ref{it:FreeCLinearCategoryBivariant}) we get that $\cC[A]$ is a dualizable object in $\Mod_{\cC}$, and that its dimension in $\Mod_{\cC}$ is given by the composite
\[
    \cC \xrightarrow{(LA)^*} \cC[LA] \xrightarrow{(LA)_!} \cC,
\]
where $LA = \Map(S^1,A) \in \Spc$ is the free loop space of $A$. By evaluating at the monoidal unit $\unit \in \cC$, we thus obtain an equivalence
\[
    \TrMod_{\cC}(\cC[A]) \simeq \one[LA].
\]
When $f\colon A \to B$ is a map of spaces, applying the $\cC$-linear trace functor $\Tr_{\cC}$ to the internal left adjoint $f_!\colon \cC[A] \to \cC[B]$ gives the map $Lf\colon \unit[LA] \to \unit[LB]$, see \Cref{cor:THHFreeLoopsOnMorphisms} below. In case the map $f$ is $\cC$-adjointable, we also get a map in the opposite direction.

\begin{defn}
\label{def:Free_Loop_Space_Transfers}
Let $f\colon A \to B$ be a $\cC$-adjointable map of spaces, and consider the $\cC$-linear functor $f^*\colon \cC[B] \to \cC[A]$, which by assumption is an internal left adjoint. The \textit{($\cC$-linear) free loop transfer} of $f$ is the map
\[
    \TrMod_{\cC}(f^*)\colon \unit[LB] \to \unit[LA]
\]
obtained by applying the $\cC$-linear trace functor to $f^*$.
\end{defn}

The goal of the remainder of this section is to show that the equivalence $\TrMod_{\cC}(\cC[A]) \simeq \one[LA]$ is functorial in $A$, see \Cref{thm:THHFreeLoops} below. Since we have already solved a similar coherence problem for the computation of dimensions in the $(\infty,2)$-category $\Span(\Spc)$ in \Cref{subsec:DimensionsInCorr}, the main ingredient will be a good understanding of the functor $\Cc[-]\colon \Span(\Spc) \to \Mod_{\Cc}$ obtained from the bivariant theory $\Cc[-]\colon \Spc \to \Mod_{\Cc}$ of \Cref{prop:FreeCofreeCLinearCategories}. More precisely, we will show that after applying $\Omega$ this functor induces the functor $\unit[-]\colon \Spc \to \Cc$.

We start with some preliminary results.

\begin{lem}
\label{lem:colim_in_func_2cat}
Let $\cE$ be a small $(\infty,2)$-category. The $\infty$-category $\iota_1\Fun(\cE,\PrL)$ is cocomplete, and colimits therein are preserved under restriction along any $2$-functor $\cE'\to \cE$. In particular, colimits are computed pointwise.
\end{lem}
\begin{proof}
For every $\infty$-category $I$, the colimit adjunction $(\PrL)^I \rightleftarrows \PrL$ is a $2$-adjunction: both the unit and co-unit transformations are transformations of $2$-functors, and they induce the appropriate equivalences on mapping \emph{$\infty$-categories} rather than just spaces. In particular, it induces an adjunction $\Fun(\cE,\PrL)^I \simeq \Fun(\cE,(\PrL)^I)\rightleftarrows \Fun(\cE,\PrL)$ where the right adjoint is equivalent to the diagonal functor. The claim follows.
\end{proof}

\begin{lem}
\label{lem:FunL_From_Spaces}
Evaluation at a point induces an equivalence of $(\infty,2)$-categories
\[
    \FunL(\Spc,\PrL)\iso \PrL.
\]
\end{lem}
\begin{proof}
By the Yoneda lemma, it suffices to check that for every small $(\infty,2)$-category $\cE$ the induced map of spaces
\[
    \Map(\cE, \FunL(\Spc,\PrL)) \to \Map(\cE,\PrL)
\]
is an equivalence. By \Cref{lem:colim_in_func_2cat}, the underlying $\infty$-category of $\Fun(\cE,\PrL)$ is again cocomplete, and there is an equivalence 
\[
    \Map^{\rm{L}}(\Spc,\Fun(\cE,\PrL)) \simeq \Map(\cE, \FunL(\Spc,\PrL)).
\]
The composite map
\[
    \Map^{\rm{L}}(\Spc,\Fun(\cE,\PrL)) \to \Map(\cE,\PrL)
\]
is given by evaluation at a point, which is an equivalence by the universal property of the $\infty$-category of spaces (see e.g.\ \cite[Theorem~5.1.5.6]{htt}).
\end{proof}

\begin{defn}
Define $\BivL(\Spc,\PrL) \subseteq \Biv(\Spc,\Cat_{\infty})$ as the sub-2-category consisting of the colimit-preserving bivariant functors $\Spc \to \PrL \subseteq \Cat_{\infty}$.
\end{defn}

\begin{lem}
\label{lem:BivL}
The inclusion
\[
    \BivL(\Spc,\PrL) \hookrightarrow \FunL(\Spc,\PrL)
\]
is an equivalence of $(\infty,2)$-categories.
\end{lem}
\begin{proof}
Note that $\BivL(\Spc,\PrL)$ is a locally full sub-2-category of $\FunL(\Spc,\PrL)$, meaning that the induced maps on mapping categories are fully faithful. It will thus suffice to show the following two claims:
\begin{enumerate}[(a)]
    \item Any object in $\FunL(\Spc,\PrL)$ is a bivariant theory;
    \item Any morphism in $\FunL(\Spc,\PrL)$ is a morphism of bivariant theories.
\end{enumerate}
For claim (a), note that a colimit preserving functor $F\colon \Spc \to \PrL$ is necessarily of the form $A \mapsto \cD[A]$ for $\cD = F(\pt) \in \PrL$, and thus claim (a) follows from \Cref{prop:FreeCofreeCLinearCategories}(\ref{it:CofreeCLinearCategoryBivariant}). For claim (b), let $\cD_0$ and $\cD_1$ be presentable $\infty$-categories, and let $G\colon \cD_0 \to \cD_1$ be a colimit-preserving functor. We have to show that the associated morphism $\cD_0[-] \to \cD_1[-]$ in $\FunL(\Spc,\PrL)$ is a morphism of bivariant theories, i.e.\ that for any map $f\colon A\to B$ of spaces the following commuting square is vertically right adjointable: 
\[\begin{tikzcd}
	{\cD_0[A]} 
	& {\cD_1[A]} \\
	{\cD_0[B]} & {\cD_1[B].}
	\arrow[from=1-1, to=1-2]
	\arrow["{f_!}"', from=1-1, to=2-1]
	\arrow[from=2-1, to=2-2]
	\arrow["{f_!}", from=1-2, to=2-2]
\end{tikzcd}\]
Under the identification $\cD[A]\simeq \cD^A$, the Beck-Chevalley transformation of this square is the natural isomorphism rendering the following square commutative:
\[\begin{tikzcd}
	{\cD_0^A} & {\cD_1^A} \\
	{\cD_0^B} & {\cD_1^B.}
	\arrow[from=1-1, to=1-2]
	\arrow["{f^*}", from=2-1, to=1-1]
	\arrow[from=2-1, to=2-2]
	\arrow["{f^*}"', from=2-2, to=1-2]
\end{tikzcd}\]
This proves that claim (b) holds, which finishes the proof that $\BivL(\Spc,\PrL)= \FunL(\Spc,\PrL)$.
\end{proof}

The functor $\Spc[-]\colon \Spc \to \PrL$ is a symmetric monoidal bivariant theory by \Cref{prop:FreeCofreeCLinearCategories}(\ref{it:FreeCLinearCategoryBivariant}), and hence by the universal property of the $(\infty,2)$-category $\Span(\Spc)$ uniquely extends to a symmetric monoidal 2-functor $\Spc[-]\colon \Span(\Spc) \to \PrL$.

\begin{prop}
\label{prop:CorrVsPrL}
The 2-functor $\Spc[-]\colon \Span(\Spc) \to \PrL$ is 2-fully-faithful, in the sense that it induces equivalences of mapping $\infty$-categories
\[
    \Hom_{\Span(\Spc)}(A,B) \iso \FunL(\Spc[A],\Spc[B])
\]
for all spaces $A,B \in \Spc$. 
In particular, it induces an equivalence $\Omega \Span(\Spc) \iso \Omega \PrL$.
\end{prop}

\begin{proof}
Consider the functor 
\[
       \Spc \to \Biv(\Spc,\Cat_{\infty})\catop, \qquad A  \mapsto (B \mapsto \Spc_{/(A \times B)}),
\]
called the `bivariant op-Yoneda functor' in \cite{Macpherson}. Here the functoriality in $B$ is via post-composition, while the functoriality in $A$ is via pullback. By \cite[Corollary~3.7.5, Theorem~4.2.6]{Macpherson}, this is a bivariant theory and the induced 2-functor $\Span(\Spc) \to \Biv(\Spc,\Cat_{\infty})\catop$ is 2-fully faithful. Note that the bivariant op-Yoneda functor factors through the locally full sub-2-category $\BivL(\Spc,\PrL)\catop$ of $\Biv(\Spc,\Cat_{\infty})\catop$, and in particular, the map $\Span(\Spc) \to \BivL(\Spc,\PrL)\catop$ is also 2-fully-faithful. By \Cref{lem:BivL} and \Cref{lem:FunL_From_Spaces}, the latter $(\infty,2)$-category is equivalent to $(\PrL)\catop$ by evaluation at a point, and it follows that also the composite
\[
    \Span(\Spc) \hookrightarrow \BivL(\Spc,\PrL)\catop \xrightarrow{\simeq }(\PrL)\catop,
\]
is 2-fully faithful. Passing to opposites, we obtain a fully faithful embedding
-\[
    \Span(\Spc) \simeq \Span(\Spc)\catop \hookrightarrow \PrL.
\]
Spelling out what this composite is and using the natural equivalences $\Spc[A] \simeq \Spc^{A} \simeq \Spc_{/A}$, one observes that it is the functor $\Spc[-]\colon \Span(\Spc) \to \PrL$, which finishes the proof.
\end{proof}

We now have the ingredients at hand to prove a coherent version of the equivalence $\TrMod_{\cC}(\cC[A]) \simeq \one[LA]$.
\begin{thm}
\label{thm:THHFreeLoops}
There is an equivalence
\begin{align*}
    \TrMod_{\cC}(\cC[A]) \simeq \one[LA] \qin \cC,
\end{align*}
natural in the space $A\in\Spc$, between the $\cC$-linear trace $\TrMod_{\cC}(\cC[A])$ of $\cC[A]$ and the tensoring $\one[LA]$ of the monoidal unit $\one \in \cC$ by the free loop space $LA \in \Spc$. 
\end{thm}
\begin{proof}
Since the $\cC$-linear trace functor $\TrMod_{\Cc}$ is natural in the variable $\cC \in \CAlg(\PrL)$, the symmetric monoidal colimit-preserving functor $\unit[-]\colon \Spc \to \cC$ gives rise to a commutative diagram
\[
\begin{tikzcd}
    (\PrL)^{\dbl} \dar[swap]{\TrMod_{\Spc}} \rar{\cC \tensor -} & \Mod_{\cC}^{\dbl} \dar{\TrMod_{\cC}} \\
    \Spc \rar{\unit[-]} & \cC.
\end{tikzcd}
\]
As $\cC[A] \simeq \cC \tensor \Spc[A]$, it thus suffices to produce a natural equivalence of spaces
\[
    \TrMod_{\Spc}(\Spc[A]) \simeq LA \qin \Spc.
\]
Consider the functor $\Spc[-]\colon \Spc \to \PrL$. By part (\ref{it:FreeCLinearCategoryBivariant}) of \Cref{prop:FreeCofreeCLinearCategories}, this functor is a symmetric monoidal bivariant theory and in particular it extends uniquely to a symmetric monoidal 2-functor $\Spc[-]\colon \Span(\Spc) \to \PrL$. We may now consider the following commutative diagram:
\[
\begin{tikzcd}
    \Spc \rar{h_{\Spc}} \dar[swap]{L} & \Span(\Spc)^{\dbl} \rar{\Spc[-]} \dar{\dim} & (\PrL)^{\dbl} \dar{\dim} \drar[bend left]{\TrMod_{\Spc}} \\
    \Spc \rar["i_{\Spc}", "\simeq"'] & \Omega \Span(\Spc) \rar["{\Omega \Spc[-]}"] & \Omega \PrL \rar[swap]{\simeq} & \Spc
\end{tikzcd}
\]
The left square commutes by \Cref{thm:DimensionInCorr}. The middle square commutes by functoriality of the dimension functor, applied to the symmetric monoidal 2-functor $\Spc[-]\colon \Span(\Spc) \to \PrL$. The right triangle commutes by definition. The functor $i_{\Spc}\colon \Spc \to \Omega \Span(\Spc)$ from \Cref{prop:OmegaCorr} is an equivalence, and by \Cref{prop:CorrVsPrL} below also the functor $\Omega \Spc[-]\colon \Omega \Span(\Spc) \to \Omega \PrL$ is an equivalence. It follows that the bottom composite $\Spc \to \Spc$ is an equivalence. In particular, it is colimit preserving and hence given by $X\times (-)$ for some space $X$. Since the only space for which this functor is invertible is $X=\pt$, we deduce that the bottom composite is homotopic to the identity on $\Spc$. Comparing the two outer paths in the diagram then gives the desired result.
\end{proof}

\Cref{thm:THHFreeLoops} has the following immediate consequences:

\begin{cor}
\label{cor:THHFreeLoopsOnMorphisms}
Let $f\colon A \to B$ be a map of spaces. Then applying the $\cC$-linear trace functor $\Tr_{\cC}$ to the internal left adjoint $f_!\colon \cC[A] \to \cC[B]$ gives the map $Lf\colon \unit[LA] \to \unit[LB]$:
\[
    \Tr_{\cC}(f_!) \simeq Lf \qin \Map_{\cC}(\unit[LA],\unit[LB]).
\]
\end{cor}

\begin{cor}
\label{THHOfCotensor}
For $(\cD,F) \in \Mod^{\trl}_{\cC}$ and $A \in \Spc$, there is a natural equivalence
\begin{align*}
    \TrMod_{\cC}((\cD,F)[A]) \simeq \TrMod_{\cC}(\cD,F)[LA] \qin \cC.
\end{align*}
\end{cor}
\begin{proof}
By symmetric monoidality of $\TrMod_{\cC}$ and the equivalence $(\cD,F)[A] \simeq (\cD,F) \otimes_{\cC} \cC[A]$ in $\Mod^{\trl}_{\cC}$, this follows immediately from the natural equivalence $\TrMod_{\cC}(\cC[A]) \simeq \one[LA]$ of \Cref{thm:THHFreeLoops}.
\end{proof}

\begin{cor}
\label{cor:constantLoopsViaTHH}
For any space $A \in \Spc$, the composite
\begin{align*}
    A \simeq \Map_{\Spc}(\pt,A) \xrightarrow{\cC[-]} \Map_{\Mod_{\cC}^{\dbl}}(\cC,\cC[A]) \xrightarrow{\TrMod_{\cC}} \Map_{\cC}(\one,\one[LA])
\end{align*}
is the mate of the map $c_A\colon \one[A] \to \one[LA]$ induced by the inclusion of constant loops $A$ into $LA$.
\end{cor}

\begin{proof}
By \Cref{thm:THHFreeLoops} this composite is equivalent to the composite
\[
    A \simeq \Map_{\Spc}(\pt,A) \xrightarrow{L} \Map_{\Spc}(\pt,LA) \xrightarrow{\unit[-]} \Map_{\cC}(\one,\one[LA]).
\]
The composite of the first two maps is the composition of the map $c\colon A \to LA$ and the equivalence $LA \simeq \Map_{\Spc}(\pt,LA)$. From this, the description given in the statement of the corollary follows directly.
\end{proof}

\subsubsection{Traces of space-indexed families of maps}
\label{subsec:THHFamilies}
The functoriality of the $\cC$-linear trace functor $\TrMod_{\Cc}$ is somewhat mysterious: although we gave an explicit description of what it does on objects and on morphisms, the higher coherences are not explicit due to the fact that $\TrMod_{\Cc}$ is defined by reduction to a universal example. In particular, given a family $\cD_{\bullet}\colon A \to \Mod^{\trl}_{\cC}$ indexed by some space $A$, the description we gave for $\TrMod_\cC$ does not immediately provide a full description of the composite functor $\TrMod_{\cC} \circ \cD_{\bullet}\colon A \to \cC$. The goal of this subsection is to show that the situation improves when we try to compute the $\cC$-linear traces of an $A$-indexed family of maps in $\Mod^{\trl}_{\cC}$ with \emph{fixed} source $(\cD,F)$ and target $(\cE,G)$. More concretely, we have the following lemma:

\begin{lem}
\label{lem:THHOfFamiliesOfDiagrams}
For $(\cD,F), (\cE,G) \in \Mod^{\trl}_{\cC}$ and $A \in \Spc$, the following diagram commutes:
\begin{equation*}
\begin{tikzcd}
    \Map_{\Mod^{\trl}_{\cC}}((\cD,F)[A],(\cE,G)) \dar{\TrMod_{\cC}} \rar{\simeq} & \Map(A,\Map_{\Mod^{\trl}_{\cC}}((\cD,F),(\cE,G))) \ar{dd}{\Map(A,\TrMod_{\cC})} \\
    \Map_{\cC}(\TrMod_{\cC}(\cD,F)[LA],\TrMod_{\cC}(\cE,G)) \dar{\simeq} \\
    \Map(LA,\Map_{\cC}(\TrMod_{\cC}(\cD,F),\TrMod_{\cC}(\cE,G)) \rar{- \circ c_A} & \Map(A,\Map_{\cC}(\TrMod_{\cC}(\cD,F),\TrMod_{\cC}(\cE,G)),
\end{tikzcd}
\end{equation*}
where $c\colon A \to LA$ is the inclusion of constant loops.
\end{lem}
\begin{proof}
By the Yoneda lemma, it suffices to prove this in the universal case, i.e.\ when $(\cE,G) = (\cD,F)[A]$ and only at the identity on $(\cD,F)[A]$. Expanding definitions, we thus need to show that the composite
\begin{align*}
    A \simeq \Map(\pt,A) \xrightarrow{(\cD,F)[-]} \Map_{\Mod^{\trl}_{\cC}}((\cD,F), (\cD,F)[A]) \xrightarrow{\TrMod_{\cC}} \Map_{\cC}(\TrMod_{\cC}(\cD,F),\TrMod_{\cC}((\cD,F)[A]))
\end{align*}
is the mate of the map
\begin{align*}
    \TrMod_{\cC}(\cD,F)[A] \xrightarrow{c_A} \TrMod_{\cC}(\cD,F)[LA] \simeq \TrMod_{\cC}((\cD,F)[A]).
\end{align*}
This follows directly from \Cref{cor:constantLoopsViaTHH}, using the equivalence $(\cD,F)[A] \simeq (\cD,F) \tensor_{\cC} \cC[A]$ in $\Mod_{\cC}^{\trl}$ and the symmetric monoidality of the functor $\TrMod_{\cC}$.
\end{proof}

The content of \Cref{lem:THHOfFamiliesOfDiagrams} may more informally be described as follows. Let $(H,\alpha)\colon (\cD,F)[A] \to (\cE,G)$ be a morphism in $\Mod^{\trl}_{\cC}$, and define for every $a \in A$ the map $(H_a,\alpha_a)\colon (\cD,F) \to (\cE,G)$ in $\Mod_{\cC}^{\trl}$ by precomposition with $a_!\colon (\cD,F) \to (\cD,F)[A]$:
\begin{equation*}
\begin{tikzcd}
    \cD \dar[swap]{F} \rar{H_a} & \cE \dar{G} \\
    \cD \urar[shorten <=6pt, shorten >=6pt, Rightarrow, "\alpha_a"] \rar{H_a} & \cE.
\end{tikzcd}
\qquad
=
\qquad
\begin{tikzcd}
    \cD \dar[swap]{F} \rar{a_!} & \cD[A] \dar["{F[A]}"{description}] \rar{H} & \cE \dar{G} \\
    \cD \urar[shorten <=8pt, shorten >=8pt, equal] \rar{a_!} & \cD[A] \urar[shorten <=5pt, shorten >=5pt,Rightarrow, "\alpha"] \rar{H} & \cE
\end{tikzcd}
\end{equation*}
By applying $\TrMod_{\cC}$ to the maps $(H_a,\alpha_a)$, we get an $A$-indexed family of maps $\TrMod_{\cC}(\cD,F) \to \TrMod_{\cC}(\cE,G)$ in $\cC$, or equivalently a map $\TrMod_{\cC}(\cD,F)[A] \to \TrMod_{\cC}(\cE,G)$. \Cref{lem:THHOfFamiliesOfDiagrams} shows that this map is equivalent to the composite
\begin{align*}
    \TrMod_{\cC}(\cD,F)[A] \xrightarrow{c_A} \TrMod_{\cC}(\cD,F)[LA] \overset{\ref{THHOfCotensor}}{\simeq} \TrMod_{\cC}((\cD,F)[A]) \xrightarrow{\TrMod_{\cC}(H,\alpha)} \TrMod_{\cC}(\cE,G),
\end{align*}
where the first map is induced by the inclusion $c_A\colon A \to LA$ of the constant loops. In particular, we may compute the effect of $\TrMod_{\cC}$ on the family of maps $(H_{a},\alpha_a)$ by computing the effect of $\TrMod_{\cC}$ on the single map $(H,\alpha)\colon (\cD,F)[A] \to (\cE,G)$.

\subsection{Hochschild homology as a \texorpdfstring{$\cC$}{C}-linear trace}\label{subsection:HHtrace}

One natural example of a dualizable object in $\Mod_\cC$ is the $\infty$-category $\RMod_R(\cC)$ of right modules in $\Cc$ over an algebra $R \in \Alg(\cC)$. Its $\cC$-linear trace $\TrMod_{\cC}(\RMod_R(\cC)) \in \cC$ identifies with the more classically defined invariant known as the \textit{(topological) Hochschild homology} of $R$. While this folklore identification will not be used in the rest of this article, we shall review it in detail for completeness.\footnote{A certain flavour of it is outlined in \cite[\S 4.5]{HSS} as well.}

\begin{defn}
\label{def:THH}
    Let $R \in \Alg(\cC)$ be an associative algebra in $\cC$, and let $M$ be an $(R,R)$-bimodule. We define the \textit{Hochschild homology of the pair $(R,M)$} as
    \[
        \THH_{\cC}(R,M) := M \tensor_{R \tensor R^{\op}} R \qin \cC,
    \]
    where we regard $R$ as a bimodule over itself via the left and right multiplication. The \textit{Hochschild homology of $R$} is defined as $\THH_{\cC}(R) := \THH_{\cC}(R,R)$.  
\end{defn}

Recall from \cite[Remark~4.8.4.8]{HA} that the $\cC$-linear $\infty$-category $\RMod_R(\cC)$ of right $R$-modules is dualizable in $\Mod_{\cC}$, with dual given by $\LMod_R(\cC)$. For the identification of its $\cC$-linear trace, we will need an explicit description of the evaluation and coevaluation maps for the duality data for $\RMod_{R}(\cC)$ and $\LMod_R(\cC)$. Alternatively, because of the equivalence $\LMod_R(\cC) \simeq \RMod_{R\catop}(\cC)$, we may exhibit explicit duality data between $\RMod_R(\cC)$ and $\RMod_{R\catop}(\cC)$. To this end, recall from \cite[Construction~4.6.3.7]{HA} that the $(R,R)$-bimodule $R$ gives rise to an \textit{evaluation module} $R^e \in \LMod_{R \otimes R\catop}(\cC)$ and a \textit{coevaluation module} $R^c \in \RMod_{R \otimes R\catop}(\cC)$.

\begin{lem}[{\cite{HA}}]
The composite
\[
    \Cc \xrightarrow{R^c} \RMod_{R \otimes R\catop}(\cC) \simeq \RMod_{R}(\cC) \otimes_{\cC} \RMod_{R\catop}(\cC)
\]
is the coevaluation of a duality datum in $\Mod_{\cC}$ between $\RMod_{R}(\cC)$ and $\RMod_{R\catop}(\Cc)$. The evaluation is given by the composite
\[
    \RMod_{R}(\cC) \otimes_{\cC} \RMod_{R\catop}(\cC) \simeq \RMod_{R \otimes R\catop}(\cC) \xrightarrow{- \otimes_{R \otimes R\catop} R^e} \cC.
\]
\end{lem}
Here, the first equivalence is an instance of \cite[Theorem 4.8.5.16(4)]{HA}.
\begin{proof}
The two triangle identities are an immediate consequence of \cite[Proposition~4.6.3.12]{HA}.
\end{proof}

Using the above explicit identification of the duality data of the $\cC$-linear $\infty$-category $\RMod_R(\cC)$, we can now calculate that its $\cC$-linear trace is the Hochschild homology of $R$.

\begin{prop}
\label{prop:THH_As_Categorical_Trace}
    Let $R \in \Alg(\cC)$ be an associative algebra in $\cC$, and let $M$ be an $(R,R)$-bimodule. Then there is an equivalence between the Hochschild homology of the pair $(R,M)$ and the $\cC$-linear trace of the $\cC$-linear endofunctor $- \tensor_R M\colon \RMod_R(\cC) \to \RMod_R(\cC)$:
    \[
        \THH_{\cC}(R,M) \simeq \TrMod_{\cC}(\RMod_R(\cC), - \tensor_R M) \qin \cC.
    \]
    In particular, there is an equivalence
    \[
        \THH_{\cC}(R) \simeq \TrMod_{\cC}(\RMod_R(\cC)) \qin \cC.
    \]
\end{prop}
\begin{proof}
Observe that the second statement is a special case of the first statement by taking $M = R$, so we focus on the first statement. By definition of the trace, we have to compute the composite
\[
    \cC \xrightarrow{\coev} \RMod_R(\cC)\otimes_\cC \RMod_{R\catop}(\cC) \xrightarrow{M \otimes_{\Cc} \id} \RMod_R(\cC)\otimes_\cC \RMod_{R\catop}(\cC) \xrightarrow{\ev} \cC.
\]
Plugging in the evaluation and coevaluation maps described above, we see that this is given by the following composite:
\[
    \cC \xrightarrow{R^c} \RMod_{R \otimes R\catop}(\cC) \xrightarrow{- \otimes_{R \otimes R\catop} (M \boxtimes R\catop)} \RMod_{R \otimes R\catop}(\cC) \xrightarrow{- \otimes_{R \otimes R\catop} R^e} \cC,
\]
where we abusively write $M$ for the functor given by tensoring with $M$, and where $\boxtimes$ denotes the external tensor product defined in \cite[Notation~4.6.3.3]{HA}. The composite of the first two maps is classified by the right $(R \otimes R\catop)$-module $R^c \otimes_{R \otimes R\catop} (M \boxtimes R\catop)$, which is by definition precisely the image of $M$ under the equivalence  ${}_R\!\BMod_R(\cC) \simeq \RMod_{R \otimes R\catop}(\cC)$ of \cite[Construction~4.6.3.9]{HA}. If we abusively denote this right $(R \otimes R\catop)$-module again by $M$ and similarly denote the left $(R \otimes R\catop)$-module $R^e$ by $R$, we see that the above composite is given by the object $M \otimes_{R \otimes R\catop} R= \THH_{\Cc}(R,M)$, as desired.
\end{proof}

\begin{rem}
The equivalence of \Cref{prop:THH_As_Categorical_Trace} above, expressing Hochschild homology as a $\cC$-linear trace, is highly expected to be natural in the triple $(\cC,R,M)$. We emphasize that the above proof does not prove this stronger statement: as in the case of traces in $(\infty,2)$-categories of spans, discussed in \Cref{subsec:DimensionsInCorrCoherence}, the subtlety lies in making the above explicit duality data for $\RMod_R(\cC)$ suitably natural in $\cC$ and $R$. We will, however, not need this natural enhancement.
\end{rem}

\subsection{Free \texorpdfstring{$\cC$}{C}-linear \texorpdfstring{$\infty$}{oo}-categories as module categories}
\label{subsec:Free_CLinear_Cats_As_Module_Cats}

For a space $A$, we have seen that the free $\cC$-linear $\infty$-category $\cC[A]$ can be identified with the cofree $\cC$-linear $\infty$-category $\cC^A$. In the case where $A$ is a pointed connected space, $\cC[A]$ can also be identified with $\RMod_{\one[\Omega A]}(\cC)$, the $\cC$-linear $\infty$-category of modules in $\cC$ over the group algebra $\one[\Omega A]$. We shall explain this in some detail to flesh out the rather strong naturality properties of this identification that will be required later on. 

Recall from \cite[\S 4.8.5]{HA} the functor
\[
    \RMod_{(-)}(\cC)\colon \Alg(\cC) \longhookrightarrow 
    (\Mod_{\cC})_{\cC/},
\]
which assigns to an algebra $R$ in $\cC$ the $\cC$-linear $\infty$-category $\RMod_R(\cC)$ of right $R$-modules in $\cC$, pointed by $R$. This functor is fully faithful and admits a right adjoint (\cite[Theorem~4.8.5.11, Remark~4.8.5.12]{HA})
\[
    \End^{\cC}(\one_{(-)})\colon 
    (\Mod_{\cC})_{\cC/} \too
    \Alg(\cC),
\]
which takes a $\cC$-linear $\infty$-category $\cD$ pointed by some $\one_\cD \in \cD$ to the $\cC$-object of endomorphisms of $\one_\cD$, viewed as an $\EE_1$-algebra in $\cC$.\footnote{The notation $\one_\cD$ is \emph{not} meant to indicate that $\cD$ has a monoidal structure. Rather, it is an $\EE_0$-algebra in $\Mod_\cC$, and so has a distinguished object which we denote $\one_{\cD}$.} 
\begin{lem}[{cf.\ \cite[Proposition~A.4]{CyclyShift}}]
\label{Cat_Grp_Alg_Mod}
    For a pointed connected space $A$, we have an equivalence
    \[
        \cC[A] \simeq \RMod_{\one[\Omega A]}(\cC)
        \qin (\Mod_\cC)_{\cC/},
    \]
    natural in both $A$ and $\cC$.
\end{lem}
\begin{proof}
    In light of the equivalence $\Spc_*^{\geq 1} \simeq \Grp(\Spc)$ between pointed connected spaces and group objects in $\Spc$, we may equivalently produce a natural equivalence
    \[
        \cC[BG] \simeq \RMod_{\unit[G]}(\cC) \qin (\Mod_\cC)_{\cC/}
    \]
    for $G \in \Grp(\Spc)$. The canonical functors $\unit[-]\colon \Spc[BG] \to \cC[BG]$ and $\unit[-]\colon \RMod_{G}(\Spc) \to \RMod_{\one[G]}(\cC)$ in $\PrL_{\Spc/}$ induce natural equivalences
    \[
        \cC \otimes \Spc[BG] \iso \cC[BG] \qquad \text{ and } \qquad \cC \otimes \RMod_{G}(\Spc) \iso \RMod_{\one[G]}(\cC)
    \]
    in $(\Mod_\cC)_{\cC/}$; the first one since $\cC \otimes -$ preserves colimits and the second one by \cite[Theorem~4.8.4.6]{HA}. It thus remains to show that $\Spc[BG] \simeq \RMod_{G}(\Spc)$, or equivalently that $\Spc^{BG} \simeq \LMod_{G}(\Spc)$ by passing to duals. Recalling from \cite[Notation~4.2.2.5]{HA} the $\infty$-category $\LMon_{G}(\Spc)$ of spaces with a left action of $G$, we then have the following sequence of natural equivalences:
    \[
        \Spc^{BG} \iso \Spc_{/BG} \iso \LMon_{G}(\Spc) \iso \LMod_{G}(\Spc);
    \]
    here the equivalence on the left is straightening-unstraightening, the middle equivalence is \cite[Proposition~3.2.76]{Sati_Schreiber} (see also \cite[4.1]{NSS_Principle_Bundles}) and the equivalence on the right is \cite[Proposition~4.2.2.9]{HA}. 
\end{proof}

\begin{cor}
    The following square naturally commutes:
\[\begin{tikzcd}
	{(\Mod_\cC)_{\cC/}} &&& {\Alg(\cC)} \\
	{\Spc_*} &&& {\Grp(\Spc).}
	\arrow["(-)^\simeq"', from=1-1, to=2-1]
	\arrow["{(-)^\times}", from=1-4, to=2-4]
	\arrow["{\End^\cC(\one_{(-)})}", from=1-1, to=1-4]
	\arrow["\Omega", from=2-1, to=2-4]
\end{tikzcd}\]
Here the right vertical functor sends an algebra $R \in \Alg(\cC)$ to the subgroup $R^{\times}$ of invertibles in the $\EE_1$-monoid $\Map_{\cC}(\unit_{\Cc},R)$.
\end{cor}
\begin{proof}
    Passing to left adjoints, we may equivalently produce a commutative square of the form
    \[\begin{tikzcd}
    	{(\Mod_\cC)_{\cC/}} &&& {\Alg(\cC)} \\
    	{\Spc_*} &&& {\Grp(\Spc).}
    	\arrow["B"', hook', from=2-4, to=2-1]
    	\arrow["{\RMod_{(-)}}"', hook', from=1-4, to=1-1]
    	\arrow["{\one[-]}"', from=2-4, to=1-4]
    	\arrow["{\cC[-]}", from=2-1, to=1-1]
    \end{tikzcd}\]
    This is the content of \Cref{Cat_Grp_Alg_Mod}.
\end{proof}

Given an algebra $R\in \Alg(\cC)$, a pointed map $\zeta \colon A \to \RMod_R(\cC)$ corresponds by adjunction to a pointed $\cC$-linear functor $\zeta_\cC \colon \cC[A] \to \RMod_R(\cC)$, hence by the lemma to a pointed $\cC$-linear functor $\RMod_{\one[\Omega A]}(\cC) \to \RMod_R(\cC)$. This can be described in terms of the group map $\Omega \zeta \colon \Omega A \to \Omega \RMod_R(\cC)$, which under the equivalence
\[
    \Omega \RMod_R(\cC)\simeq \End^{\Cc}(\RMod_R(\cC))^{\times} \simeq R^\times
\]
corresponds to an algebra map $\one[\Omega A] \to R$. 

\begin{lem}\label{Aug_Base_Change}
    The composite
    \[
        \RMod_{\one[\Omega A]}(\cC) \iso 
        \cC[A] \oto{\:\zeta_\cC\:}
        \RMod_R(\cC)
    \]
    is given by extension of scalars along 
    $\one[\Omega A] \to R$.
\end{lem}
\begin{proof}
    Since the functoriality of $\RMod_{(-)}(\cC)$ is given by extension of scalars, this follows from a straightforward diagram chase in the above commutative square of adjunctions.
\end{proof}

\begin{rem}
Let $A$ be a pointed connected space. Combining \Cref{prop:THH_As_Categorical_Trace}, \Cref{Cat_Grp_Alg_Mod} and \Cref{thm:THHFreeLoops}, we get a chain of equivalences
\[
    \THH_{\cC}(\unit[\Omega A]) \simeq \TrMod_{\Cc}(\RMod_{\unit[\Omega A]}(\cC)) \simeq \TrMod_{\Cc}(\cC[A]) \simeq \unit[LA].
\]
The resulting identification of the Hochschild homology of $\unit[\Omega A]$ with the free loop spac $\unit[LA]$ is well-known among experts, and is usually stated and proved in terms of the cyclic bar construction. For example, when $\cC$ is the $\infty$-category of spectra, this is \cite[Corollary IV.3.3]{NS}. For earlier references, one can consult \cite[Proposition 3.7]{BHM}, or \cite[Lemma V.1.3]{Goodwillie} in the integral case. 
\end{rem}

\section{Traces and characters via categorified traces}
\label{sec:GeneralizedCharacters}

Let $\Cc$ be a presentably symmetric monoidal $\infty$-category, taken to be fixed. In this section, we will study the interaction between (generalized) traces in $\Cc$ and categorified traces in the $(\infty,2)$-category $\Mod_{\Cc}$ of $\Cc$-linear $\infty$-categories, studied in the previous section. 

The main ingredient for this interaction is the fact that one can express generalized traces in $\Cc$ in terms of $\Cc$-linear traces, so let us start by explaining how this goes. Consider objects $X$, $Y$ and $Z$ of $\cC$, where $X$ is dualizable, and let $f\colon Z \tensor X \to X \tensor Y$ be a generalized endomorphism. Forming the tensor product with $X$ gives a functor $X\colon \Cc \to \Cc$, which by dualizability of $X$ is an internal left adjoint, see \Cref{obs:Int_Left_Adj_Vs_Dbl}. Similarly we obtain $Y\colon \Cc \to \Cc$ and $Z\colon \Cc \to \Cc$, and the morphism $f$ corresponds to a $\Cc$-linear natural transformation fitting in the following diagram:
\[\begin{tikzcd}
	\cC && \cC \\
	\cC && \cC.
	\arrow["Z"', from=1-1, to=2-1]
	\arrow["X", from=1-1, to=1-3]
	\arrow["X"', from=2-1, to=2-3]
	\arrow["Y", from=1-3, to=2-3]
	\arrow["f"{description},shorten <=7pt, shorten >=7pt, Rightarrow, from=2-1, to=1-3]
\end{tikzcd}\]
We denote by $(X,f)\colon (\cC,Z) \to (\cC,Y)$ the resulting morphism in $\Mod_{\Cc}^{\trl}$. Applying the $\Cc$-linear trace functor to this map provides a morphism of the form
\[
    \TrMod_{\Cc}(X,f)\colon Z \to Y \qin \cC;
\]
here we use the identifications $\TrMod_{\Cc}(\Cc,Z) = Z$ and $\TrMod_{\Cc}(\Cc,Y) = Y$ of \Cref{ex:CLinear_Trace_Scalar}, which we will leave implicit from now on. The next lemma shows that this morphism is given by the generalized trace $\gentr{f}{X}\colon Z \to Y$ of $f$.

\begin{lem}[Trace comparison lemma, pointwise version]
\label{THH_Trace}
Let $X$, $Y$ and $Z$ be objects of $\cC$, where $X$ is dualizable, and let $f\colon Z \otimes X \to X \otimes Y$ be a generalized endomorphism of $X$ in $\cC$.
Then there is an equivalence
\[
    \TrMod_{\Cc}(X,f) \simeq \gentr{f}{X} \qin \Map_{\Cc}(Z,Y).
\]
\end{lem}
\begin{proof}
We compute $\TrMod_{\cC}(X,f)$ using the description of the functoriality of $\TrMod_{\cC}$ from the diagram (\ref{eq:FunctorialityTHH}) on page \pageref{eq:FunctorialityTHH}. Since $\cC$ is the monoidal unit of $\Mod_{\Cc}$, it is in particular self-dual, and we get $\cC \otimes_{\cC} \cC^\vee \simeq \cC$. The right adjoint of $X\colon \Cc \to \Cc$ is $X^{\vee}\colon \Cc \to \Cc$, whose $\Cc$-linear dual is again $X^{\vee}\colon \Cc \to \Cc$. It follows that the map $\TrMod_{\cC}(X,f)\colon Z \to Y$ is given by evaluating the following composite transformation at $\one \in \cC$:
\[\begin{tikzcd}
	&& \cC &&& \cC \\
	\cC &&&&&&& \cC \\
	&& \cC &&& \cC
	\arrow[Rightarrow, bend left = 10, no head, from=2-1, to=1-3]
	\arrow[""{name=0, anchor=center, inner sep=0}, Rightarrow, no head, bend right = 10, from=2-1, to=3-3]
	\arrow["Y"', from=3-3, to=3-6]
	\arrow["{X \otimes X^\vee}", from=1-3, to=3-3]
	\arrow["{X \otimes X^\vee}"', from=1-6, to=3-6]
	\arrow[""{name=1, anchor=center, inner sep=0}, Rightarrow, no head, from=1-6, to=2-8, bend left = 10]
	\arrow[Rightarrow, no head, from=3-6, to=2-8, bend right = 10]
	\arrow["f"{description}, shorten <=24pt, shorten >=24pt, Rightarrow, from=1-6, to=3-3]
	\arrow["Z", from=1-3, to=1-6]
	\arrow["{\coev_{X}}"{description}, shorten <=7pt, shorten >=7pt, Rightarrow, from=1, to=3-6]
	\arrow["{\ev_X}"{description}, shorten <=11pt, shorten >=7pt, Rightarrow, from=1-3, to=0]
\end{tikzcd}\]
This is precisely the definition of the generalized trace $\gentr{f}{X}\colon Z \to Y$ of $f$.
\end{proof}
The goal of this section is to exploit the above interaction to set up a \emph{generalized character theory} to study these generalized traces in families indexed by spaces. Specifically, given an $A$-indexed family of generalized endomorphisms
\[
    f_a\colon X_a\to X_a\otimes Y, \quad a\in A
\]
we will associate a generalized character $\chi_f\colon \one[LA]\to Y$, encoding the generalized traces of twists of $f$ by free loops in $A$. These generalized characters are the main object of study of this section, and of the entire article. 

This section is organized as follows. We start in \Cref{subsec:Generalized_Characters} with the definition of the generalized character $\chi_f\colon \unit[LA] \to Y$, and prove an explicit description of it in terms of generalized traces of twists of $f$. Showing that the various pointwise descriptions we give are in fact natural equivalences is a bit subtle, and will be discussed in \Cref{subsec:Coherence_Trace_Formulas}. In \Cref{subsec:RestrictionInductionCharacters}, we conclude the study of these characters with explicit formulas for induced and restricted characters.

In \Cref{subsec:Additivity_Traces}, we demonstrate another use of categorification to deduce the additivity of generalized traces from exactness properties of the $\Cc$-linear trace functor $\TrMod_{\Cc}$. This final subsection is independent of first three subsections, in that it does not use our generalized character theory.

\subsection{Generalized characters}
\label{subsec:Generalized_Characters}
Let $G$ be a finite group, and let $V$ be a finite-dimensional complex $G$-representation. The \textit{character} $\chi_V\colon G \to \CC$ of $V$ is the function which assigns to every group element $g \in G$ the trace of the action map $g\colon V \to V$. The character of $V$ is a \textit{class function}: for two conjugate group elements $g$ and $g'$, we have $\chi_V(g) = \chi_V(g')$. If we let $G/\conj$ denote the set of conjugacy classes, it is thus natural to regard the character of $V$ as a function
\[
    \chi_V\colon G/\conj \to \CC.
\]

In this subsection, we will generalize the definition of characters in the following three ways:
\begin{enumerate}[(1)]
    \item The category $\Vect_{\CC}$ of complex vector spaces gets replaced by an arbitrary presentably symmetric monoidal $\infty$-category $\cC$, from now on taken to be fixed;
    \item Instead of $G$-equivariant objects, regarded as local systems on the classifying space $BG$, we will work with local systems on an arbitrary space $A$.
    \item Rather than studying the character of an object, we will study the character associated to a generalized endomorphism of an object.
\end{enumerate}

Taking inspiration from \Cref{THH_Trace}, our notion of generalized character map will be defined as the $\Cc$-linear trace of a suitable map in $\Mod_{\Cc}^{\trl}$.

\begin{cons}
Assume given a space $A$, a dualizable object $X \in \cC^A$, and a generalized endomorphism of $X$ in $\cC^A$ of the form $f\colon X \to X \tensor A^*Y$ for some object $Y \in \Cc$. We denote by $(X,f)\colon \cC[A] \to (\cC,Y)$ the map in $\Mod^{\trl}_{\cC}$ presented by the following diagram:
\[\begin{tikzcd}
	{\cC[A]} && \cC \\
	{\cC[A]} && \cC.
	\arrow[Rightarrow, no head, from=1-1, to=2-1]
	\arrow["Y", from=1-3, to=2-3]
	\arrow["X", from=1-1, to=1-3]
	\arrow["X"', from=2-1, to=2-3]
	\arrow["f"{description},shorten <=7pt, shorten >=7pt,  Rightarrow, from=2-1, to=1-3]
\end{tikzcd}\]
Here we regard $X$ as a $\cC$-linear functor $X\colon \cC[A] \to \cC$ using the equivalence $\Fun_{\cC}(\cC[A],\cC) \simeq \cC^A$ of \Cref{lem:MappingOutOfFreeCModule}, and we similarly regard $f$ as a $\cC$-linear transformation. Since $X$ is pointwise dualizable, the functor $X\colon \cC[A] \to \cC$ is an internal left adjoint by \Cref{cor:Criterion_Pointwise_Dualizable}, and thus this diagram is indeed a morphism in $\Mod^{\trl}_{\cC}$ by \Cref{cor:Morphisms_In_Dbl}.
\end{cons}

The generalized character of $f$ is obtained by applying the $\cC$-linear trace functor to the map $(X,f)\colon \Cc[A] \to (\Cc,Y)$.

\begin{defn}
\label{def:Character} 
Let $A$ be a space, $X \in \cC^A$ a local system of dualizable objects, $Y \in \cC$ an object and $f\colon X \to X \tensor A^*Y$ a morphism in $\cC^A$. We define the \textit{generalized character} $\chi_f\colon \unit[LA] \to Y$ of $f$ to be the map
    \[
        \chi_f := \TrMod_{\cC}(X,f)\colon \one[LA] \too Y
        \qin \cC.
    \]
    We will sometimes abuse notation by identifying it with its mate
    \[
        \chi_f \colon LA \too \Map_{\cC}(\one, Y)
        \qin \Spc.
    \]
When $Y = \unit$ and $f=\id_X$, we will write $\chi_X$ for $\chi_f$.
\end{defn}

\begin{example}
When $A = \pt$, the data of $f$ is a generalized endomorphism $f\colon X \to X \tensor Y$ in $\cC$, and the character $\chi_f$ is just the generalized trace $\gentr{f}{X}\colon \unit \to Y$ by \Cref{THH_Trace}.
\end{example}

\begin{example}[Characters in representation theory]
When $A$ is the classifying space $BG$ of a finite group $G$ and $\Cc$ is the (ordinary) category $\Vect_{\CC}$ of complex vector spaces, the data of $X$ is that of a finite-dimensional complex $G$-representation $V$. Since there is an isomorphism of complex vector spaces
\[
    \CC[LBG] \simeq \CC[\pi_0(LBG)] \simeq \CC[G/\conj],
\]
the generalized character map $\chi_V\colon \CC[LBG] \to \CC$ may be identified with a function
\[
    \chi_V\colon G/\conj \to \CC.
\]
We shall see in \Cref{cor:Character_Formula_Id} that this is precisely the ordinary character of the $G$-representation $V$.
\end{example}

The goal of the remainder of this subsection is to give an explicit formula for the generalized character $\chi_f$ in terms of generalized traces. Concretely, we shall show that its value $\chi_f(\gamma)\colon \unit \to Y$ on a free loop $\gamma \in LA$ is given by the following generalized trace:
\[
    \chi_f(\gamma) = \gentr{
        X_a \oto{\gamma} 
        X_a \oto{f_a} 
        X_a \otimes Y}{X_a} \qin \Map_{\cC}(\unit,Y).
\]
Here we let $a := \gamma(\ast)$ denote the basepoint of $\gamma$, and $\gamma\colon X_a \to X_a$ denotes the automorphism of $X_a$ induced by applying the functor $X\colon A \to \Cc$ to the automorphism $\gamma \colon a \iso a$ in $A$. 

The main ingredient in the proof of this character formula is the following lemma:

\begin{lem}[Loop indicator lemma, pointwise version]
\label{Loop_Indicator}
    Consider a free loop $\gamma \in LA$ with basepoint $a := \gamma(\ast)$, and consider the morphism $(a_!,\gamma_!)\colon \Cc \to \Cc[A]$ in $\Mod^{\trl}_{\Cc}$ given by
    \[\begin{tikzcd}
    	\cC && {\cC[A]} \\
    	\cC && {\cC[A],}
    	\arrow[Rightarrow, no head, from=1-1, to=2-1]
    	\arrow[Rightarrow, no head, from=1-3, to=2-3]
    	\arrow["{a_!}", from=1-1, to=1-3]
    	\arrow["{a_!}"', from=2-1, to=2-3]
    	\arrow[shorten <=6pt, shorten >=6pt, "\gamma_!"{description}, Rightarrow, from=2-1, to=1-3]
    \end{tikzcd}\]
    where $\gamma$ is regarded as a self-homotopy of the map $a\colon \pt \to A$. Then the $\Cc$-linear trace $\TrMod_{\Cc}(a_!,\gamma_!)\colon \unit \to \unit[LA]$ is homotopic to the map $\one \oto{\gamma} \one[LA]$ induced from the map $\pt \oto{\gamma} LA$:
    \[
        \TrMod_{\Cc}(a_!,\gamma_!) \simeq \gamma \qin \Map_{\Cc}(\unit,\unit[LA]).
    \]
\end{lem}
In other words, the map $\TrMod_{\cC}(a_!,\gamma_!)\colon \one \to \one[LA]$ can be thought of as the ``Dirac measure'' supported at the loop $\gamma \in LA$. 
\begin{proof}
By \Cref{prop:FreeCofreeCLinearCategories}, the functor $\cC[-]\colon \Spc \to \Mod_\cC$ is a symmetric monoidal bivariant theory, and hence uniquely extends to a symmetric monoidal functor $\Span(\Spc) \to \Mod_\cC$. Since the formation of the categorified trace functor is manifestly natural in symmetric monoidal functors, the claim for $\Mod_\cC$ follows from the analogous statement in the $(\infty,2)$-category $\Span(\Spc)$. It thus suffices to show that applying the trace functor to the morphism 
\[
\begin{tikzcd}
    	{\pt} && {A}\\
    	{\pt} && {A}
    	\arrow[Rightarrow, no head, from=1-1, to=2-1]
    	\arrow[Rightarrow, no head, from=1-3, to=2-3]
    	\arrow["{a}", from=1-1, to=1-3]
    	\arrow["{a}"', from=2-1, to=2-3]
    	\arrow[shorten <=6pt, shorten >=6pt, "\gamma"{description}, Rightarrow, from=2-1, to=1-3]
\end{tikzcd}
\]
in $\Span(\Spc)^{\trl}$ yields the map $\gamma \colon \pt \to LA$ in $\Omega \Span(\Spc) \simeq \Spc$. 

To this end, recall from \Cref{lem:DblCorr} that $A$ is self-dual in $\Span(\Spc)$ with evaluation and coevaluation maps given by 
\[
    A \times A \xleftarrow{\;\Delta\;} A \xrightarrow{\;\pi\;} \pt \qquad \text{and} \qquad 
    \pt \xleftarrow{\;\pi\;} A \xrightarrow{\;\Delta\;}  A \times A. 
\]
We can now compute $\Tr_{\Span(\Spc)}(a;\gamma)$ by substituting these evaluation and coevaluation maps into the diagram (\ref{eq:2FunctorialityTrace}), which in our case assumes the form 
\[
\begin{tikzcd}
	\pt && {\pt} && {\pt} \\
	\\
	{A\times A} && {A} && {A} && {\pt} \\
	\\
	&& {A\times A} && {A\times A} && \pt.
	\arrow["{\coev_A}"', from=1-1, to=3-1]
	\arrow["{\id}", from=1-1, to=1-3]
	\arrow["{(a,\id)^r}"{description}, from=3-1, to=3-3]
	\arrow["{a}"{description}, from=1-3, to=3-3]
	\arrow[""{name=0, anchor=center, inner sep=0}, "{(a,\id)}"{description}, from=3-3, to=5-3]
	\arrow[""{name=1, anchor=center, inner sep=0}, "{\id}"', curve={height=18pt}, from=3-1, to=5-3]
	\arrow["{a^r}"{description}, from=3-5, to=3-7]
	\arrow["{\id}", from=3-7, to=5-7]
	\arrow[""{name=2, anchor=center, inner sep=0}, "{\id}", curve={height=-18pt}, from=1-5, to=3-7]
	\arrow[""{name=3, anchor=center, inner sep=0}, "{a}"{description}, from=1-5, to=3-5]
	\arrow[equal, shorten <=16pt, shorten >=16pt, from=1-3, to=3-1]
	\arrow[equal, shorten <=16pt, shorten >=16pt, from=5-5, to=3-7]
	\arrow["{\id}", from=1-3, to=1-5]
	\arrow["{\ev_A}"', from=5-5, to=5-7]
	\arrow["{(\id,a)}"{description}, from=3-5, to=5-5]
	\arrow["\id"', from=5-3, to=5-5]
	\arrow["{(c,\id)}"{description}, shorten <=9pt, shorten >=9pt, Rightarrow, from=3-3, to=1]
	\arrow["u"{description}, shorten <=9pt, shorten >=9pt, Rightarrow, from=2, to=3-5]
	\arrow["{(\gamma, a)}"{description}, shorten <=19pt, shorten >=19pt, Rightarrow, from=3, to=0]
\end{tikzcd}
\]
Here we identify objects and morphisms in $\Spc$ with their images under the inclusion $h_{\Spc}\colon \Spc \hookrightarrow \Span(\Spc)$, and given a morphism $f\colon A \to B$ of spaces, we denote its right adjoint in $\Span(\Spc)$ by $f^r\colon B \to A$, as in \Cref{lem:AdjointInCorr}. We can now explicitly identify each of the natural transformations in this diagram as a morphism in a suitable over-category of a space: 

\begin{enumerate}
    \item By the construction of the unit $2$-morphism $u\colon \id_\pt \to a^r \circ a$, it corresponds via the equivalence 
    \[
    \Hom_{\Span(\Spc)}(\pt,\pt) \simeq \Spc
    \] 
    to the map of spaces  
    $\pt \to \pt \times_A \pt$ classified by the square 
    \[
    \begin{tikzcd}
\pt	&&\pt   \\
\pt	&&{A.} 
	\arrow[Rightarrow, no head, from=1-1, to=1-3]
	\arrow[Rightarrow, no head, from=1-1, to=2-1]
	\arrow["a"{description}, from=2-1, to=2-3]
    \arrow["a"{description}, from=1-3, to=2-3]
    \arrow[Rightarrow,no head, shorten <=7pt, shorten >=7pt, from=2-1, to=1-3]
\end{tikzcd}
    \]
By definition, $\Omega_a A  = \pt \times_A \pt$, and the above square classifies the constant loop map $c_a \colon \pt \to \Omega_a A$. 

\item The $2$-morphism $\gamma \times a\colon (a,\id)\circ a \to (\id,a)\circ a$ corresponds via the equivalence 
\[
\Hom_{\Span(\Spc)}(\pt,A\times A) \simeq  \Spc_{/A\times A}
\] 
to the automorphism of the map $a\times a \colon \pt \to A\times A$ given by applying $\gamma$ to the first coordinate.  Composing it from the right with the evaluation map $A\times A \to \pt$ in $\Span(\Spc)$ amounts to pulling it back along the diagonal map $\Delta \colon A\to A\times A$ and then applying the forgetful functor $\Spc_{/A} \to \Spc$.
The space $\pt \times_{A\times A} A$ identifies with $\Omega_a A$, and the pullback of $(\gamma ,a)$ along $\Delta \colon A\to A\times A$ corresponds via this identification to the map $\gamma \star (-)\colon \Omega_a A \to \Omega_a A$ of concatenation with $\gamma$.  

\item The $2$-morphism $c\times \id \colon (a,\id)\circ(a,\id)^r \to \id_{A\times A}$ 
corresponds via the equivalence
\[
\Hom_{\Span(\Spc)}(A\times A, A\times A) \simeq \Spc_{/A\times A \times A\times A} 
\]
to the commutative triangle 
\[
\begin{tikzcd} 
A && A\times A \\ 
  & A\times A \times A \times A,
\arrow["{(a,\id_A)}",from=1-1, to=1-3]
\arrow["\Delta_{A\times A}"{description},from=1-3, to=2-2]
\arrow["{(a,\id_A,a,\id_A)}"{description},from=1-1, to=2-2]
\end{tikzcd}
\]
regarded as a morphism in the over-category of $A\times A \times A\times A$. Precomposing it with the evaluation of $A$ and post-composing with the coevaluation of $A$ amounts to pulling it back along the map $\Delta_A\times \Delta_A \colon A\times A \to A\times A \times A\times A$ and then applying the forgetful functor $\Spc_{/A\times A} \to \Spc$. Now, pulling back the morphism $(a,\id_A, a, \id_A)$ along $\Delta_A \times \Delta_A$ gives the space $\Omega_a A$, with a certain map to $A\times A$ that we then forget. Pulling back the map $\Delta_{A\times A}$ along $\Delta_A \times \Delta_A$ gives the space $LA$, again with a map to $A\times A$ that we forget. Finally, via these identifications, the map $(a,\id_A)$ pulls back to the map $\Omega_a A \to LA$ embedding the based loops at $A$ inside the space of all free loops. 
\end{enumerate}
We thus deduce that the trace
$
\Tr_{\Span(\Spc)}(a;\gamma)
$ 
is given by the composition 
\[
\pt \oto{c_a} \Omega_a A \oto{\gamma \star (-)} \Omega_a A \to LA.
\]
But this composition is clearly given by $\gamma\colon \pt \to LA$, finishing the proof.
\end{proof}

We may use \Cref{Loop_Indicator} to prove the description of the character in terms of generalized traces.
\begin{prop}[Character formula, pointwise version]
\label{prop:Character_Formula}
Let $A$ be a space, let $X\in \cC^A$ be a dualizable object, let $Y\in \cC$ be an object and let $f\colon X \to X \tensor A^*Y$ be a generalized endomorphism. For every free loop $\gamma \in LA$ with basepoint $a := \gamma(\ast)$, there is an equivalence
    \[
    \chi_f(\gamma) \simeq \gentr{
        X_a \oto{\gamma} 
        X_a \oto{f_a} 
        X_a \otimes Y}{X_a} \qin \Map_{\cC}(\unit,Y).
    \]
\end{prop}
\begin{proof}
Consider the following composition of morphisms in $\Mod_{\cC}^{\trl}$:
\[\begin{tikzcd}
\cC	&&{\cC[A]} && \cC \\
\cC	&&{\cC[A]} && \cC.
	\arrow[Rightarrow, no head, from=1-3, to=2-3]
	\arrow[Rightarrow, no head, from=1-1, to=2-1]
	\arrow["a_!", from=1-1, to=1-3]
	\arrow["a_!"', from=2-1, to=2-3]
	\arrow["Y", from=1-5, to=2-5]
	\arrow["X", from=1-3, to=1-5]
	\arrow["X"', from=2-3, to=2-5]
	\arrow["f"{description},shorten <=7pt, shorten >=7pt,  Rightarrow, from=2-3, to=1-5]
	\arrow["\gamma_!"{description},shorten <=7pt, shorten >=7pt,  Rightarrow, from=2-1, to=1-3]
\end{tikzcd}
\]
It follows from \Cref{cor:RestrictionInduction} that this composition is given by the morphism
\[
\begin{tikzcd}
\cC	 && \cC \\
\cC	 && \cC.
	\arrow[Rightarrow, no head, from=1-1, to=2-1]
	\arrow["X_a", from=1-1, to=1-3]
	\arrow["X_a"', from=2-1, to=2-3]
	\arrow["Y", from=1-3, to=2-3]
	\arrow["f_a\circ \gamma"{description},shorten <=7pt, shorten >=7pt,  Rightarrow, from=2-1, to=1-3]
\end{tikzcd}
\]
In particular, we deduce from the functoriality of $\Cc$-linear trace functor that there is an identification
\[\TrMod_{\Cc}(X,f) \circ \TrMod_{\Cc}(a_!,\gamma_!) \simeq \TrMod_{\Cc}(X_a,f_a \circ \gamma) \qin \Map_{\Cc}(\unit,Y).
\] 
The $\Cc$-linear trace $\TrMod_{\Cc}(X,f)$ is by definition the character $\chi_f\colon \one[LA] \to Y$, and by \Cref{Loop_Indicator} the $\Cc$-linear trace $\TrMod_{\cC}(a_!,\gamma_!)$ is given by $\one \oto{\gamma} \one[LA]$. It follows that the left hand side of the above identification is given by the composite 
\[
\one \oto{\gamma} \one[LA] \oto{\chi_f} Y,
\]
which is exactly $\chi_f(\gamma)$. On the other hand, we have by \Cref{THH_Trace} that the right hand side, $\Tr_{\cC}(X_a,f_a \circ \gamma)$, is the generalized trace of the generalized endomorphism $f_a \circ \gamma\colon X_a \to X_a \otimes Y$. The result follows. 
\end{proof}

\begin{cor} 
\label{cor:Character_Formula_Id}
Let $A$ be a space and let $\gamma\in LA$ with $a := \gamma(\pt)$. For $X\in \cC[A]$, we have 
    \[
        \chi_X(\gamma) \simeq \gentr{X_a \oto{\gamma} X_a}{X_a}
        \qin \End(\one).
    \]
\end{cor}

\begin{proof}
This follows from \Cref{prop:Character_Formula} by taking $Y=\one$ and $f = \id_X$. 
\end{proof}

\subsection{Coherence of trace formulas}
\label{subsec:Coherence_Trace_Formulas}
In the previous subsection, we have shown various explicit formulas of $\Cc$-linear traces. The goal of this subsection is to show that these formulas are natural in the input variables. This question is somewhat more subtle than one might first expect, due to the fact that the explicit description of the functoriality of the $\Cc$-linear trace functor given at the beginning of \Cref{sec:CLinearCategories} is a pointwise description which involves choices of duality data for the objects involved. 

\subsubsection{Coherent version of the trace comparison lemma}
We will start by proving a natural version of the trace comparison lemma, \Cref{THH_Trace}, in space-indexed families of generalized endomorphisms. The main input will be the following observation about the interaction between $\cC$-linear traces and $\cC^A$-linear traces for a given space $A$. Since $\Mod_{\cC}^{\trl}$ is functorial in $\cC$, there is a coassembly map
\[
    \Mod_{\cC^A}^{\trl} = \Mod_{\lim_A\cC}^{\trl} \xrightarrow{\coas} \lim_A \Mod_{\cC}^{\trl} = \Fun(A,\Mod_{\cC}^{\trl})
\]
associated to the functor $\Mod_{(-)}^{\trl}\colon \CAlg(\PrL) \to \PrL$. Informally speaking, it sends a pair $(\cD,F)$ to the functor
\[
    A \to \Mod^{\trl}_{\cC}, \qquad a \mapsto (\cD_a,F_a)
\]
where $\cD_a = \cD \otimes_{\cC^A} \cC$ is the base change of $\cD$ along the map $a^*\colon \cC^A \to \cC$ in $\CAlg(\PrL)$ and similarly for $F_a$.

\begin{lem}
\label{lem:AssemblyTHH}
For a space $A$, the following diagram commutes:
\begin{equation*}
\begin{tikzcd}
    \Mod^{\trl}_{\cC^A}  \dar[swap]{\coas} \rar{\TrMod_{\cC^A}} & \cC^A \dar[equal] \\
    \Fun(A,\Mod^{\trl}_{\cC}) \rar{\TrMod_{\cC} \circ -} & \Fun(A,\cC).
\end{tikzcd}
\end{equation*}
\end{lem}
\begin{proof}
Viewing the right vertical map as the coassembly map for forgetful functor $\CAlg(\PrL)\to \PrL$, the commutativity of the square follows from the naturality of the coassembly map.
\end{proof}

In other words, given a $\cC^A$-linear $\infty$-category $\cD$ equipped with an endofunctor $F\colon \cD \to \cD$, its $\cC^A$-linear trace $\TrMod_{\cC^A}(\cD,F) \in \cC^A$ is the $A$-indexed family of objects of $\cC$ obtained by pointwise taking the $\cC$-linear trace of the family $(\cD_a,F_a)_{a \in A}$:
\begin{align*}
    \TrMod_{\cC^A}(\cD,F)_a \simeq \TrMod_{\cC}(\cD_a,F_a) \qin \Fun(A,\cC).
\end{align*}

We now obtain the following generalization of \Cref{THH_Trace} for space-indexed family of generalized endomorphisms:

\begin{lem}[Trace comparison lemma, coherent version]
\label{THH_Trace_Coherent}
Let $A$ be a space, let $X,Y,Z\colon A \to \cC$ be $A$-indexed families of objects of $\cC$, and let $f_a\colon Z_a \tensor X_a \to X_a \tensor Y_a$ be a family of generalized endomorphisms in $\cC$. Assume that $X_a$ is dualizable in $\cC$ for every $a \in A$. Then the choice of homotopy
\[
    \TrMod_{\cC}(X_a,f_a) \simeq \gentr{f_a}{X_a}\qin \Map_{\cC}(Z_a, Y_a)
\]
of \Cref{THH_Trace} can be chosen to depend naturally on $a \in A$.
\end{lem}
\begin{proof}
We may consider $f$ as a generalized endomorphism $Z \tensor X \to X \tensor Y$ of $X$ in $\cC^A$. By \Cref{THH_Trace}, its generalized trace $\gentrC{\cC^A}{f}{X}\colon Z \to Y$ in $\cC^A$ is homotopic to the map $\TrMod_{\cC^A}(X,f)$ obtained by forming the $\cC^A$-linear trace of the map $(X,f)\colon (\cC^A,Z) \to (\cC^A,Y)$ in $\Mod^{\trl}_{\cC^A}$:
\[\begin{tikzcd}
	\cC^A && \cC^A \\
	\cC^A && \cC^A.
	\arrow["Z"', from=1-1, to=2-1]
	\arrow["X", from=1-1, to=1-3]
	\arrow["X"', from=2-1, to=2-3]
	\arrow["Y", from=1-3, to=2-3]
	\arrow["f"{description},shorten <=7pt, shorten >=7pt, Rightarrow, from=2-1, to=1-3]
\end{tikzcd}\]
Since generalized traces are functorial under symmetric monoidal functors
(\Cref{rmk:GeneralizedTraceFunctorial}) the map $\gentrC{\cC^A}{f}{X}\colon Z \to Y$ in $\cC^A$ corresponds to the $A$-indexed family of maps $\gentrC{\cC}{f_a}{X_a}\colon Z_a \to Y_a$:
\[
    \gentrC{\cC^A}{f}{X}_a \simeq \gentrC{\cC}{f_a}{X_a} \qin \Map_{\Cc}(Z_a,Y_a).
\]
For the same reason, we deduce from \Cref{lem:AssemblyTHH} that the map $\TrMod_{\cC^A}(X,f)\colon Z \to Y$ in $\cC^A$ is given by the $A$-indexed family of maps $\TrMod_{\cC}(X_a,f_a)\colon Z_a \to Y_a$:
\[
    \TrMod_{\cC^A}(X,f)_a \simeq \TrMod_{\cC}(X_a,f_a) \qin \Map_{\Cc}(Z_a,Y_a).
\]
As $\TrMod_{\cC^A}(X,f) \simeq \gentrC{\cC^A}{f}{X} \in \Map_{\cC^A}(Z,Y)$, this finishes the proof.
\end{proof}

\subsubsection{Coherent version of the loop indicator lemma}
We will next prove a coherent version of the loop indicator lemma, \Cref{Loop_Indicator}. Recall that the loop indicator lemma describes the effect of $\TrMod_{\Cc}$ on the map $(a_!,\gamma_!)$ in $\Mod_{\Cc}^{\trl}$, which is obtained from a free loop $\gamma \in LA$ by interpreting it as a self-homotopy of the basepoint inclusion $\gamma(\ast)\colon \pt \to A$. We will start by generalizing this description to the case of an arbitrary self-homotopy of a map of spaces.

\begin{prop}\label{HomotopyviaTHH}
Let $f\colon A\to B$ be a map of spaces, and let $H\colon A \times S^1 \to B$ be a self-homotopy of $f$. Applying the functor $\Cc[-]\colon \Spc \to \Mod_{\Cc}$ gives an automorphism $H_!$ of $f_!$ in $\Fun_{\Cc}(\Cc[A],\Cc[B])$, which gives rise to a morphism $(f_!,H_!)\colon \cC[A] \to \cC[B]$ in $\Mod_{\cC}^{\trl}$:
\[\begin{tikzcd}
	{\cC[A]} && {\cC[B]} \\
	{\cC[A]} && {\cC[B].}
	\arrow["{f_!}", from=1-1, to=1-3]
	\arrow["{f_!}"', from=2-1, to=2-3]
	\arrow[Rightarrow, no head, from=1-3, to=2-3]
	\arrow[Rightarrow, no head, from=1-1, to=2-1]
	\arrow[shorten <=6pt, shorten >=6pt,"{H_!}"{description}, Rightarrow, from=2-1, to=1-3]
\end{tikzcd}\] 
Applying $\TrMod_{\cC}$ to this map gives the map $\one[LA]\to \one[LB]$ induced by the composite
\begin{align*}
    LA \simeq LA \times \pt \oto{LA \times [\id_{S^1}]} LA\times LS^1 \simeq  L(A\times S^1)\oto{LH} LB
\end{align*}
which informally sends a loop $\gamma \in LA$ in $A$ to the loop $s\mapsto H(\gamma(s),s)$ in $B$. 
\end{prop}
\begin{proof}
Let $H_A\colon S^1 \to \Map(A,A \times S^1)$ denote the mate of the identity of $A \times S^1$. It is a self-homotopy of the map $A \times b\colon A \to A \times S^1$, where $b\colon \pt \to S^1$ is the inclusion of the basepoint. Observe that the self-homotopy $S^1 \to \Map(A,B)$ classified by $H$ can be obtained from $H_A$ by postcomposing with $H\colon A \times S^1 \to B$, and in particular the square in the statement of the proposition can be written as the composite of the squares
\[\begin{tikzcd}
    	{\cC[A]} && {\cC[A\times S^1]} \\
    	{\cC[A]} && {\cC[A\times S^1]}
    	\arrow["{(A \times b)_!}", from=1-1, to=1-3]
    	\arrow["{(A \times b)_!}"', from=2-1, to=2-3]
    	\arrow[Rightarrow, no head, from=1-3, to=2-3]
    	\arrow[Rightarrow, no head, from=1-1, to=2-1]
    	\arrow[shorten <=6pt, shorten >=6pt, "{(H_A)_!}"{description}, Rightarrow, from=2-1, to=1-3]
    \end{tikzcd}
    \qquad \text{ and } \qquad
    \begin{tikzcd}
    	{\cC[A\times S^1]} && {\cC[B]} \\
    	{\cC[A\times S^1]} && {\cC[B].}
    	\arrow["{H_!}", from=1-1, to=1-3]
    	\arrow["{H_!}"', from=2-1, to=2-3]
    	\arrow[Rightarrow, no head, from=1-3, to=2-3]
    	\arrow[Rightarrow, no head, from=1-1, to=2-1]
    	\arrow[equal, shorten <=8pt, shorten >=8pt, from=2-1, to=1-3]
    \end{tikzcd}
    \] 
By \Cref{thm:THHFreeLoops}, applying $\TrMod_{\cC}$ to the right square gives the map $\one[LH]\colon \one[L(A \times S^1)] \to \one[LB]$. For the left square we observe that the homotopy $H_A$ is equivalent to the product $A \times H_b$, where $H_b\colon S^1 \to \Map(\pt, S^1)$ is the canonical self-homotopy of $b\colon \pt \to S^1$. By symmetric monoidality of $\TrMod_{\cC}$, we thus find that $\TrMod_{\cC}$ of the left square is given by
$$LA\times L\pt\xrightarrow{LA\times \TrMod_{\cC}(b_!,(H_b)_!)} LA\times LS^1 \simeq L(A \times S^1).$$
By \Cref{Loop_Indicator}, applied to the identity $\gamma = \id_{S^1} \in LS^1$, the map $\TrMod_{\cC}(b_!,(H_b)_!))\colon \pt \to LS^1$ is exactly the map $\pt\xrightarrow{[\id_{S^1}]} LS^1$. This finishes the proof.
\end{proof}

In the following, let $d\colon L^2A \to LA$ denote the map induced by the diagonal $S^1 \to S^1 \times S^1$, and let $\can_A\colon S^1 \to \Map(LA,A)$ denote the mate of the evaluation map $\ev_A\colon LA \times S^1 \to A$. It is a canonical self-homotopy of the map $e\colon LA \to A$ given by evaluation at $* \in S^1$.

\begin{cor}
\label{DiagonalViaHH}
    Applying $\TrMod_{\cC}$ to the diagram
    \[\begin{tikzcd}
    	{\cC[LA]} && {\cC[A]} \\
    	{\cC[LA]} && {\cC[A].}
    	\arrow["{e_!}", from=1-1, to=1-3]
    	\arrow["{e_!}"', from=2-1, to=2-3]
    	\arrow[Rightarrow, no head, from=1-3, to=2-3]
    	\arrow[Rightarrow, no head, from=1-1, to=2-1]
    	\arrow[shorten <=6pt, shorten >=6pt,"{\can_{A,!}}"{description}, Rightarrow, from=2-1, to=1-3]
    \end{tikzcd}\] 
    gives the map $\one[L^2A] \oto{d} \one[LA]$.
\end{cor}
\begin{proof}
By \Cref{HomotopyviaTHH}, the map $\TrMod_{\cC}(e_!,\can_{A,!})\colon \one[L^2A] \to \one[LA]$ is induced from the composite 
$$L^2A = LLA\to L(LA\times S^1)\oto{L(\ev_A)} LA,$$ 
where the first map sends a double loop $\gamma\colon S^1 \times S^1 \to A$ in $A$ to the loop $(s\mapsto (\gamma(s,-),s))$ in $LA \times S^1$. In particular, the composite is given by $\gamma \mapsto (s\mapsto \gamma(s,s))$, which is precisely the map $d\colon L^2A \to LA$ as claimed. 
\end{proof}

All in all, we obtain a coherent version of the loop indicator lemma, \Cref{Loop_Indicator}.

\begin{cor}[Loop indicator lemma, coherent version]
\label{Loop_Indicator_Coherent}
For a space $A$, the equivalence
\[
    \TrMod_{\cC}(a_!,\gamma_!) \simeq \gamma \qin \Map_{\cC}(\unit,\unit[LA])
\]
from \Cref{Loop_Indicator} may be chosen naturally in $\gamma \in LA$.
\end{cor}
\begin{proof}
Consider the map $(e_!,\can_{A,!})\colon \cC[LA] \to \cC[A]$ in $\Mod_{\cC}^{\trl}$ considered in \Cref{DiagonalViaHH}. It corresponds to an $LA$-indexed family of maps $\cC \to \cC[A]$ given at $\gamma \in LA$ by precomposing with $\gamma_!\colon \cC \to \cC[LA]$. We claim that this is the $LA$-indexed diagram of maps $(a_!,\gamma_!)\colon \cC \to \cC[A]$ whose $\cC$-linear trace we want to compute:
\[
\begin{tikzcd}
	\cC & {\cC[LA]} & {\cC[A]} \\
	\cC & {\cC[LA]} & {\cC[A]}
	\arrow["{\gamma_!}", from=1-1, to=1-2]
	\arrow["{e_!}", from=1-2, to=1-3]
	\arrow["{\gamma_!}"', from=2-1, to=2-2]
	\arrow["{e_!}"', from=2-2, to=2-3]
	\arrow[Rightarrow, no head, from=1-1, to=2-1]
	\arrow[Rightarrow, no head, from=1-2, to=2-2]
	\arrow[Rightarrow, no head, from=1-3, to=2-3]
	\arrow["{\can_{A,!}}"{description}, Rightarrow, from=2-2, to=1-3]
	\arrow[shorten <=4pt, shorten >=4pt, Rightarrow, no head, from=2-1, to=1-2]
\end{tikzcd}
\qquad = \qquad
\begin{tikzcd}
	\cC && {\cC[A]} \\
	\cC && {\cC[A].}
	\arrow[Rightarrow, no head, from=1-1, to=2-1]
	\arrow[Rightarrow, no head, from=1-3, to=2-3]
	\arrow["{a_!}", from=1-1, to=1-3]
	\arrow["{a_!}"', from=2-1, to=2-3]
	\arrow[shorten <=6pt, shorten >=6pt,"\gamma_!"{description}, Rightarrow, from=2-1, to=1-3]
\end{tikzcd}
\]
Indeed, one observes that evaluating the canonical homotopy $\can_A\colon S^1 \to \Map(LA,A)$ at $\gamma \in LA$ gives the loop $\gamma\colon S^1 \to A$, functorially in $\gamma$. By \Cref{lem:THHOfFamiliesOfDiagrams}, applying the $\cC$-linear trace functor $\TrMod_{\Cc}$ to this family of maps in $\Mod_{\cC}^{\trl}$ gives the map $LA \to \Map_{\cC}(\unit,\unit[LA])$ which is the mate of the composite
\begin{align*}
    \one[LA] \xrightarrow{c_{LA}} \one[L^2A] \xrightarrow{\TrMod_{\cC}(e_!,\can_{A,!})} \one[LA].
\end{align*}
By \Cref{DiagonalViaHH} the second map is the map $d\colon \one[L^2A] \to \one[LA]$, and thus the composite is the identity $\one[LA] \to \one[LA]$. The mate of this is the map $LA \to \Map_{\cC}(\unit,\unit[LA])$ which sends $\gamma$ to $\unit \xrightarrow{\gamma} \unit[LA]$, as desired.
\end{proof}

\subsubsection{Coherent version of the character formula}
We are now in a position to prove that the formula for generalized characters given in \Cref{prop:Character_Formula} is natural in the loop $\gamma \in LA$.

\begin{prop}[Character formula, coherent version]
\label{prop:Character_Formula_Coherent}
Let $X \in \Cc^A$ and $f\colon X \to X \tensor A^*Y$ be as in \Cref{prop:Character_Formula}. Then the equivalence
    \[
        \chi_f(\gamma) \simeq \gentr{X_a \oto{\gamma} 
            X_a \oto{f_a} 
            X_a \otimes Y}{X_a} \qin \Map_{\cC}(\unit,Y)
    \]
    of \Cref{prop:Character_Formula} is natural in $\gamma \in LA$.
\end{prop}
\begin{proof}
We will go through the proof of \Cref{prop:Character_Formula} and check that everything is natural in $\gamma$. For every $\gamma \in LA$, we consider the following composition of morphisms in $\Mod_{\cC}^{\trl}$:
\[\begin{tikzcd}
\cC	&&{\cC[A]} && \cC \\
\cC	&&{\cC[A]} && \cC
	\arrow[Rightarrow, no head, from=1-3, to=2-3]
	\arrow[Rightarrow, no head, from=1-1, to=2-1]
	\arrow["a_!", from=1-1, to=1-3]
	\arrow["a_!"', from=2-1, to=2-3]
	\arrow["Y", from=1-5, to=2-5]
	\arrow["X", from=1-3, to=1-5]
	\arrow["X"', from=2-3, to=2-5]
	\arrow["f"{description},shorten <=7pt, shorten >=7pt,  Rightarrow, from=2-3, to=1-5]
	\arrow["\gamma_!"{description},shorten <=7pt, shorten >=7pt,  Rightarrow, from=2-1, to=1-3]
\end{tikzcd}
\qquad \simeq \qquad 
\begin{tikzcd}
\cC	 && \cC \\
\cC	 && \cC.
	\arrow[Rightarrow, no head, from=1-1, to=2-1]
	\arrow["X_a", from=1-1, to=1-3]
	\arrow["X_a"', from=2-1, to=2-3]
	\arrow["Y", from=1-3, to=2-3]
	\arrow["f_a\circ \gamma"{description},shorten <=7pt, shorten >=7pt,  Rightarrow, from=2-1, to=1-3]
\end{tikzcd}
\]
Observe that this identification is natural in $\gamma \in LA$. By \Cref{THH_Trace_Coherent}, the $\Cc$-linear trace of the right-hand side can naturally be identified with the generalized trace of the map $f_a \circ \gamma\colon X_a \to X_a \otimes Y$. Furthermore, by \Cref{Loop_Indicator_Coherent} the $\Cc$-linear trace of $(a_!,\gamma_!)\colon \Cc \to \Cc[A]$ can naturally be identified with the morphism $\gamma\colon \unit \to \unit[LA]$. Since we have $\chi_f = \TrMod_{\Cc}(X,f)$ by definition, the above diagram thus gives a natural equivalence
\[
    \chi_f(\gamma) = \chi_f \circ \gamma \simeq \gentr{f_a \circ \gamma}{X_a} \qin \Map_{\Cc}(\unit,Y),
\]
finishing the proof.
\end{proof}

Given a generalized endomorphism $f \colon X \to X \otimes A^*Y$ in $\cC^A$, its generalized trace in $\Cc^A$ is a map $A^*\unit \to A^*Y$, which by adjunction corresponds to a map
\[
    \gentr{f}{X}\colon \unit[A] \simeq A_!A^*\unit \to Y \qin \Cc.
\]
It follows from the coherent character formula of \Cref{prop:Character_Formula_Coherent} that the generalized trace $\gentr{f}{X}$ of $f$ can be obtained from the the character $\chi_f\colon \unit[LA] \to Y$ by precomposition with the inclusion of constant loops $c\colon \unit[A] \to \unit[LA]$.

\begin{cor}
\label{Char_Trace}
    Let $A$ be a space, $X \in \cC^A$ a dualizable object and $f\colon X \to X \tensor A^*Y$ in $\cC^A$ a generalized morphism. Then the following diagram commutes:
    \[\begin{tikzcd}
    	{\one[LA]}  \\
    	{\one[A]} && Y
    	\arrow["c", from=2-1, to=1-1]
    	\arrow["{\chi_f}",pos=0.4, from=1-1, to=2-3]
    	\arrow["{\gentr{f}{X}}"', from=2-1, to=2-3]
    \end{tikzcd}\]
\end{cor}
\begin{proof}
For $a \in A$, let $c_a \in LA$ denote the constant loop on $a$. By \Cref{prop:Character_Formula_Coherent} and \Cref{THH_Trace_Coherent}, there are natural equivalences
\[
    \chi_f(c_a) \simeq \gentr{f_a\colon X_a \to X_a \otimes Y}{X_a} \simeq \gentr{f}{X}_a \qin \Map_{\Cc}(\unit,Y).
\]
This finishes the proof.
\end{proof}

\subsection{Restriction and induction of characters}
\label{subsec:RestrictionInductionCharacters}
In this subsection, we consider the behavior of the character maps under restriction and induction along maps of spaces $g\colon A \to B$.

\begin{defn}
\label{def:RestrictionFamily}
Let $X \in \cC^B$ be dualizable and let $f\colon X \to X \tensor B^*Y$ be a morphism in $\cC^B$. We define the \textit{restriction} of $f$ along $g$ as the the composite
\[
    \Res_g(f)\colon g^*X \xrightarrow{g^*(f)} g^*(X \tensor B^*Y) \simeq g^*X \tensor A^*Y \qin \cC^A.
\]
\end{defn}
Note that $g^*X \in \cC^A$ is still pointwise dualizable, so $\Res_g(f)$ corresponds to an $A$-family of generalized endomorphisms
\[
    \Res_g(f)_a = f_{g(a)} \colon 
    X_{g(a)} \to   
    X_{g(a)}\otimes Y.
\]

\begin{lem}[Restricted character formula]
\label{Res_Char}
In the setting of \Cref{def:RestrictionFamily}, the character $\chi_{\Res_g(f)}\colon \unit[LA] \to Y$ of the restriction $\Res_g(f)\colon g^*X \to g^*X \tensor A^*Y$ is given by the composite
    \[
        \one[LA] \oto{\TrMod_{\cC}(g_!)} 
        \one[LB] \oto{\:\chi_f\:} 
        Y.
    \]
\end{lem}
\begin{proof}
Regarding $f$ as a transformation of $\cC$-linear functors $\cC[B] \to \cC$ and $\Res_g(f)$ as a transformation of $\cC$-linear functors $\cC[A] \to \cC$, \Cref{cor:RestrictionInduction} tells us that $\Res_g(f)$ is obtained from $f$ by whiskering with the functor $g_!\colon \cC[A] \to \cC[B]$. In particular, we get the following equivalence of maps $\cC[A] \to (\cC,Y)$ in $\Mod^{\trl}_{\cC}$:
    \[
    \begin{tikzcd}
    	{\cC[B]} && \cC \\
    	{\cC[B]} && \cC,
    	\arrow[Rightarrow, no head, from=1-1, to=2-1]
    	\arrow["{g^*X}", from=1-1, to=1-3]
    	\arrow["{g^*X}"', from=2-1, to=2-3]
    	\arrow["Y", from=1-3, to=2-3]
    	\arrow["{\Res_g(f)}"{description}, shorten <=7pt, shorten >=7pt, Rightarrow, from=2-1, to=1-3]
    \end{tikzcd}
    \qquad \simeq \qquad
    \begin{tikzcd}
    	{\cC[A]} && {\cC[B]} && \cC \\
    	{\cC[A]} && {\cC[B]} && \cC.
    	\arrow[Rightarrow, no head, from=1-1, to=2-1]
    	\arrow[Rightarrow, no head, shorten <=7pt, shorten >=7pt, from=2-1, to=1-3]
    	\arrow["{g_!}", from=1-1, to=1-3]
    	\arrow["{g_!}"', from=2-1, to=2-3]
    	\arrow[Rightarrow, no head, from=1-3, to=2-3]
    	\arrow["{Y}", from=1-5, to=2-5]
    	\arrow["{X}", from=1-3, to=1-5]
    	\arrow["{X}"', from=2-3, to=2-5]
    	\arrow["f"{description}, shorten <=7pt, shorten >=7pt, Rightarrow, from=2-3, to=1-5]
    \end{tikzcd}\]
    We thus obtain the claim by applying the functor $\TrMod_{\cC}\colon \Mod^{\trl}_{\cC} \to \cC$.
\end{proof}

\begin{defn}
\label{def:InductionFamily}
Let $X \in \cC^A$ be dualizable and let $f\colon X \to X \tensor A^*Y$ be a morphism in $\cC^A$. Assume that the map $g\colon A \to B$ is $\cC$-adjointable. We define the \textit{induction} of $f$ along $g$ as the the composite
\[
    \Ind_g(f)\colon g_!X \xrightarrow{g_!(f)} g_!(X \tensor A^*Y) \simeq g_!(X \tensor g^*B^*Y) \simeq g_!X \tensor B^*Y \qin \cC^B.
\]
\end{defn}
Note that the object $g_!X \in \cC^B$ is again dualizable by \Cref{prop:ColimIsDbl}, and thus $\Ind_g(f)$ corresponds to a $B$-family of generalized endomorphisms
\[
    \Ind_g(f)_b = \colim_{a \in A_b} f_a \colon 
    (\colim_{a \in A_b} X_a) \too   
    (\colim_{a \in A_b} X_a)\otimes Y,
\]
where $A_b$ is the fiber of $A\oto{g} B$ over $b\in B$. 

\begin{war}
Without the assumption that $g\colon A \to B$ is $\cC$-adjointable, the object $g_!X\in \cC^B$ will not in general be dualizable and the character of $\Ind_g(f)$ is not defined.
\end{war}

\begin{thm}[Induced character formula]
\label{Ind_Char}
In the setting of \Cref{def:InductionFamily}, the character $\chi_{\Ind_g(f)}\colon \unit[LB] \to Y$ of the induction $\Ind_g(f)\colon g_!(X) \to g_!(X) \tensor B^*Y$ is given by the composition
    \[
        \one[LB] \oto{\TrMod_{\cC}(g^*)} 
        \one[LA] \oto{\quad\chi_f\quad} 
        Y.
    \]
\end{thm}
\begin{proof}
The proof is essentially the same as that of \Cref{Res_Char}. Regarding $f$ as a transformation of $\cC$-linear functors $\cC[A] \to \cC$ and $\Ind_g(f)$ as a transformation of $\cC$-linear functors $\cC[B] \to \cC$, \Cref{cor:RestrictionInduction} tells us that $\Ind_g(f)$ is obtained from $f$ by whiskering with the functor $g^*\colon \cC[B] \to \cC[A]$. In particular, we get the following equivalence of maps $\cC[B] \to (\cC,Y)$ in $\Mod^{\trl}_{\cC}$:
    \[
    \begin{tikzcd}
    	{\cC[B]} && \cC \\
    	{\cC[B]} && \cC.
    	\arrow[Rightarrow, no head, from=1-1, to=2-1]
    	\arrow["{g_!X}", from=1-1, to=1-3]
    	\arrow["{g_!X}"', from=2-1, to=2-3]
    	\arrow["Y", from=1-3, to=2-3]
    	\arrow["{\Ind_g(f)}"{description}, shorten <=7pt, shorten >=7pt, Rightarrow, from=2-1, to=1-3]
    \end{tikzcd}
    \qquad \simeq \qquad
    \begin{tikzcd}
    	{\cC[B]} && {\cC[A]} && \cC \\
    	{\cC[B]} && {\cC[A]} && \cC.
    	\arrow[Rightarrow, no head, from=1-1, to=2-1]
    	\arrow[Rightarrow, no head, shorten <=7pt, shorten >=7pt, from=2-1, to=1-3]
    	\arrow["{g^*}", from=1-1, to=1-3]
    	\arrow["{g^*}"', from=2-1, to=2-3]
    	\arrow[Rightarrow, no head, from=1-3, to=2-3]
    	\arrow["{Y}", from=1-5, to=2-5]
    	\arrow["{X}", from=1-3, to=1-5]
    	\arrow["{X}"', from=2-3, to=2-5]
    	\arrow["f"{description}, shorten <=7pt, shorten >=7pt, Rightarrow, from=2-3, to=1-5]
    \end{tikzcd}\]
    We thus obtain the claim by applying the functor $\TrMod_{\cC}\colon \Mod^{\trl}_{\cC} \to \cC$.
\end{proof}

As a special case of \Cref{Ind_Char}, we obtain main result of this section: a formula for the trace of a colimit of endomorphisms.

\begin{cor}[Traces and dimensions of colimits]
\label{cor:Trace_Dimension_Colimit}
Let $A$ be a $\cC$-adjointable space, and let $X \in \cC^A$ be a dualizable object. For an endomorphism $f\colon X \to X$ in $\cC^A$, the trace of the induced endomorphism $A_!(f)\colon A_!(X) \to A_!(X)$ on colimits over $A$ is given by the composite
\[
    \tr(A_!(f))\colon   
    \unit \xrightarrow{\TrMod_{\cC}(A^*)} 
    \unit[LA] \xrightarrow{\quad\chi_f\quad} \unit.
\]
In particular, the dimension $\dim(A_!(X))$ of the colimit of $X$ is given by the composite
\[
    \dim(A_!(X))\colon 
    \one \xrightarrow{\TrMod_{\cC}(A^*)} 
    \one[LA] \xrightarrow{\:\:\chi_X\:\:} 
    \unit.
\]
\end{cor}

\subsection{Additivity of generalized traces}
\label{subsec:Additivity_Traces}
Let $V_1$, $V_2$ and $V_3$ be finite-dimensional vector spaces over a field $k$, fitting in a short exact sequence $0 \to V_1 \xrightarrow{\varphi} V_2 \xrightarrow{\psi} V_3 \to 0$. Assume that $V_i$ comes equipped with an endomorphism $f_i\colon V_i \to V_i$ for $i=1,2,3$ and that these fit into a commutative diagram
\[
\begin{tikzcd}
    V_1 \dar[swap]{f_1} \rar{\varphi} & V_2 \dar{f_2} \rar{\psi} & V_3 \dar{f_3} \\
    V_1 \rar{\varphi} & V_2 \rar{\psi} & V_3.
\end{tikzcd}
\]
In this case, the trace of $f_2$ may be expressed in terms of the traces of $f_1$ and $f_3$ as follows:
\[
   \gentr{f_2}{V_2} = \gentr{f_1}{V_1} + \gentr{f_3}{V_3}
\]
To see this, one may choose an identification $V_2 \simeq V_1 \oplus V_3$ and write the map $f_2\colon V_2 \to V_2$ as a block matrix $\begin{pmatrix} f_1 & g \\ 0 & f_3\end{pmatrix}\colon V_1 \oplus V_3 \to V_1 \oplus V_3$ for some map $g\colon V_3 \to V_1$.

An additivity result for traces in more general general contexts, such as tensor-triangulated categories, has proved subtle. In \cite{ferrand2005}, Daniel Ferrand provided an example of an endomorphism in the derived category of an exact triangle of perfect complexes for which the trace of the middle map is not equal to the sum of the traces of the two other ones. The issue can be phrased as a coherence problem: when regarding the short exact sequence as a cofiber sequence in the derived $\infty$-category, the chain homotopies exhibiting the diagram as commutative should be suitably compatible with the null-homotopies witnessing the sequences as cofiber sequences. In particular, the bare triangulated category contains too little information to allow for an additivity result for traces.

Since then, additivity of traces has been proved in more structured settings. In \cite{may2001additivity}, May proved additivity of traces for tensor-triangulated categories coming from a stable closed symmetric monoidal model category. In \cite{GPS2014AdditivityTraces}, Groth, Ponto and Shulman translated the statement and proof of May's result into the language of stable derivators. In fact, as Grothendieck explained in a letter to Thomason\footnote{See \url{http://matematicas.unex.es/~navarro/res/thomason.pdf}.}, the failure of additivity of traces in the bare triangulated setting was his main motivation for the development of the formalism of derivators, where mapping cones would be functorial. Finally, in \cite{PSLinearity}, Ponto and Shulman gave a proof of the additivity of traces which is much closer to our approach. 

The goal of this section is to give a new proof of additivity of generalized traces in a stable presentably symmetric monoidal $\infty$-category $\cC$. Our method is fundamentally different from those of \cite{may2001additivity} and \cite{GPS2014AdditivityTraces}, in which the required homotopy between trace maps was produced directly via a large diagram chase. Instead, we will deduce the result from a categorification of the problem, similarly to \cite{PSLinearity}: the $\cC$-linear trace functor $\TrMod_{\Cc}\colon \Mod_{\cC}^{\trl} \to \cC$ is ``additive'' in a sense analogous to the ``additive invariants'' of \cite{BGT_KTheory}. The advantage of this approach is that the additivity of the functor $\TrMod_{\Cc}$ is a \textit{property} (sending localization sequences to cofiber sequences, cf.\ \Cref{thm:localization}) rather than \textit{structure} (a homotopy between the various trace maps). The approach in \cite{PSLinearity} is to use a variant of \Cref{cor:Trace_Dimension_Colimit} with $A$ being a $1$-category in place of a space. 

For the convenience of the reader, we will recall from \cite{HSS} the relevant definitions and statements about the additivity of the categorical trace. Throughout this section, $\cC$ is a fixed \textit{stable} presentably symmetric monoidal $\infty$-category.

\begin{defn}[{\cite[Definition~3.2]{HSS}}]
\label{def:localization_sequence}
	A sequence $\cD \xrightarrow\iota \cE \xrightarrow\pi \cF$ in $\Mod_{\cC}$ is called a \emph{localization sequence} if the following conditions hold:
	\begin{itemize}
		\item $\iota$ and $\pi$ have $\cC$-linear right adjoints $\iota^r$ and $\pi^r$;
		\item the composite $\pi\iota\colon \cD \to \cF$ is the zero functor;
		\item the unit $\eta\colon\id_\cD \to \iota^r \iota$ and the counit $\epsilon\colon \pi\pi^r\to\id_\cF$ are equivalences;
		\item the sequence $\iota\iota^r \to \id_\cE \to \pi^r\pi$, with its unique nullhomotopy, is a cofiber sequence in $\Fun_{\cC}(\cE,\cE)$.
	\end{itemize}
	A sequence $(\cD,F) \xrightarrow{(\iota,\alpha)} (\cE,G) \xrightarrow{(\pi,\beta)} (\cF,H)$ in $\Mod_{\cC}^{\trl}$ is called a \emph{localization sequence} if $\cD\stackrel\iota\to \cE\stackrel\pi\to \cF$ is a localization sequence and moreover the morphisms $(\iota,\alpha)$ and $(\pi,\beta)$ are right adjointable in $\Mod_\cC$.
\end{defn}

The assumption that the composite $\pi \iota$ in a localization sequence is the zero functor implies that the sequence $\TrMod_{\cC}(\cD,F) \to \TrMod_{\cC}(\cE,G) \to \TrMod_{\cC}(\cF,H)$ in $\cC$ obtained by applying $\TrMod_{\cC}$ admits a canonical null-homotopy.

\begin{theorem}[{\cite[Theorem~3.4]{HSS}}]
\label{thm:localization}
Let 
\[
    (\cD,F) \xrightarrow{(\iota,\alpha)} (\cE,G) \xrightarrow{(\pi,\beta)} (\cF,H)
\] 
be a localization sequence in $\Mod_{\cC}^{\trl}$. Then the canonical null-homotopy exhibits the sequence
\[
    \TrMod_{\cC}(\cD,F) \to \TrMod_{\cC}(\cE,G) \to \TrMod_{\cC}(\cF,H)
\]
as a cofiber sequence in $\cC$.
\end{theorem}

One might refer to this result as saying that the functor $\TrMod_{\cC}$ is `localizing', cf.\ \cite[Definition 8.1]{BGT_KTheory}. We will show next how this implies that $\TrMod_{\cC}$ sends bifiber sequences between \textit{morphisms} in $\Mod_{\cC}^{\trl}$ to sums in mapping spaces in $\cC$.

\begin{defn}
\label{def:Bifiber_Sequence_Mod_Trl}
Let $(\cD,F)$ and $(\cE,G)$ be in $\Mod_{\cC}^{\trl}$ and consider morphisms $(\iota_i,\alpha_i)\colon (\cD,F) \to (\cE,G)$ for $i = 1,2,3$, that is, $\iota_i\colon \cD \to \cE$ is an internal left adjoint in $\Mod_{\cC}$ and $\alpha_i\colon \iota_i\circ F\to G\circ \iota_i$ is a morphism in $\Fun_\cC(\cD,\cE)$. We define a \textit{bifiber sequence}
\[
    (\iota_1,\alpha_1) \xrightarrow{\varphi} (\iota_2,\alpha_2) \xrightarrow{\psi} (\iota_3,\alpha_3)
\]
to be a bifiber sequence $\iota_1\xrightarrow{\varphi} \iota_2\xrightarrow{\psi} \iota_3$ in $\Fun_\cC(\cD,\cE)$ which fits into a bifiber sequence in $\Fun_{\cC}(\cD,\cE)^{[1]}$ of the form
\[
\begin{tikzcd}
	{\iota_1\circ F} \dar{\alpha_1} \rar{\varphi\circ F} & {\iota_2\circ F} \rar{\psi\circ F} \dar{\alpha_2} & {\iota_3\circ F} \dar{\alpha_3}\\
	{G\circ \iota_1} \rar{G\circ \varphi} & {G\circ \iota_2} \rar{G\circ \psi} & {G\circ \iota_3},
\end{tikzcd}
\]
where the top (resp.\ bottom) bifiber sequence is obtained by precomposing with $F$ (resp.\ postcomposing with $G$)
\end{defn}

\begin{rmk}
We emphasize at this point that the data of a bifiber sequence $X \xrightarrow{\varphi} Y \xrightarrow{\psi} Z$ in a stable $\infty$-category $\cE$ is that of a bicartesian square in $\cE$ of the form
\[
\begin{tikzcd}
    X \rar{\varphi} \dar & Y \dar{\psi} \\
    0 \rar & Z.
\end{tikzcd}
\]
In particular, a bifiber sequence in $\cE^{[1]}$ contains the data of a commutative cube in $\cE$. We will often leave the null-homotopy implicit in the notation to enhance readability.
\end{rmk}

\begin{corollary}
\label{cor:Additivity_THH}
Let 
\[
    (\iota_1,\alpha_1) \xrightarrow{\:\varphi\:} (\iota_2,\alpha_2) \xrightarrow{\:\psi\:} (\iota_3,\alpha_3)
\] 
be a bifiber sequence of morphisms $(\cD,F) \to (\cE,G)$ in $\Mod_{\cC}^{\trl}$. There is an equivalence
\[
    \TrMod_{\cC}(\iota_2, \alpha_2) \simeq \TrMod_{\cC}(\iota_1, \alpha_1) + \TrMod_{\cC}(\iota_3, \alpha_3)
\]
in the mapping space $\Map_{\cC}(\TrMod_{\cC}(\cD, F),\TrMod_{\cC}(\cE, G)))$.
\end{corollary}
\begin{proof}
This is a standard trick due to Waldhausen, adapted to the current context. Consider the full subcategory $S_2(\cE) \subseteq \Fun([1] \times [1], \cE)$ spanned by the bifiber sequences in $\cE$. Equip $S_2(\cE)$ with the endomorphism $S_2(G)\colon S_2(\cE)\to S_2(\cE)$ defined by pointwise application of $G$, which is well-defined as $G$ preserves cofiber sequences. The $\infty$-category $S_2(\cE)$ is canonically $\cC$-linear, and as such it is equivalent to $\cE^{[1]} \simeq \PSh([1]) \tensor \cE$. In particular, $S_2(\cE)$ is dualizable in $\Mod_{\cC}$, and thus the pair $(S_2(\cC),S_2(G))$ forms an object of $\Mod^{\trl}_{\cC}$. It fits into a localization sequence
\[
    \begin{tikzcd}
        (\cE,G) \ar[rrr, "{i(X) = (X,X,0)}", shift left, hookrightarrow]  &&& (S_2(\cE),S_2(G)) \ar[lll,shift left, "{r(X,Y,Z) = X}"] \ar[rrr,shift left, "{p(X,Y,Z) = Z}"] &&& (\cE,G) \ar[lll,shift left, hookrightarrow, "{s(Z) = (0,Z,Z)}"],
    \end{tikzcd}
\]
see for example \cite[Proof of proposition 7.17]{BGT_KTheory}. By \Cref{thm:localization}, the induced sequence
\[
    \TrMod_\cC(\cE,G) \xrightarrow{\TrMod_\cC(i)} \TrMod_\cC(S_2(\cE),S_2(G)) \xrightarrow{\TrMod_\cC(p)} \TrMod_\cC(\cE,G)
\]
is a cofiber sequence in $\cC$, and thus the maps $r\colon S_2(\cE) \to \cE$ and $p\colon S_2(\cE) \to \cE$ induce an equivalence
\[
    (\TrMod_\cC(r), \TrMod_\cC(p))\colon \TrMod_\cC(S_2(\cE),S_2(G)) \iso \TrMod_\cC(\cE,G) \oplus \TrMod_\cC(\cE,G),
\] 
whose inverse is $(\TrMod_\cC(i), \TrMod_\cC(s))$. In particular, the identity on $\TrMod_\cC(S_2(\cE),S_2(G))$ is equivalent to the sum of the maps $\TrMod_\cC(ir)$ and $\TrMod_\cC(sp)$. 

Observe that a bifiber sequence $(\iota_1,\alpha_1)\xrightarrow{\varphi} (\iota_2,\alpha_2)\xrightarrow{\psi} (\iota_3,\alpha_3)$ as in \Cref{def:Bifiber_Sequence_Mod_Trl} is the same data as a map $(\iota,\alpha)\colon (\cD,F) \to (S_2(\cE),S_2(G))$ in $\Mod_{\cC}^{\trl}$ whose three components are given by $(\iota_i,\alpha_i)$ for $i=1,2,3$. After applying $\TrMod_\cC(-)$, we then get that the map in $\cC$
\[
    \TrMod_\cC\left(\iota, \alpha\right) \colon \TrMod_\cC(\cD,F) \to \TrMod_\cC(S_2(\cE),S_2(G))
\]
is equivalent to the sum of
\[
    \TrMod_\cC(\iota_1,\iota_1,0) + \TrMod_\cC(0, \iota_3,\iota_3)\colon \TrMod_\cC(\cD,F) \to \TrMod_\cC(S_2(\cE),S_2(G)),
\]
where $\alpha_1$ and $\alpha_3$ are dropped from the notation to enhance readability. In particular, if we postcompose it with the map induced by $S_2(\cE) \to \cE\colon (X,Y,Z) \mapsto Y$, we get a homotopy
\[
    \TrMod_\cC(\iota_2,\alpha_2) \simeq \TrMod_\cC(\iota_1,\alpha_1) + \TrMod_\cC(\iota_3,\alpha_3)
\]
as desired.
\end{proof}

As a consequence of the above additivity property of the $\cC$-linear trace functor, we immediately obtain the main result of this subsection: the additivity of generalized traces in $\cC$.

\begin{thm}[Additivity of generalized traces, cf.\ \cite{may2001additivity,PSLinearity,ramzi2021additivity}]
\label{thm:Additivity_Traces}
    Let $\cC$ be a stable presentably symmetric monoidal $\infty$-category, and let $X_1 \xrightarrow{\varphi} X_2 \xrightarrow{\psi} X_3$ be a bifiber sequence of dualizable objects in $\cC$. Let $f_i\colon Z \tensor X_i \to X_i \tensor Y$ be morphisms fitting in bifiber sequence in $\cC^{[1]}$ of the form
    \[
    \begin{tikzcd}
        Z \tensor X_1 \dar[swap]{f_1} \rar{1_Z \tensor \varphi} & Z \tensor X_2 \dar{f_2} \rar{1_Z \tensor \psi} & Z \tensor X_3 \dar{f_3} \\
        X_1 \tensor Y \rar{\varphi \tensor 1_Y} & X_2 \tensor Y \rar{\psi \tensor 1_Y} & X_3 \tensor Y,
    \end{tikzcd}
    \]
    where the top and bottom sequences are obtained from the original bifiber sequence by tensoring with $Z$ resp.\ $Y$. Then there is an equivalence
    \[
        \gentr{f_2}{X_2} \simeq \gentr{f_1}{X_1} + \gentr{f_3}{X_3} \qin \Map_{\cC}(Z,Y).
    \]
\end{thm}
\begin{proof}
Recall from \Cref{THH_Trace} that the generalized trace $\gentr{f_i}{X_i}\colon Z \to Y$ of the generalized endomorphism $f_i\colon Z \tensor X_i \to X_i \tensor Y$ is the result of applying the $\cC$-linear trace functor $\TrMod_{\cC}$ to the morphism $(X_i,f_i)\colon (\cC,Z) \to (\cC,Y)$ in $\Mod^{\trl}_{\cC}$. One observes that the maps $\varphi\colon X_1 \to X_2$ and $\psi\colon X_2 \to X_3$ induce a bifiber sequence
\[
    (X_1,f_1) \xrightarrow{\varphi} (X_2,f_2) \xrightarrow{\psi} (X_3,f_3)
\]
of morphisms in $\Mod_\cC^\trl$. It then follows from \Cref{cor:Additivity_THH} that $\gentr{f_2}{X_2} \simeq \gentr{f_1}{X_1} + \gentr{f_3}{X_3}$ as desired.
\end{proof}
\begin{rmk}
Replacing $\cC$ with $\Ind(\cC^\dbl)$, we see that we can drop the presentability assumption on $\cC$. 
\end{rmk}
A special case of additivity of traces is the situation where $X_2 \simeq X_1 \oplus X_3$ is given by a direct sum. While this follows for stable $\Cc$ from \Cref{thm:Additivity_Traces}, the proof in this case is much more elementary and only requires $\Cc$ to be semiadditive, in the sense that finite products in $\Cc$ are also finite coproducts. For completeness, we give a proof.

\begin{lem}
\label{lem:Additivity_Traces}
Let $\Cc$ be a semiadditive presentably symmetric monoidal $\infty$-category. Let $X_1$ and $X_2$ be dualizable objects, and set $X := X_1 \oplus X_2$. Consider morphisms $f_i\colon Z \otimes X_i \to X_i \otimes Y$ for $i = 1,2$, and let $f\colon Z \otimes X \to X \otimes Y$ be their direct sum. Then there is an equivalence
\[
    \gentr{f}{X} \simeq \gentr{f_1}{X_1} + \gentr{f_2}{X_2} \qin \Map_{\Cc}(Z,Y).
\]
\end{lem}
\begin{proof}
The object $X = X_1 \oplus X_2$ is dualizable, with duality data given by
\begin{align*}
    \begin{pmatrix}\coev_{X_1} & 0 \\ 0 & \coev_{X_2}\end{pmatrix} \colon \unit \xrightarrow{} (X_1 \oplus X_2) \otimes (X_1^\vee \oplus X_2^\vee),\\
    \begin{pmatrix}\ev_{X_1} & 0 \\ 0 & \ev_{X_2}\end{pmatrix} \colon (X_1 \oplus X_2) \otimes (X_1^\vee \oplus X_2^\vee) \xrightarrow{} \unit.
\end{align*}
The result now follows from spelling out the definition of the generalized trace and the definition of addition of morphisms in $\cC$.
\end{proof}

\section{The Becker-Gottlieb transfer}
\label{sec:BeckerGottlieb}

Let $f\colon A \to B$ be a map of spaces. By functoriality of the suspension spectrum functor, there is an induced map of spectra $\SS[f]\colon \SS[A] \to \SS[B]$. When the fibers of $f$ are compact, there is also a `wrong-way' map $f^!\colon \SS[B] \to \SS[A]$, discovered by Becker and Gottlieb \cite{BG1975,BG} and now known as the \textit{Becker-Gottlieb transfer}. While the first construction in \cite{BG1975} was geometric in nature and only for fiber bundles of smooth manifolds, Becker and Gottlieb realized in \cite{BG} that their transfer map has a description purely in terms of duality data. In modern language: it is a special case of a generalized trace.

Because of the formal nature of generalized traces, the definition of the Becker-Gottlieb transfer makes sense in more general settings than spectra and maps of spaces with compact fibers. In \cite{dwyer1996transfer}, Dwyer introduces for every ring spectrum $R$ the notion of an \textit{$R$-small fibration} $f\colon A \to B$, and constructs for every such fibration a transfer map of spectra of the form $R[B] \to R[A]$. More generally, if $\cC$ is an arbitrary presentably symmetric monoidal $\infty$-category, then there is for every $\cC$-adjointable map $f\colon A \to B$ a Becker-Gottlieb transfer
\[
    f^!\colon \unit[B] \to \unit[A].
\]
When $\Cc = \RMod_R$ for a commutative ring spectrum $R$, this recovers Dwyer's construction in the commutative case.

In this section, we will study these Becker-Gottlieb transfers utilizing the perspective on generalized traces developed earlier. We start by recalling the definition and the basic properties of Becker-Gottlieb transfers in \Cref{subsec:Definition_BG_Transfers}. In \Cref{subsec:LM}, we will apply the methods from the previous sections to give a short proof of a theorem by John Lind and Cary Malkiewich \cite{LM} which expresses the Becker-Gottlieb transfer $f^!\colon \unit[B] \to \unit[A]$ as the composite of the inclusion of constant loops $c\colon \unit[B] \to \unit[LB]$, the free loop transfer $\TrMod_{\Cc}(f^*)\colon \unit[LB] \to \unit[LA]$ and the evaluation map $e\colon \unit[LA] \to \unit[A]$, see \Cref{thm:LM}. In \Cref{sec:composability_BG}, we discuss the question of composability of Becker-Gottlieb transfers, and prove several instances where it holds. A simple counterexample when $\cC$ is the $\infty$-category of rational vector spaces shows that this cannot be expected to hold in full generality. We end the section in \Cref{subsec:Relative_Free_Loops} by explaining how keeping track of the naturality in the map $f$ formally leads to a refinement of \Cref{thm:LM}.

Throughout this section, we fix a presentably symmetric monoidal $\infty$-category $\cC$.

\subsection{Definition and basic properties}
\label{subsec:Definition_BG_Transfers}

For a space $A$, the diagonal map $\Delta_A\colon A \to A \times A$ induces a map in $\cC$ of the form
\[
    \Delta_A\colon \unit[A] \xrightarrow{} \unit[A \times A] \simeq \unit[A] \tensor \unit[A] \qin \Cc.
\]
If $A$ is $\cC$-adjointable, then the object $\unit[A] \simeq A_!A^*\unit$ is a dualizable object by \Cref{prop:ColimIsDbl}, and thus we may form the generalized trace of this diagonal map. The resulting map, denoted
\[
    A^!\colon \unit \to \unit[A]\qin \cC,
\]
is called the \textit{Becker-Gottlieb transfer} of $A$. 

For a relative version of this construction, consider a $\cC$-adjointable map of spaces $f\colon A \to B$. For every fiber $A_b$ of $f$, the object $\one[A_b]$ is dualizable by \Cref{prop:ColimIsDbl}, so that the above construction gives a $B$-indexed family of Becker-Gottlieb transfers $A_b^!\colon \unit \to \unit[A_b]$ in $\cC$. Passing to colimits over $B$ then gives a map
\[
    \unit[B] = \colim_{b \in B} \unit \to \colim_{b \in B} \unit[A_b] = \unit[A] \qin \Cc,
\]
called the \textit{Becker-Gottlieb transfer of $f$}. 

Throughout this section, we will frequently use arguments involving space-indexed families of objects or morphisms, and thus it will be convenient to set up some special notation for this. For a space $B$, consider the composite functor
\[
    \unit_B[-]\colon \Spc_{/B} \iso \Spc^B \xrightarrow{\unit[-]} \Cc^B,
\]
where the first equivalence is the straightening equivalence and the second is given by post-composition with $\unit[-]\colon \Spc \to \Cc$. It preserves colimits and is naturally symmetric monoidal. We think of the functor $\unit_B[-]\colon \Spc_{/B} \to \Cc^B$ as the $B$-parameterized analogue of the functor $\unit[-]\colon \Spc \to \Cc$. Observe that it is compatible with restriction and pushforward: if $\beta\colon B' \to B$ is a map of spaces, then the following diagrams naturally commute:
\[
\begin{tikzcd}
    \Spc_{/B} \ar{rr}{\unit_B[-]} \dar[swap]{\beta^*} && \Cc^B \dar{\beta^*} \\
    \Spc_{/B'} \ar{rr}{\unit_{B'}[-]} && \Cc^{B'},
\end{tikzcd}
\qquad \qquad
\begin{tikzcd}
    \Spc_{/B'} \ar{rr}{\unit_{B'}[-]} \dar[swap]{\beta \circ -} && \Cc^{B'} \dar{\beta_!} \\
    \Spc_{/B} \ar{rr}{\unit_B[-]} && \Cc^B.
\end{tikzcd}
\]
Applying this to $\beta = f$, we see in particular that $\unit_B[A] \simeq f_!\unit_A$ for every space $A$ over $B$. Applying it to the map $B \to \pt$, we see that $B_!\unit_B[A] = \unit[A]$.

\begin{war}
The notation $\unit_B[A]$ is abusive, as this object depends on a map of spaces $f\colon A \to B$: at a point $b \in B$ it is given by $\unit[A_b]$. We emphasize that this is different from the object $B^*(\unit[A]) \in \Cc^B$, which does not depend on $f$: it is given by $\unit[A]$ at every point $b \in B$.
\end{war}

We may rewrite the definition of the Becker-Gottlieb transfer in terms of this functor $\unit_B[-]$ as follows. Observe that for a $\cC$-adjointable map $f\colon A \to B$, the object $\unit_B[A] \simeq f_!\unit_A \in \Cc^B$ is dualizable by \Cref{prop:ColimIsDbl}. Furthermore, the diagonal map $\Delta_A\colon A \to A \times_B A$ in $\Spc_{/B}$ induces a map
\[
    \Delta_A\colon \unit_B[A] \to \unit_B[A] \otimes \unit_B[A] \qin \Cc^B
\]
of which we can take the generalized trace.

\begin{defn}
\label{def:BeckerGottlieb}
Let $f\colon A \to B$ be a $\cC$-adjointable map of spaces. We define the \textit{Becker-Gottlieb pretransfer}
\[
    f^{!,\pre}\colon \unit_B \to \unit_B[A] \qin \cC^B
\]
as the generalized trace in $\cC^B$ of the diagonal map $\Delta_A\colon \unit_B[A] \to \unit_B[A] \tensor \unit_B[A]$.

We define the \textit{Becker-Gottlieb transfer} $f^!\colon \unit[B] \to \unit[A]$ as the composite
\[
    \unit[B] = B_!\unit_B \xrightarrow{B_!f^{!,\pre}} B_!\unit_B[A] = \unit[A] \qin \cC,
\]
obtained from the pretransfer $f^{!,\pre}$ by applying the colimit functor $B_!\colon \cC^B \to \cC$.
\end{defn}

For completeness, we state and prove the main basic properties of Becker-Gottlieb transfers, which can be found for example in \cite[Section~2]{dwyer1996transfer} and \cite[p. 189]{LMS}.

\begin{lemma}
\label{lem:Basic_Properties_BG_Transfer}
\begin{enumerate}[(1)]
 \item (Identity) The transfer for the identity $\id\colon B \to B$ is the identity $\unit[B] \to \unit[B]$;
    \item \label{it:BGNaturality} (Naturality) Consider a pullback square of spaces as on the left, where the maps $f$ and $f'$ are $\cC$-adjointable. Then there is a commutative diagram as on the right:
    \[
    \begin{tikzcd}
        A' \dar[swap]{f'} \drar[pullback] \rar{\alpha} & A \dar{f} \\
        B' \rar{\beta} & B
    \end{tikzcd}
    \qquad \qquad \qquad \qquad
    \begin{tikzcd}
        \unit[A'] \rar{\alpha} & \unit[A] \\
        \unit[B'] \uar{{f'}^!} \rar{\beta} & \unit[B]; \uar[swap]{f^!}
    \end{tikzcd}
    \]
   
    \item \label{it:BGProduct} (Product) Let $f\colon A \to B$ and $f'\colon A' \to B'$ be $\cC$-adjointable maps of spaces. Then the transfer map $(f \times f')^!\colon \unit[B \times B'] \to \unit[A \times A']$ of the product map $f \times f'\colon A \times A' \to B \times B'$ is equivalent to $f^! \tensor {f'}^!\colon \unit[B] \tensor \unit[B'] \to \unit[A] \tensor \unit[A']$.
\end{enumerate}
\end{lemma}
\begin{proof}
Part (1) is obvious. For part (2), we will first relate the pretransfer of $f'$ with the pretransfer of $f$. Since $A'$ is a pullback of $A$ along $\beta\colon B' \to B$, we get an equivalence $\unit_{B'}[A'] \simeq \beta^*\unit_B[A]$ in $\Cc^{B'}$. Under this equivalence, the diagonal of $\unit_{B'}[A']$ is pulled back from the diagonal of $\unit_B[A]$:
\begin{equation*}
\begin{tikzcd}
    \unit_{B'}[A'] \dar{\simeq} \ar{r}{\Delta_{A'}} & \unit_{B'}[A'] \tensor \unit_{B'}[A'] \dar{\simeq} \\
    \beta^*\unit_B[A] \rar{\beta^*(\Delta_D)} & \beta^*(\unit_B[A] \tensor \unit_B[A]).
\end{tikzcd}
\end{equation*}
By symmetric monoidality, the functor $\beta^*\colon \Cc^B \to \Cc^{B'}$ preserves generalized traces, and thus we get an equivalence $\preBG{f'} \simeq \beta^*\preBG{f}$. 

Next, consider the following diagram:
\[
\begin{tikzcd}
    \unit[A'] \ar[bend left=15]{rrr}{\alpha} \rar[equal] & B_!\beta_!\beta^*\unit_B[A] \rar{c^!_{\beta}} & B_!\unit_B[A] \rar[equal] & \unit[A] \\
    \unit[B'] \ar[bend right=12, swap]{rrr}{\beta} \rar[equal] \uar{{f'}^!} & B_!\beta_!\beta^*\unit_B[B] \rar{c^!_{\beta}} \uar{B_!\beta_!\beta^*\preBG{f}} & B_!\unit_B[B] \rar[equal] \uar{B_!\preBG{f}} & \unit[B] \uar[swap]{f^!}.
\end{tikzcd}
\]
Here $c^!_{\beta}\colon \beta_!\beta^* \to \id$ denotes the counit of the adjunction $\beta_! \dashv \beta^*$. The right square commutes by definition, while the left square commutes by the description of the pretransfer of $f'$ in terms of the pretransfer of $f$. The middle square commutes by naturality. One checks that the top and bottom composites of the diagram are induced by $\alpha$ and $\beta$, respectively. This proves the claim.

For part (3), we will again first describe the pretransfer of $f \times f'$ in terms of the pretransfers of $f$ and $f'$. Consider the \textit{external tensor product} $- \boxtimes -\colon \cC^B \times \cC^{B'} \to \cC^{B \times B'}$, defined as the composite
\[
    \cC^B \times \cC^{B'} \xrightarrow{- \times -} (\Cc \times \Cc)^{B \times B'} \xrightarrow{- \tensor -} \Cc^{B \times B'}.
\]
Informally, it sends a pair $(X,Y)$ to the family $\{X_b \tensor Y_{b'}\}_{(b,b') \in B \times B'}$. Observe that this is a symmetric monoidal functor, and that it preserves colimits in both variables. Furthermore, it is compatible with left Kan extension, in the sense that the following square commutes:
\[
\begin{tikzcd}
    \Cc^A \times \Cc^{A'} \dar[swap]{f_! \times f'_!} \rar{\boxtimes} & \Cc^{A \times A'} \dar{(f \times f')_!} \\
    \Cc^B \times \Cc^{B'} \rar{\boxtimes} & \Cc^{B \times B'}.
\end{tikzcd}
\]
It follows that the object $\unit_{B \times B'}[A \times A'] \in \cC^{B \times B'}$ is equivalent to the external tensor product of the objects $\unit_B[A] \in \cC^B$ and $\unit_{B'}[A'] \in \cC^{B'}$, and that the diagonal of $\unit_{B \times B'}[A \times A']$ is obtained by taking the external tensor product of the diagonals of $\unit_B[A]$ and $\unit_{B'}[A']$. Since symmetric monoidal functors preserve generalized traces, see \Cref{rmk:GeneralizedTraceFunctorial}, we get that the pretransfer $(f \times f')^{!,\pre}$ is equivalent to $f^{!,\pre} \boxtimes {f'}^{!,\pre}$. The claim about Becker-Gottlieb transfers now follows from applying $(B \times B')_!\colon \cC^{B \times B'} \to \cC$, using the above commutative diagram for the map $B \times B' \to \pt \times \pt$.
\end{proof}

From the additivity of generalized traces, we may deduce additivity of Becker-Gottlieb transfers.

\begin{lem}[Additivity of Becker-Gottlieb transfers] \label{lem:Additivity_BG_Transfers}
Let $f\colon A \to B$ be a map of spaces.
\begin{enumerate}[(1)]
    \item Assume that $\cC$ is a semiadditive $\infty$-category, and assume that $A$ can be decomposed as a disjoint union $A = A_1 \sqcup A_2$. For $i = 1,2$, let $j_i\colon A_i \hookrightarrow A$ denote the inclusion and let $f_i := f \circ j_i\colon A_i \to B$. If the maps $f_1$ and $f_2$ are $\cC$-adjointable, then $f$ is also $\cC$-adjointable, and there is an equivalence
    \[
        f^! \simeq j_1 \circ f_1^! + j_2 \circ f_2^! \qin \Map_{\Cc}(\unit[B],\unit[A]).
    \]
    \item Assume that $\cC$ is a stable $\infty$-category, and assume that $A$ is given as a pushout of the form
    \[
    \begin{tikzcd}
        A_3 \rar{\alpha} \dar[swap]{\beta} \drar[pushout] & A_1 \dar{j_1} \\
        A_2 \rar{j_2} & A.
    \end{tikzcd}
    \]
    Set $j_3\colon A_3\to A$ to be the diagonal composite, and let $f_i := f \circ j_i\colon A_i \to B$ for $i = 1,2,3$. If each $f_i$ is $\cC$-adjointable, then $f$ is also $\cC$-adjointable, and there is an equivalence
    \[
        f^! \simeq j_1 \circ f_1^! + j_2 \circ f_2^! - j_3 \circ f_3^! \qin \Map_{\cC}(\unit[B],\unit[A]).
    \]
\end{enumerate}
\end{lem}
\begin{proof}
In part (1), the fact that $f$ is $\cC$-adjointable follows from the fact that the diagonal functor $\Delta\colon \Cc \to \Cc \times \Cc$ is an internal left adjoint, as semiadditivity implies that the left and right adjoint of $\Delta$ are equivalent. For the second part, we we will prove the stronger claim that there is an equivalence
\[
    \preBG{f} \simeq j_! \circ \preBG{f}_1 + j_2 \circ \preBG{f}_2
\]
in $\Map_{\Cc^B}(\unit_B,\unit_B[A])$; the desired claim will follow by applying $B_!\colon \cC^B \to \cC$. For $i = 1,2$, we get from \Cref{Trace_Functorial} that the map $j_i \circ \preBG{f}_i\colon \unit_B \to \unit_B[A]$ is the generalized trace in $\Cc^B$ of the composite
\[
    (1,j_i)\colon \unit_B[A_i] \xrightarrow{\Delta_{A_i}} \unit_B[A_i] \otimes \unit_B[A_i] \xrightarrow{1 \otimes j_i} \unit_B[A_i] \otimes \unit_B[A],
\]
induced by the map $(1,j_i)\colon A_i \to A_i \times_B A$ in $\Spc_{/B}$. The direct sum of these two maps is the diagonal of $\unit_B[A]$, hence the claim follows from \Cref{lem:Additivity_Traces}.

In part (2), the fact that $f$ is $\cC$-adjointable follows from \Cref{lem:Stable_Ambi_Pushouts}. For the second part, we will prove the stronger claim that there is an equivalence
\[
    \preBG{f} \simeq j_1 \circ \preBG{f_1} + j_2 \circ \preBG{f_2} - j_3 \circ \preBG{f_3}
\]
in $\Map_{\Cc^B}(\unit_B,\unit_B[A])$. As before, the desired claim follows by applying $B_!\colon \cC^B \to \cC$. We follow the proof of \cite[IV.2.9, p.184]{LMS}. Since the functor $\unit_B[-]\colon \Spc_{/B} \to \cC^B$ preserves colimits and $\cC^B$ is an additive $\infty$-category, we obtain a cofiber sequence
\[
    \unit_B[A_3] \xrightarrow{(\alpha,-\beta)} \unit_B[A_1] \oplus \unit_B[A_2] \xrightarrow{(j_1,j_2)} \unit_B[A]
\]
in $\Cc^B$. Tensoring with $\unit_B[A]$ gives rise to a second cofiber sequence. Now consider the following morphism of cofiber sequences:
\[\begin{tikzcd}
	{\unit_B[A_3]} && {\unit_B[A_1] \oplus \unit_B[A_2]} && {\unit_B[A]} \\
	{\unit_B[A_3] \otimes \unit_B[A]} && {(\unit_B[A_1] \oplus \unit_B[A_2]) \otimes \unit_B[A]} && {\unit_B[A] \otimes \unit_B[A].}
	\arrow["{(1,j_3)}", from=1-1, to=2-1]
	\arrow["{(1,j_1) \oplus (1,j_2)}", from=1-3, to=2-3]
	\arrow["\Delta_A", from=1-5, to=2-5]
	\arrow["{(\alpha,-\beta)}", from=1-1, to=1-3]
	\arrow["{(\alpha,-\beta) \otimes 1}", from=2-1, to=2-3]
	\arrow["{(j_1,j_2)}", from=1-3, to=1-5]
	\arrow["{(j_1,j_2) \otimes 1}", from=2-3, to=2-5]
\end{tikzcd}\]
As in part (1), the generalized trace of the maps 
\[
    (1,j_i)\colon \unit_B[A_i] \to \unit_B[A_i] \otimes \unit_B[A]
\] 
is $j_i \circ \preBG{f}_i\colon \unit \to \unit_B[A]$, and thus by \Cref{lem:Additivity_Traces} the generalized trace of the middle vertical map is the sum $j_1 \circ \preBG{f}_1 + j_2 \circ \preBG{f}_1$. The claim thus follows from additivity of generalized traces, \Cref{thm:Additivity_Traces}.
\end{proof}

\begin{rem}
When $\cC = \Sp$ is the $\infty$-category of spectra, the four properties of the Becker-Gottlieb transfer listed above (naturality, identity, product and additivity) can be used to uniquely characterize the Becker-Gottlieb transfers on certain fiber bundles with compact fibers, see \cite{beckerSchultz1998axioms}. In similar spirit is the result of \cite{lewis1983uniqueness}, where axioms are introduced to characterize the Becker-Gottlieb transfers on the level of cohomology.
\end{rem}

\subsection{Description in terms of the free loop transfer}
\label{subsec:LM}
Let $f\colon A \to B$ be a map of spaces with compact fibers. It was proved by Lind and Malkiewich \cite[Theorem 1.2]{LM} that the Becker-Gottlieb transfer $f^!\colon \SS[B] \to \SS[A]$ can be obtained from the free loop transfer $\TrMod_{\cC}(f^*)\colon \SS[LB] \to \SS[LA]$ by precomposing with the constant loop map $c\colon \SS[B] \to \SS[LB]$ and postcomposing with the evaluation map $e\colon \SS[LA] \to \SS[A]$. A proof of this statement for a smaller class of fibrations was given in \cite{DorabialaJohnson} using the
Becker-Schultz axiomatization of the Becker-Gottlieb transfer. The equivalence of these maps after composing with the projection $\SS[A] \to \SS$ from $A$ to the point was obtained by \cite{douglas2005trace}. 

The result by Lind and Malkiewich admits a quick and elegant proof using the methods developed in this article, which works in the generality of an arbitrary presentably symmetric monoidal $\infty$-category $\cC$. The main input is the simple but crucial observation that we can relate the Becker-Gottlieb transfer with the generalized character of a well-chosen map $\delta(f)$ in $\Cc^B$.

\begin{defn}
Let $f\colon A \to B$ be a map of spaces. We define the map
\[
    \delta(f)\colon \unit_B[A] \to \unit_B[A] \otimes B^*\unit[A] \qin \cC^B
\]
as the image under $\unit_B[-]\colon \Spc_{/B} \to \cC^B$ of the diagonal map
\[
    A \xrightarrow{\Delta_A} A \times A \simeq A \times_B (A \times B) \qin \Spc_{/B},
\]
where we regard $A \times A$ as living over $B$ via its first factor. Over a point $b \in B$, it is given by the generalized endomorphism
\[
    \one[A_b] \oto{\:\Delta_{A_b}\:} 
    \one[A_b]\otimes \one[A_b] \oto{1\otimes \iota_b}
    \one[A_b] \otimes \one[A],
\]
where $\iota_b\colon A_b \to A$ is the inclusion of the fiber.
\end{defn}

If $f\colon A \to B$ is $\cC$-adjointable, then $\unit_B[A]$ is dualizable in $\Cc^B$ 
(\Cref{prop:ColimIsDbl}), and thus we obtain a generalized character map
\[
    \chi_{\delta(f)}\colon \unit[LB] \to \unit[A].
\]
The Becker-Gottlieb transfer of $f$ can be recovered from this character map as follows.
\begin{prop}\label{BG_Char}
    For a $\cC$-adjointable map $f\colon A\to B$, the following diagram commutes:
    \[\begin{tikzcd}
    	{\one[LB]} \\
    	{\one[B]} && {\one[A].}
    	\arrow["c", from=2-1, to=1-1]
    	\arrow["f^!"', from=2-1, to=2-3]
    	\arrow["\chi_{\delta(f)}", from=1-1, to=2-3]
    \end{tikzcd}\]
\end{prop}
\begin{proof}
Observe that the map $\delta(f)$ can be written as the composite
\[
    \delta(f)\colon \unit_B[A] \xrightarrow{\Delta_A} \unit_B[A] \tensor \unit_B[A] \xrightarrow{1 \otimes u_B} \unit_B[A] \tensor B^*\unit[A] \qin \Cc^B,
\]
where $u_B$ denotes the map $\unit_B[A] \to \unit_B[A \times B] \simeq B^*\unit[A]$ induced by the map $(1,f)\colon A \to A \times B$ over $B$. By \Cref{Trace_Functorial}, it follows that the generalized trace of $\delta(f)$ in $\cC^B$ is given by the composite
\[
    \unit_B \xrightarrow{\preBG{f}} \unit_B[A] \xrightarrow{u_B} B^*\unit[A], 
\]
which adjoints over to the map $f^! \colon \unit[B] \to \unit[A]$. The claim thus follows from \Cref{Char_Trace}.
\end{proof}

Our next aim is a better understanding of the character of $\delta(f)$. We start with the case where $f$ is the identity on a space $A$.

\begin{prop}
\label{prop:Character_Of_Identity}
    The character $\chi_{\delta(\id_A)}\colon \unit[LA] \to \unit[A]$ is the evaluation map $e\colon \unit[LA] \to \unit[A]$.
\end{prop}
\begin{proof}
The map $\delta(\id_A)$ is given by the $A$-indexed family of generalized traces $\unit[\{a\}] \to \unit[\{a\}] \tensor \unit[A]$ induced by the inclusions $\iota_a\colon \{a\} \hookrightarrow A$. By \Cref{ex:Generalized_Trace_Unit}, the generalized character of this map is simply the map $a\colon \unit \to \unit[A]$. It thus follows from the coherent character formula, \Cref{prop:Character_Formula_Coherent}, that the character $\chi_{\delta(\id_A)}\colon \unit[LA] \to \one[A]$ is the mate of the map $LA \to \Map_{\Cc}(\unit,\unit[A])$ which sends a free loop $\gamma \in LA$ to the inclusion $\gamma(\pt)\colon \unit \to \unit[A]$ of the basepoint of $\gamma$. The claim follows by adjunction.
\end{proof}

To understand the character $\chi_{\delta(f)}$ for general maps $f\colon A \to B$, we observe that for a second map $g\colon B \to C$ the generalized endomorphisms $\delta(f)$ and $\delta(gf)$ are closely related.

\begin{prop}
\label{prop:Delta_Induction}
Consider two composable maps of spaces $f \colon A \to B$ and $g\colon B \to C$. Then there is an equivalence of generalized endomorphisms
\[
    \delta(gf) \simeq \Ind_g(\delta(f))
    \qin \Map_{\Cc^C}(\unit_C[A],\unit_C[A] \otimes C^*\unit[A]).
\]
\end{prop}
\begin{proof}
Spelling out the definitions, one observes that both maps are given by applying the functor $\unit_C[-]\colon \Spc_{/C} \to \Cc^C$ to the diagonal map $A \xrightarrow{\Delta_A} A \times A \simeq A \times_C (A \times C)$.
\end{proof}

As a consequence, we obtain a description of the character $\chi_{\delta(gf)}$ in terms of the character $\chi_{\delta(f)}$ and the free loop transfer of $g$:

\begin{cor}
\label{cor:Character_Induction}
    Let $f\colon A\to B$ and $g\colon B \to C$ be $\cC$-adjointable maps of spaces. Then the following diagram commutes:
    \[\begin{tikzcd}
    	{\one[LC]} && {\one[LB]} \\
    	&& {\one[A].}
    	\arrow["\chi_{\delta(gf)}"', from=1-1, to=2-3]
    	\arrow["\chi_{\delta(f)}", from=1-3, to=2-3]
    	\arrow["{\TrMod_{\cC}(g^*)}", from=1-1, to=1-3]
    \end{tikzcd}\]
\end{cor}
\begin{proof}
By \Cref{prop:Delta_Induction}, this is immediate from the induced character formula (\Cref{Ind_Char}).
\end{proof}

In particular, we obtain a description of the character $\chi_{\delta(f)}$ in terms of the free loop transfer of $f$:
\begin{cor}
\label{cor:Character_Diagonal}
    For a $\cC$-adjointable map of spaces $f\colon A\to B$, the following diagram commutes:
    \[\begin{tikzcd}
    	{\one[LB]} && {\one[LA]} \\
    	&& {\one[A].}
    	\arrow["\chi_{\delta(f)}"', from=1-1, to=2-3]
    	\arrow["e", from=1-3, to=2-3]
    	\arrow["{\TrMod_{\cC}(f^*)}", from=1-1, to=1-3]
    \end{tikzcd}\]
\end{cor}
\begin{proof}
Applying \Cref{cor:Character_Induction} to the maps $\id_A\colon A \to A$ and $f\colon A \to B$, this is immediate from the equivalence $\chi_{\delta(\id_A)} \simeq e$ of \Cref{prop:Character_Of_Identity}.
\end{proof}

We thus obtain the main theorem of this section.
\begin{thm}[cf.\ Lind-Malkiewich {\cite[Theorem 1.2]{LM}}]
\label{thm:LM}
For a $\cC$-adjointable map of spaces $f\colon A \to B$, the following diagram commutes:
    \[\begin{tikzcd}
    	{\one[LB]} && {\one[LA]} \\
    	{\one[B]} && {\one[A].}
    	\arrow["\chi_{\delta(f)}"{description}, dashed,from=1-1, to=2-3]
    	\arrow["e", from=1-3, to=2-3]
    	\arrow["\TrMod_{\cC}(f^*)", from=1-1, to=1-3]
    	\arrow["c", from=2-1, to=1-1]
    	\arrow["{f^!}"', from=2-1, to=2-3]
    \end{tikzcd}\]
\end{thm}
\begin{proof}
Combine \Cref{BG_Char} and \Cref{cor:Character_Diagonal}.
\end{proof}

\subsection{Composability of Becker-Gottlieb transfers}
\label{sec:composability_BG}
Consider two composable maps of spaces $f\colon A \to B$ and $g\colon B \to C$ and assume both $f$ and $g$ are $\cC$-adjointable. Since the composite $gf\colon A \to C$ is again $\cC$-adjointable, we may form the Becker-Gottlieb transfers $f^!\colon \one[B] \to \one[A]$, $g^!\colon \one[C] \to \one[B]$ and $(gf)^!\colon \one[C] \to \one[A]$ in $\cC$. It is natural to wonder how the map $(gf)^!$ relates to the composite $f^! \circ g^!$.

\begin{defn}
We will say that $f$ and $g$ have \textit{composable Becker-Gottlieb transfers}\footnote{This is sometimes referred to as \textit{functoriality} of the Becker-Gottlieb transfer.}  if the following triangle commutes in $\cC$:
\[
    \begin{tikzcd}
        & \unit[B] \drar{f^!} \\
        \unit[C] \urar{g^!} \ar{rr}{(gf)^!} && \unit[A].
    \end{tikzcd}
\]
\end{defn}

The naive guess that Becker-Gottlieb transfers are composable in full generality turns out to be incorrect.

\begin{counterexample}
Let $\cC = \Vect_{\QQ}$ be the category of rational vector spaces, and consider the maps $f\colon \pt \to BH$ and $g\colon BH \to \pt$, where $H$ is a non-trivial finite group. By 1-semiadditivity of $\Vect_{\QQ}$, the maps $f$ and $g$ are $\cC$-adjointable, see \Cref{ex:AdjointableFromAmbi}. We claim that the composition $f^! \circ g^!\colon \QQ \to \QQ$ is not given by the identity but rather by multiplication by the group order $\lvert H \rvert$ of $H$. To see this, observe that the map $g$ induces an equivalence $\QQ[BH] \simeq \QQ_H \iso \QQ$. Since this equivalence is compatible with the diagonal and with the evaluation and coevaluation, it follows that the composite $\QQ \oto{g^!} \QQ_H \oto{\QQ[g]} \QQ$ is the identity. But note that the $H$-equivariant pretransfer $\QQ \to \QQ[H]$ of $f$ is given by $1 \mapsto \sum_{h \in H}h$, so that passing to $H$-orbits gives a map $f^!\colon \QQ_H \to \QQ$ which is $\lvert H \rvert$ times the map $\QQ_H \oto{\QQ[g]} \QQ$, proving the claim.
\end{counterexample}

\begin{war}
It is claimed in \cite{kleinMalkiewich2018transfer} that when $\cC$ is the $\infty$-category of spectra and the maps $f$ and $g$ have compact fibers,\footnote{In \cite{kleinMalkiewich2018transfer}, such maps $f$ and $g$ are called `finitely dominated'.} the Becker-Gottlieb transfers of $f$ and $g$ are composable. Unfortunately, the proof in \cite{kleinMalkiewich2018transfer} contains a mistake: the diagram (13) on page 1135 does not commute, which renders the proof invalid. As this mistake appeared to be unfixable, Klein and Malkiewich have published a corrigendum \cite{kleinMalkiewich2022Corrigendum} retracting the main theorem of \cite{kleinMalkiewich2018transfer}. To the best of our knowledge it is currently an open problem whether or not the Becker-Gottlieb transfers in spectra of maps with compact fibers are composable in full generality.
\end{war}

To see the subtlety of this problem, we may use the description of Becker-Gottlieb transfers established in the previous section. Consider the following diagram:
\[\begin{tikzcd}
	{\one[LC]} && {\one[LB]} && {\one[LB]} && {\one[LA]} \\
	{\one[C]} &&& {\one[B]} &&& {\one[A].}
	\arrow["c", from=2-1, to=1-1]
	\arrow["{\TrMod_{\cC}(g^*)}"', from=1-1, to=1-3]
	\arrow["{\TrMod_{\cC}(f^*)}"', from=1-5, to=1-7]
	\arrow["e", from=1-3, to=2-4]
	\arrow["c", from=2-4, to=1-5]
	\arrow["e", from=1-7, to=2-7]
	\arrow["{g^!}"', from=2-1, to=2-4]
	\arrow["{f^!}"', from=2-4, to=2-7]
	\arrow["{\TrMod_{\cC}((gf)^*)}"{description}, curve={height=-30pt}, from=1-1, to=1-7]
	\arrow[Rightarrow, no head, from=1-3, to=1-5]
\end{tikzcd}\]
From \Cref{thm:LM} we obtain that the two trapezoids commute and that the composite along the top of the diagram is the Becker-Gottlieb transfer $(gf)^!\colon \unit[C] \to \unit[A]$. The top part of the diagram commutes by functoriality of the $\Cc$-linear trace. However, the middle triangle of the diagram \textit{does not commute}: it replaces each free loop in $B$ by the constant loop on its basepoint. Since the diagonal of the right square is the character $\chi_{\delta(f)}\colon \unit[LB] \to \unit[A]$, we obtain the following two descriptions of $(gf)^!$ and $f^!g^!$:

\begin{cor}
\label{cor:ObstructionComposability}
Let $f\colon A\to B$ and $g\colon B\to C$ be $\cC$-adjointable maps of spaces. The Becker-Gottlieb transfer $(gf)^!$ of the composite $gf$ is homotopic to the following composite: 
\[ 
\one[C]\xrightarrow{c}\one[LC] \xrightarrow{\TrMod_\cC(g^*)} \begin{tikzcd}[column sep = 44pt] \!\!\!\one[LB] \rar[equal] & \one[LB]\!\!\! \end{tikzcd} \xrightarrow{\chi_{\delta(f)}} \one[A].
\]
On the other hand, the composite of the Becker-Gottlieb transfers $f^!$ and $g^!$ is given by the composite
\[
    \one[C]\xrightarrow{c}\one[LC] \xrightarrow{\TrMod_\cC(g^*)} \one[LB]\xrightarrow{e} \one[B] \xrightarrow{c} \one[LB] \xrightarrow{\chi_{\delta(f)}} \one[A].
\]
\end{cor}

\Cref{cor:ObstructionComposability} gives us a concrete obstruction for composability of Becker-Gottlieb transfers. This has been used for instance by \cite{KMR} to prove composability for finitely dominated maps at the level of $\pi_0$:
\begin{prop}[{\cite[Theorem~B]{KMR}}]
Let $f\colon A \to B$ and $g\colon B \to C$ be maps of spaces with compact fibers. Then the Becker-Gottlieb transfers compose on $\pi_0$: the diagram
\[
\begin{tikzcd}
    \pi_0(\SS[C]) \ar[bend left = 20]{rr}{(gf)^!} \rar[swap]{g^!} & \pi_0(\SS[B]) \rar[swap]{f^!} & \pi_0(\SS[A])
\end{tikzcd}
\]
commutes.
\end{prop}

In certain situations, the Becker-Gottlieb transfers of two maps are composable even before passing to $\pi_0$. In the remainder of this subsection, we will give an overview of such results that were previously known in the literature, and provide various generalizations. The results are summarized in the following theorem:

\begin{thm}[cf.\ \cite{LMS}, \cite{williams2000bivariantRR}, \cite{LM}, \cite{kleinMalkiewich2022Corrigendum}]\label{thm:compBG}
Let $f\colon A \to B$ and $g\colon B \to C$ be $\cC$-adjointable morphisms of spaces. In the following situations, there is a homotopy $f^! \circ g^! \simeq (g \circ f)^!$ of maps $\unit[C] \to \unit[A]$:
\begin{enumerate}
    \item[(1)] The map $f\colon A \to B$ is a base change of some map $\cC$-adjointable map $B' \to C$ along $g\colon B \to C$ (\Cref{lem:Becker_Gottlieb_Vs_Pullbacks});
    \item[(1a)] The map $f\colon A \to B$ is a projection $A_0 \times B \to B$ (\Cref{cor:Becker_Gottlieb_Projection});
    \item[(2)] The free loop transfer of $g$ restricts to constant loops, in the sense of \Cref{def:RestrictToConstantLoops} (\Cref{prop:Criterion_Composability_BG});
    \item[(2a)] The $\infty$-category $\cC$ is stable and the map $g\colon B \to C$ is a smooth fiber bundle with closed manifold fibers (\Cref{cor:Composability_BG_Smooth_Bundle});
    \item[(2b)] The $\infty$-category $\cC$ is semidadditive and the map $g\colon B \to C$ is a finite covering map (\Cref{cor:Composability_BG_Finite_Covering});
    \item[(3)] The space $C$ is the classifying space of a compact Lie group $G$, and the map $g\colon B \to C$ is of the form $B'_{hG} \to \pt_{hG} \simeq BG$ for some finite $G$-CW-complex $B'$ (\Cref{cor:LMS_Equivariant_Transfers}).
\end{enumerate}
\end{thm}

Let us comment about the history of these results. The case (1) seems to be a new observation. Its consequence (1a) has been obtained using different methods by \cite{kleinMalkiewich2022Corrigendum}. The case (2) was already observed by \cite{LM}. The special case (2a) requires geometric input, which can be decuded from \cite{williams2000bivariantRR} but was reproved using different methods by \cite{LM}. When $\Cc$ is stable, the special case (2b) follows from (2a), and was also proved using different methods by \cite{kleinMalkiewich2022Corrigendum}. The case in (2b) where $\Cc$ is only assumed to be semiadditive seems to be new. The case (3) is a variant of a result proved by \cite{LMS}, who worked in the setting of genuine $G$-spectra instead. 

\subsubsection{Trivial bundles}

As a first example, we show that Becker-Gottlieb transfers compose when the map $f\colon A \to B$ is equivalent to a trivial bundle $B \times A_0 \to B$ for some space $A_0$. More generally, we will show composability of Becker-Gottlieb transfers whenever $f$ is obtained as a base change of some map along $g$. This seems to be a new observation.

\begin{lem}
\label{lem:Becker_Gottlieb_Vs_Pullbacks}
Consider a pullback diagram of spaces
\[
\begin{tikzcd}
    A \dar[swap]{f} \drar[pullback] \rar{k} & D \dar{h} \\
    B \rar{g} & C
\end{tikzcd}
\]
and assume that the maps $f$, $g$ and $h$ are $\cC$-adjointable. Then the Becker-Gottlieb transfers for $f$ and $g$ are composable: $f^! \circ g^! \simeq (g \circ f)^!$
\end{lem}

We may think of this lemma as an enhanced version of \Cref{lem:Basic_Properties_BG_Transfer}(\ref{it:BGProduct}), parameterized over the space $C$. Indeed, for $C = \pt$ we have an equivalence $A \simeq B \times D$, and the claim follows from \Cref{lem:Basic_Properties_BG_Transfer}(\ref{it:BGProduct}). By naturality of Becker-Gottlieb transfers, we see that for general $C$ the equivalence $f^!(g^!(c)) \simeq (g \circ f)^!(c)$ holds pointwise for all $c \in C$. The claim of \Cref{lem:Becker_Gottlieb_Vs_Pullbacks} is that this works naturally in $c\in C$.

\begin{proof}
Spelling out the definition of the Becker-Gottlieb transfer in terms of the pretransfer, it will suffice to show that the following two maps in $\cC^C$ are homotopic:
\begin{enumerate}
    \item[(1)] the pretransfer $\preBG{(gf)}\colon \unit_{C} \to \unit_{C}[A]$ of the composite $gf\colon A \to C$;
    \item[(2)] the composite $\unit_C \xrightarrow{\preBG{g}} \unit_{C}[B] \xrightarrow{g_!\preBG{f}} \unit_C[A]$ of the pretransfers of $g$ and $f$.
\end{enumerate}
We will show they are both homotopic to the following map:
\begin{enumerate}
    \item[(3)] the tensor product $\unit_C \simeq \unit_C \tensor \unit_C \xrightarrow{\preBG{h} \tensor \preBG{g}} \unit_{C}[D] \tensor \unit_{C}[B] \simeq \unit_C[D \times_C B] = \unit_C[A]$ of the pretransfers of $h$ and $g$.
\end{enumerate}

We start by proving that (1) and (3) are homotopic. Since the objects $\unit_C[D]$ and $\unit_C[B]$ are dualizable in $\Cc^C$, so is their tensor product $\unit_C[D] \tensor \unit_C[B] \simeq \unit_C[A]$, and the duality data may be chosen to be the tensor product of the duality data for $\unit_C[D]$ and $\unit_C[B]$. It thus remains to show that the following square commutes:
\begin{equation*}
\begin{tikzcd}
    \unit_C[A] \rar{\Delta_{A}} \dar{\simeq} & \unit_C[A] \tensor \unit_C[A] \dar{\simeq}\\
    \unit_C[D] \otimes \unit_C[B] \rar{\Delta_D \tensor \Delta_B} & \unit_C[D] \otimes \unit_C[D] \otimes \unit_C[B] \tensor \unit_C[B].
\end{tikzcd}
\end{equation*}
But this is clear, as both horizontal maps are induced by the diagonal of $A \simeq D \times_C B$ in $\Spc_{/C}$.

Next we show that (2) and (3) are homotopic. For this, we rewrite (3) as a composite
\begin{align*}
    \unit_C \xrightarrow{\preBG{g}} \unit_C[B] \simeq \unit_C \tensor \unit_C[B] \xrightarrow{\preBG{h} \tensor 1 } \unit_C[D] \tensor \unit_C[B] \simeq \unit_C[A].
\end{align*}
As we saw in the proof of \Cref{lem:Basic_Properties_BG_Transfer}(\ref{it:BGNaturality}), under the equivalence $\unit_B[A] \simeq g^*\unit_C[D]$ the pretransfer of $f$ is pulled back along $g$ from the pretransfer of $h$, in the sense that the following diagram commutes:
\[
\begin{tikzcd}
    \unit_B \rar{\preBG{f}} \dar{\simeq} & \unit_B[A] \dar{\simeq} \\
    g^*\unit_C \rar{g^*\preBG{h}} & g^*\unit_C[D].
\end{tikzcd}
\]
Identifying $\unit_C[B]$ with $g_!g^*\unit_C$, the claim thus follows from the naturality of the projection formula:
\begin{equation*}
\begin{tikzcd}
    g_!g^*\unit_C \rar{g_!g^*\preBG{h}} \dar{\simeq}[swap]{p.f.} & g_!g^*\unit_C[D] \dar{\simeq}[swap]{p.f.} \\
    \unit_C \tensor \unit_C[B] \rar{\preBG{h} \tensor 1} & \unit_C[D] \tensor \unit_C[B].
\end{tikzcd}
\end{equation*}
It follows that also the maps (1) and (2) are homotopic, finishing the proof.
\end{proof}

From \Cref{lem:Becker_Gottlieb_Vs_Pullbacks}, we immediately obtain composability of Becker-Gottlieb transfers for trivial bundles.

\begin{cor}
\label{cor:Becker_Gottlieb_Projection}
Let $f\colon A \to B$ and $g\colon B \to C$ be $\cC$-adjointable maps of spaces. Assume that $f$ is equivalent to the trivial bundle $B \times A_0 \to B$ for some $\cC$-adjointable space $A_0$. Then the Becker-Gottlieb transfers for $f$ and $g$ are composable: $f^! \circ g^! \simeq (g \circ f)^!$.
\end{cor}
\begin{proof}
The map $f$ fits in a pullback diagram of the form
\[
\begin{tikzcd}
    B \times A_0 \dar{f} \rar \drar[pullback] & C \times A_0 \dar{\pr_C} \rar \drar[pullback] & A_0 \dar \\
    B \rar{g} & C \rar & \pt,
\end{tikzcd}
\]
and thus this is a special case of \Cref{lem:Becker_Gottlieb_Vs_Pullbacks}.
\end{proof}

\begin{rem}
    When $\Cc$ is the $\infty$-category of spectra, the result of \Cref{cor:Becker_Gottlieb_Projection} has also been obtained by \cite{kleinMalkiewich2022Corrigendum} via different methods. 
\end{rem}

\subsubsection{The free loop transfer restricts to constant loops}

As noted in \cite[Remark 8.10]{LM}, the Becker-Gottlieb transfers of two $\cC$-adjointable maps $f\colon A \to B$ and $g\colon B \to C$ will always compose if $g$ satisfies the property that its free loop transfer $\TrMod_{\Cc}(g^*)\colon \unit[LC] \to \unit[LB]$ \textit{restricts to constant loops}, the sense of the following definition:

\begin{defn}
\label{def:RestrictToConstantLoops}
Let $g\colon B \to C$ be a $\cC$-adjointable map of spaces. We say that the \textit{free loop transfer of $g$ restricts to constant loops} if there exists a map $\alpha\colon \unit[C] \to \unit[B]$ making the following diagram commute:
\begin{equation*}
\begin{tikzcd}
	{\one[LC]} && {\one[LB]} \\
	{\one[C]} && {\one[B].}
	\arrow["c", from=2-1, to=1-1]
	\arrow["{\TrMod_{\cC}(g^*)}", from=1-1, to=1-3]
	\arrow[dashed, "\alpha", from=2-1, to=2-3]
	\arrow["c"', from=2-3, to=1-3]
\end{tikzcd}
\end{equation*}
\end{defn}

\begin{rem}
From functoriality and symmetric monoidality of the $\Cc$-linear trace functor, maps of spaces whose free loop transfer restricts to constant loops are closed under composition and under cartesian products.
\end{rem}

Although the map $\alpha\colon \unit[C] \to \unit[B]$ in the above definition can a priori be any map, it follows a posteriori that it must be the Becker-Gottlieb transfer of $g$.

\begin{lem}
\label{lem:Strong_LM}
Let $g\colon B \to C$ be a $\cC$-adjointable map of spaces, and assume that the free loop transfer of $g$ restricts to some map $\alpha\colon \unit[C] \to \unit[B]$ on constant loops. Then there is a homotopy $\alpha \simeq g^!$ of maps $\unit[B] \to \unit[C]$. In particular, the following diagram commutes:
\begin{equation}
\label{eq:Strong_LM}
\begin{tikzcd}
	{\one[LC]} && {\one[LB]} \\
	{\one[C]} && {\one[B].}
	\arrow["c", from=2-1, to=1-1]
	\arrow["{\TrMod_{\cC}(g^*)}", from=1-1, to=1-3]
	\arrow["g^!", from=2-1, to=2-3]
	\arrow["c"', from=2-3, to=1-3]
\end{tikzcd}
\end{equation}
\end{lem}
\begin{proof}
As the evaluation map $e\colon LB \to B$ is a retraction of the map $c\colon B \to LB$, the assumption implies
\[
    \alpha = e \circ c \circ \alpha \simeq e \circ \TrMod_{\Cc}(g^*) \circ c \simeq g^!,
\]
where the last equivalence is \Cref{thm:LM}. The claim follows.
\end{proof}

If the free loop transfer of $g$ restricts to constant loops, composability of Becker-Gottlieb transfers is automatic.

\begin{prop}
\label{prop:Criterion_Composability_BG}
Let $f\colon A \to B$ and $g\colon B \to C$ be a $\cC$-adjointable maps of spaces. Assume that the free loop transfer of $g$ restricts to constant loops. Then the Becker-Gottlieb transfers for $f$ and $g$ are composable: $f^! \circ g^! \simeq (g \circ f)^!$.
\end{prop}
\begin{proof}
Consider the following diagram:
\[\begin{tikzcd}
	& {\one[LC]} && {\one[LB]} & {\one[LA]} \\
	{\one[C]} && {\one[B]} &&& {\one[A].}
	\arrow["c", from=2-1, to=1-2]
	\arrow["{\TrMod_{\Cc}(f^*)}"', from=1-4, to=1-5]
	\arrow["c", from=2-3, to=1-4]
	\arrow["e", from=1-5, to=2-6]
	\arrow["{g^!}"', from=2-1, to=2-3]
	\arrow["{f^!}"', from=2-3, to=2-6]
	\arrow["{\TrMod_{\Cc}((gf)^*)}"{description}, curve={height=-18pt}, from=1-2, to=1-5]
	\arrow["{\TrMod_{\Cc}(g^*)}"', from=1-2, to=1-4]
\end{tikzcd}\]
From \Cref{thm:LM} we obtain that the right trapezoid commutes and that the composite along the top of the diagram is the Becker-Gottlieb transfer $(gf)^!\colon \unit[C] \to \unit[A]$. The left parallelogram commutes by \Cref{lem:Strong_LM}. The top part of the diagram commutes by functoriality of the $\Cc$-linear trace. Going around the diagram then gives the claim.
\end{proof}

Because of \Cref{prop:Criterion_Composability_BG}, we are interested in maps of spaces $g\colon B \to C$ whose free loop transfer restricts to constant loops. When $\Cc$ is stable, it follows from the topological Riemann-Roch theorem \cite[2.7]{williams2000bivariantRR} that this holds whenever $g$ is a smooth fiber bundle which has compact manifold fibers. This will be discussed below, see \Cref{prop:SmoothBundlesConstantLoops}. The special case where $g$ is a finite covering map also admits a completely formal proof, which will be given in \Cref{subsec:Relative_Free_Loops}. 

The question of whether or not the free loop transfer of an arbitrary map $g$ restricts to constant loops is surprisingly subtle: as we will explain in the next remark, it is closely related to the Bass trace conjecture \cite[Strong conjecture]{bass}, an open conjecture in geometric group theory.

\begin{rmk}[Relation to Bass' trace conjecture]
\label{rem:Bass}
Let $\cC = \Sp$ be the $\infty$-category of spectra and let $C = \pt$ be the point. We may identify the free loop transfer $\Tr_{\Sp}(B^*)\colon \SS \to \SS[LB]$ of $B$ with a class in the 0-th stable homotopy group of $LB$: $\Tr_{\Sp}(B^*) \in \pi_0(\SS[LB])$. The condition that the free loop transfer restricts to constant loops is saying that this class is in the image of the constant loop map
\[
    c_B\colon \pi_0(\SS[B]) \hookrightarrow \pi_0(\SS[LB]).
\]
This condition is closely related to the Bass trace conjecture \cite[Strong conjecture]{bass}, which states that for any group $G$ the Hattori-Stallings trace $$K_0(\ZZ[G]) \to \THH_0(\ZZ[G]) \simeq \bigoplus_{G/\operatorname{cong}} \ZZ$$ lands in the direct summand indexed by the neutral element $e$ of $G$. Concretely, the following two conditions are equivalent for a given finitely presented group $G$:
\begin{enumerate}[(1)]
    \item The Bass trace conjecture holds for $G$;
    \item For every connected compact space $B$ with $\pi_1(B)\simeq G$, the free loop transfer $\Tr_{\Sp}(B^*) \in \pi_0(\SS[LB])$ lies in the image of the constant loop map $c_B\colon \pi_0(\SS[B]) \hookrightarrow \pi_0(\SS[LB])$.
\end{enumerate}
See \cite[1.1]{kleinMalkiewich2018transfer} for a related discussion. 
\end{rmk}

\subsubsection{Smooth fiber bundles}
An important instance of a situation in which the commutativity of (\ref{eq:Strong_LM}) is known is when $\cC$ is the $\infty$-category of spectra and $g\colon B \to C$ is a smooth fiber bundle with compact manifold fibers. This is a consequence of the topological Riemann-Roch theorem \cite[2.7]{williams2000bivariantRR}, using the natural map from $A$-theory to topological Hochschild homology. A proof for \textit{closed} manifold fibers, directly for topological Hochschild homology, was given by \cite[Corollary~1.9]{LM}. As the result for spectra in fact implies the result for an arbitrary stable presentably symmetric monoidal $\infty$-category $\cC$, we will phrase the result in this generality.

\begin{prop}[{\cite[2.7]{williams2000bivariantRR}, \cite[Corollary~1.9]{LM}}]
\label{prop:SmoothBundlesConstantLoops}
Let $\cC$ be stable, and let $g\colon B \to C$ be a smooth fiber bundle with compact manifold fibers, so that $g$ is $\cC$-adjointable by \Cref{ex:CompactSpacesAdjointable}. Then the square 
\begin{equation*}
\begin{tikzcd}
	{\one[LC]} && {\one[LB]} \\
	{\one[C]} && {\one[B].}
	\arrow["c", from=2-1, to=1-1]
	\arrow["{\TrMod_{\cC}(g^*)}", from=1-1, to=1-3]
	\arrow["g^!", from=2-1, to=2-3]
	\arrow["c"', from=2-3, to=1-3]
\end{tikzcd}
\end{equation*}
commutes.
\end{prop}
\begin{proof}
The case $\cC = \Sp$ is \cite[Corollary~1.9]{LM}. The case for arbitrary $\cC$ follows, as the unique colimit-preserving symmetric monoidal functor $F\colon \Sp \to \cC$ preserves traces and categorical traces.
\end{proof}

\begin{cor}
\label{cor:Composability_BG_Smooth_Bundle}
Assume $\cC$ is stable and let $g\colon B \to C$ be a smooth fiber bundle with closed manifold fibers. Then for any other $\cC$-adjointable map $f\colon A \to B$, the Becker-Gottlieb transfers for $f$ and $g$ are composable: $f^! \circ g^! \simeq (g \circ f)^!$.\qed
\end{cor}

\subsubsection{Via additivity of Becker-Gottlieb transfers}
We may use the additivity of Becker-Gottlieb transfers (\Cref{lem:Additivity_BG_Transfers}) to recursively obtain more cases of composability of Becker-Gottlieb transfers. Roughly speaking, the next two lemmas say that the collection of pairs of maps $f\colon A \to B$ and $g\colon B \to C$ with composable Becker-Gottlieb transfers is closed under pushouts in $A$ and $B$.

\begin{lem}
\label{lem:ComposabilityViaAdditivity1}
Let $\Cc$ be stable and let $f\colon A \to B$ and $g\colon B \to C$ be $\cC$-adjointable maps of spaces. Assume that $A$ is given as a pushout of the form
\[
\begin{tikzcd}
    A_3 \rar{\alpha} \dar[swap]{\beta} \drar[pushout] & A_1 \dar{j_1} \\
    A_2 \rar{j_2} & A,
\end{tikzcd}
\]
and let $j_3\colon A_3 \to A$ be the diagonal map. Assume further that each of the maps $f_i := f \circ j_i\colon A_i \to B$ is $\cC$-adjointable. If the Becker-Gottlieb transfer of $g$ composes with that of $f_i$ for every $i = 1,2,3$, then also the Becker-Gottlieb transfers of $g$ and $f$ compose.
\end{lem}
\begin{proof}
This follows immediately from additivity of Becker-Gottlieb transfers:
\begin{align*}
    (gf)^! &\simeq j_1 \circ (gf_1)^! + j_2 \circ (gf_2)^! - j_3 \circ (gf_3)^! \\
    &\simeq j_1 \circ f_1^! \circ g^! + j_2 \circ f_2^! \circ g^! - j_3 \circ f_3^! \circ g^! \\
    &\simeq (j_1 \circ f_1^! + j_2 \circ f_2^! - j_3 \circ f_3^!) \circ g^! \\
    &\simeq f^! \circ g^!.
\end{align*}
\end{proof}

\begin{lem}
\label{lem:ComposabilityViaAdditivity2}
Let $\Cc$ be stable and let $f\colon A \to B$ and $g\colon B \to C$ be $\cC$-adjointable maps of spaces. Assume that $B$ is given as a pushout of the form
\[
\begin{tikzcd}
    B_3 \rar{\alpha} \dar[swap]{\beta} \drar[pushout] & B_1 \dar{h_1} \\
    B_2 \rar{h_2} & B,
\end{tikzcd}
\]
and let $h_3\colon B_3 \to B$ be the diagonal map. For every $i=1,2,3$, define the map $f_i\colon A_i \to B_i$ via the following pullback square:
\[
\begin{tikzcd}
    A_i \rar{j_i} \dar[swap]{f_i} \drar[pullback] & A \dar{f} \\
    B_i \rar{h_i} & B.
\end{tikzcd}
\]
Assume further that each of the maps $g_i := g \circ h_i\colon B_i \to C$ is $\cC$-adjointable. If the Becker-Gottlieb transfer of $g_i$ composes with that of $f_i$ for every $i = 1,2,3$, then also the Becker-Gottlieb transfers of $g$ and $f$ compose.
\end{lem}
\begin{proof}
This follows immediately from additivity and naturality of Becker-Gottlieb transfers:
\begin{align*}
    (gf)^! &\simeq h_1 \circ (g_!f_1)^! + h_2 \circ (g_2f_2)^! - h_3 \circ (g_3f_3)^! \\
    &\simeq h_1 \circ f_1^! \circ g_1^! + h_2 \circ f_2^! \circ g_2^! - h_3 \circ f_3^! \circ g_3^! \\
    &\simeq f^! \circ j_1 \circ g_1^! + f^! \circ j_2 \circ g_2^! - f^! \circ j_3 \circ g_3^! \\
    &\simeq f^! \circ (j_1 \circ g_1^! + j_2 \circ g_2^! - j_3 \circ g_3^!) \\
    &\simeq f^! \circ g^!.
\end{align*}
\end{proof}

\Cref{lem:ComposabilityViaAdditivity2} has the following immediate corollary:

\begin{corollary}
\label{cor:ComposabilityClosedUnderFiniteColimits}
Assume $\Cc$ is stable. Let $C$ be a space, and let $\Spc^{\mathrm{comp}}_{/C}$ the subcategory of $\Spc_{/C}$ spanned by those $\cC$-adjointable morphisms $g\colon B \to C$ such that for every other $\cC$-adjointable morphism $f\colon A \to B$ the Becker-Gottlieb transfers of $f$ and $g$ are composable. Then $\Spc^{\mathrm{comp}}_{/C}$ is closed under finite colimits in $\Spc_{/C}$.
\end{corollary}
\begin{proof}
It is clear that $\Spc^{\mathrm{comp}}_{/C}$ contains the initial object $\emptyset \to C$, as any map $A \to \emptyset$ of spaces is an equivalence. It thus remains to show $\Spc^{\mathrm{comp}}_{/C}$ is closed under pushouts. By \Cref{lem:Stable_Ambi_Pushouts}, the $\cC$-adjointable morphisms $B \to C$ are closed under finite colimits in $\Spc_{/C}$. The result now follows directly from \Cref{lem:ComposabilityViaAdditivity2}.
\end{proof}

For example, \Cref{cor:Composability_BG_Smooth_Bundle} shows that any smooth fiber bundle $g\colon B \to C$ with closed manifold fibers belongs to $\Spc^{\mathrm{comp}}_{/C}$.

\subsubsection{Equivariant transfers}

In \cite[Theorem~IV.7.1]{LMS}, Lewis, May and Steinberger prove a result about composability of Becker-Gottlieb transfers in the setting of genuine equivariant homotopy theory for a compact Lie group $G$. This result has an analogue in the setting of spectra with $G$-action, which we will now discuss.

\begin{defn}
     Let $G$ be a compact Lie group and let $N \leqslant G$ be a closed normal subgroup. An orbit $G/H$ of $G$ is called \textit{$N$-free} if the action of $N$ on $G/H$ is free, or equivalently if $N \cap H = e$. A space with $G$-action $X \in \Spc^{BG}$ is called an \textit{$N$-free finite $G$-CW-complex} if it is contained in the smallest subcategory of $\Spc^{BG}$ which contains the $N$-free orbits $G/H$ and is closed under finite colimits.
\end{defn}

\begin{prop}[{cf.\ \cite[Theorem~IV.7.1]{LMS}}]
\label{prop:LMS_Equivariant_Transfers}
Let $\Cc$ be a stable presentably symmetric monoidal $\infty$-category. Let $G$ be a compact Lie group, let $N \leqslant G$ be a closed normal subgroup, and let $J := G/N$ denote the quotient group. Let $B$ be an $N$-free finite $G$-CW-complex, and let $g$ be the composite $g\colon B_{hG} \to BG \to BJ$. 
\begin{enumerate}[(1)]
    \item For any $\cC$-adjointable map $f\colon A \to B_{hG}$, the Becker-Gottlieb transfers of $f$ and $g$ compose.
    \item If $f'\colon E \to B$ is a $G$-map with compact fibers, the induced map $f := f'_{hG}\colon E_{hG} \to B_{hG}$ on homotopy orbits is $\cC$-adjointable.
\end{enumerate}
\end{prop}
\begin{proof}
For part (1), we have to show that the map $g\colon B_{hG} \to BJ$ is contained in the subcategory $(\Spc_{/BJ})^{\mathrm{comp}} \subseteq \Spc_{/BJ}$ defined in \Cref{cor:ComposabilityClosedUnderFiniteColimits}. We will first prove the case where $B = G/H$ is an $N$-free orbit, i.e.\ $H \cap N = e$. In this case, we have $B_{hG} = (G/H)_{hG} \simeq BH$. Note that the assumption on $H$ guarantees that the composite $H \hookrightarrow G \twoheadrightarrow J$ is injective, allowing us to regard $H$ as a subgroup of $J$. It follows that the map $BH \to BJ$ is a smooth fiber bundle with fiber given by the compact manifold $J/H$. Consequently we deduce from \Cref{cor:Composability_BG_Smooth_Bundle} that the map $BH \to BJ$ is in $(\Spc_{/BJ})^{\mathrm{comp}}$.

The claim now follows for an arbitrary $N$-free finite $G$-CW-complex $B$, using \Cref{cor:ComposabilityClosedUnderFiniteColimits} and the fact that the functor $(-)_{hG}\colon \Spc^{BG} \to \Spc$ preserves colimits. This finishes the proof of (1).

For part (2), it follows by descent that the map $f$ fits in a pullback square
\[
\begin{tikzcd}
    E \dar[swap]{f'} \rar \drar[pullback] & E_{hG} \dar{f} \\
    B \rar & B_{hG}.
\end{tikzcd}
\]
Since the fibers of $f'$ are compact, so are the fibers of $f$, and thus $f$ is $\cC$-adjointable by \Cref{ex:CompactSpacesAdjointable}.
\end{proof}

By letting $N$ be the trivial subgroup of $G$, we obtain the following two immediate corollaries:

\begin{cor}
\label{cor:LMS_Equivariant_Transfers}
Let $G$ be a compact Lie group, let $B$ be a finite $G$-CW-complex, and let $f'\colon E \to B$ be a $G$-map with compact fibers. Then the Becker-Gottlieb transfers for the maps
\[
    E_{hG} \xrightarrow{f} B_{hG} \xrightarrow{g} BG
\]
compose.
\end{cor}

\begin{cor}
Let $G$ be a compact Lie group and let $K \subseteq H \subseteq G$ be nested closed subgroups. Then the Becker-Gottlieb transfers for the maps
\[
    BK \to BH \to BG
\]
compose.
\end{cor}

\subsection{The free loop transfer and relative free loops}
\label{subsec:Relative_Free_Loops}
Let $f\colon A \to B$ be a $\cC$-adjointable map of spaces. Because of \Cref{prop:Criterion_Composability_BG}, we are interested in a better understanding of the composite
\[
    \eta_f\colon \unit[B] \xrightarrow{c_B} \unit[LB] \xrightarrow{\TrMod_{\cC}(f^*)} \unit[LA],
\]
and in the question of whether or not it factors through the map $c_A\colon \unit[A] \to \unit[LA]$. The goal of this subsection is to show that, at least under the stronger assumption that all fibers of $f$ are $\cC$-adjointable, cf.\ \Cref{lem:FiberwiseAdjointableImpliesAdjointable}, this composite always canonically factors through the map $\one[LA \times_{LB} B] \to \one[LA]$. Here $LA \times_{LB} B$ is the space of `relative free loops', i.e.\ those free loops in $A$ which live in a single fiber of $f$. We will deduce this factorization formally from the functoriality of the free-loop space transfer in the fiberwise $\cC$-adjointable map $f$. To set up this functoriality, we first need to assemble the fiberwise $\cC$-adjointable maps into a  suitable $\infty$-category.

\begin{defn}
Let $\Spc^{\pb}$ denote the (non-full) wide subcategory of the arrow category $\Spc^{[1]}$ containing all arrows as objects but only the pullback squares as morphisms. 
With $\cC$ implicit in the notation, we let $\Spc^{\pb}_\fadj$ denote the full subcategory of $\Spc^\pb$ spanned by arrows whose fibers are $\cC$-adjointable spaces.
For $f\in \Spc^{\pb}$, we denote its source by $A_f$ and its target by $B_f$. 
\end{defn}

Observe that the constructions $f\mapsto A_f$ and $f\mapsto B_f$ assemble into functors $A_{(-)},B_{(-)}\colon \Spc^{\pb} \to \Spc$, given by the following two compositions:
\[
A_{(-)}\colon \Spc^{\pb}_\fadj \hookrightarrow \Spc^{[1]} \xrightarrow{\ev_0} \Spc, \qquad \qquad B_{(-)}\colon \Spc^{\pb}_\fadj \hookrightarrow \Spc^{[1]} \xrightarrow{\ev_1} \Spc.
\]
In particular, we obtain functors $\unit[B_{(-)}]\colon \Spc^{\pb} \to \Cc$ and $\unit[LA_{(-)}]\colon \Spc^{\pb} \to \Cc$. We can now state and prove the naturality of the map $\eta_f\colon \unit[B_f] \to \unit[LA_f]$ in $f$.

\begin{prop}
\label{cons:Functoriality_THH_Transfer}
There is a natural transformation $\eta \colon \one[B_{(-)}] \to \one[LA_{(-)}]$ of functors $\Spc^\pb_{\fadj} \to \cC$ whose component at $f \in \Spc^{\pb}_{\fadj}$ is the composition
\[
 \eta_f\colon    \unit[B_f] \xrightarrow{c_{B_f}} \unit[LB_f] \xrightarrow{\TrMod_{\cC}(f^*)} \unit[LA_f].
\]
\end{prop}
\begin{proof}
Observe first that the individual maps $c_{B_f} \colon \one[B_f]\to \one[LB_f]$ are the $f$-components of a natural transformation $\one[B_{(-)}]\to \one[LB_{(-)}]$. Indeed, this natural transformation is obtained from the natural transformation $c\colon \id_\Spc \to L$ by precomposing with $B_{(-)}\colon \Spc^\pb_\fadj \to \Spc$ and post-composing with $\one[-] \colon \Spc \to \cC$. Hence,  it remains to assemble the free loop transfers $\TrMod_{\cC}(f^*)\colon \unit[LB_f] \to \unit[LA_f]$ into a natural transformation 
$\unit[LB_{(-)}] \to \unit[LA_{(-)}]$. 

The functor $\cC[-]\colon \Spc \to \Mod_{\cC}$ induces a functor $\Spc^{[1]} \to \Mod_{\cC}^{[1]}$ on arrow categories. By adjunction, we may regard the latter as a morphism in $\Fun(\Spc^{[1]},\Mod_{\cC})$ from $f \mapsto \cC[A_f]$ to $f\mapsto \cC[B_f]$ which is pointwise given by $f_!\colon \cC[A_f] \to \cC[B_f]$.
If we restrict this natural transformation to $\Spc^\pb$, then the naturality squares of this transformation are right adjointable by \Cref{prop:FreeCofreeCLinearCategories}(\ref{it:FreeCLinearCategoryBivariant}), and it follows from \cite[Theorem~4.6]{Haugseng2021Lax} that this morphism admits a right adjoint in the $(\infty,2)$-category $\Fun(\Spc^\pb,\Mod_{\cC})$. Furthermore, this right adjoint is given pointwise by the adjoints $f^*\colon \cC[B] \to \cC[A]$.
If we further restrict to $\Spc^\pb_\fadj$, this natural transformation is valued in the subcategory $\Mod_\cC^{\dbl} \subseteq \Mod_{\cC}$. Indeed,  by definition, the functor $f^*\colon \cC[B] \to \cC[A]$ is an internal left adjoint for every $\cC$-adjointable map $f\colon A \to B$. Composing with $\TrMod_\cC\colon \Mod_{\cC}^{\dbl} \to \cC$ then gives the desired natural transformation $\TrMod_\cC(f^*)\colon \one[LB]\to \one[LA]$. 
\end{proof}

Having established the functoriality of the map $\eta_f = \TrMod_{\Cc}(f^*) \circ c_B\colon \unit[B] \to \unit[LA]$, we turn to deduce the main result of this section: the factorization of this map through the relative constant loops $\unit[LA \times_{LB} B]$.
We shall deduce this factorization from the fact that, as we shall show next, the functor $\one[B_{(-)}]$ is colimit preserving.
First, we show that its source admits all small colimits.

\begin{lem}
\label{lem:ColimitsOfCartesianDiagrams}
The $\infty$-category $\Spc^\pb_\fadj$ admits all colimits, and the inclusion $\Spc^\pb_\fadj \subseteq \Spc^{[1]}$ preserves colimits.
\end{lem}
\begin{proof}
By definition, a morphism of spaces is fiberwise $\cC$-adjointable if and only if each of its fibers are $\cC$-adjointable, and thus the class of fiberwise $\cC$-adjointable morphisms of spaces is a local class in the sense of \cite[Definition~6.1.3.8]{htt}. The result then follows from \cite[Theorem 6.1.3.5(3)]{htt}.
\end{proof}

\begin{cor}
The functor $\one[B_{(-)}]\colon \Spc_{\fadj}^\pb \to \cC$ is colimit preserving. 
\end{cor}
\begin{proof}
The functor $\Spc_{\fadj}^\pb \to \Spc^{[1]}$ is colimit preserving by \Cref{lem:ColimitsOfCartesianDiagrams} above. Furthermore, the evaluation functor $\ev_1 \colon \Spc^{[1]} \to \Spc$ and the functor $\one[-]\colon \Spc \to \cC$ are both colimit preserving. Composing these three functors we deduce that $\one[B_{(-)}]\colon \Spc^\pb_\fadj \to \cC$ is colimit preserving as well. 
\end{proof}

\begin{prop}
\label{lem:Assembly_LM}
    Let $f\colon A \to B$ be a fiberwise $\cC$-adjointable map of spaces. There exists a map $\alpha\colon \one[B] \to \one[LA\times_{LB} B]$ making the following diagram commute:    
    \[\begin{tikzcd}
    	{\one[LB]} && {\one[LA]} \\
    	{\one[B]} && {\one[LA\times_{LB} B].}
    	\arrow["c_B", from=2-1, to=1-1]
    	\arrow["{\TrMod_{\cC}(f^*)}", from=1-1, to=1-3]
    	\arrow["\alpha", dashed, from=2-1, to=2-3]
    	\arrow["p_{LA}"', from=2-3, to=1-3]
    \end{tikzcd}\]
\end{prop}
\begin{proof}
By descent, the object $(f\colon A \to B)$ of $\Spc_{\fadj}^{\pb}$ may be written as a colimit over $B$ of the individual maps $f_b\colon A_b \to \{b\}$, where $A_b$ denotes the fiber of $f$ over $b \in B$. As the construction $f \mapsto \eta_f$ is natural in $f$, we may consider the resulting assembly diagram
\[
\begin{tikzcd}
    \colim_{b \in B} \unit[\{b\}] \dar[swap]{\as} \ar{rr}{\colim_{b \in B}\eta_{f_b}} && \colim_{b \in B} \unit[LA_b] \dar{\as}\\
    \unit[B] \ar{rr}{\eta_f} && \unit[LA].
\end{tikzcd}
\]
Since the functor $f \mapsto \unit[B_f]$ preserves colimits, the left vertical map is an equivalence, and we deduce that the map $\eta_f\colon \unit[B] \to \unit[LA]$ factors through the assembly map $\colim_{b \in B} \unit[LA_b] \to \unit[LA]$. The claim thus follows from the observation that there is an equivalence $\colim_{b \in B} LA_b \simeq LA\times_{LB}B$ in $\Spc_{/LA}$.
\end{proof}

Using \Cref{lem:Assembly_LM}, we obtain quick formal proofs of two strengthenings of \Cref{thm:LM}, both already appearing in \cite{LM} when $\Cc = \Sp$. The first one shows that finite covering maps satisfy the criterion from \Cref{prop:Criterion_Composability_BG}, and thus always satisfy composability of Becker-Gottlieb transfers. In the stable setting, this is a special case of the more general statement for smooth fiber bundle with closed manifold fibers, see \Cref{prop:SmoothBundlesConstantLoops}, but the proof for this special case is much simpler.

\begin{cor}
\label{cor:Finite_Cover_Strong_LM}
    Let $\cC$ be semiadditive, and let $f\colon A \to B$ be a finite covering map, so that $f$ is fiberwise $\cC$-adjointable by \Cref{ex:AdjointableFromAmbi}. Then the following diagram commutes:
    \[\begin{tikzcd}
    	{\one[LB]} && {\one[LA]} \\
    	{\one[B]} && {\one[A].}
    	\arrow["c", from=2-1, to=1-1]
    	\arrow["{\TrMod_{\cC}(f^*)}", from=1-1, to=1-3]
    	\arrow["f^!", from=2-1, to=2-3]
    	\arrow["c"', from=2-3, to=1-3]
    \end{tikzcd}\]
\end{cor}
\begin{proof}
We will start by showing that there exists an equivalence $A \simeq LA \times_{LB} B$. Indeed, we claim that the square
\[
\begin{tikzcd}
    A & B \\
    LA & LB
    \arrow["f", from=1-1, to=1-2]
    \arrow["Lf", from=2-1, to=2-2]
    \arrow["c", from=1-2, to=2-2]
    \arrow["c"', from=1-1, to=2-1]
\end{tikzcd}
\]
is a pullback square of spaces. To see this, it suffices to check that for every point $b \in B$ the induced map on horizontal fibers over $b$ is an equivalence. This induced map is given by the inclusion $c\colon A_b \to L(A_b)$ of the constant loops for the fiber $A_b$. But since the map $f\colon A \to B$ is a finite covering, its fibers are discrete, and thus the map $c\colon A_b \to L(A_b)$ is an equivalence. This proves that $A \simeq LA \times_{LB} B$. Note that under this equivalence, the projection map $LA \times_{LB} B \to LA$ corresponds to the inclusion $c\colon A \to LA$ of constant loops.

It thus follows from \Cref{lem:Assembly_LM} that there exists a map $\alpha\colon \unit[B] \to \unit[A]$ making the following diagram commute:
    \[\begin{tikzcd}
    	{\one[LB]} && {\one[LA]} \\
    	{\one[B]} && {\one[A].}
    	\arrow["c", from=2-1, to=1-1]
    	\arrow["{\TrMod_{\cC}(f^*)}", from=1-1, to=1-3]
    	\arrow["\alpha", dashed, from=2-1, to=2-3]
    	\arrow["c"', from=2-3, to=1-3]
    \end{tikzcd}\]
The desired claim is now an instance of \Cref{lem:Strong_LM}.
\end{proof}

\begin{cor}
\label{cor:Composability_BG_Finite_Covering}
Assume $\cC$ is semiadditive and let $g\colon B \to C$ be a finite covering map. Then for any other $\cC$-adjointable map $f\colon A \to B$, the Becker-Gottlieb transfers for $f$ and $g$ are composable: $f^! \circ g^! \simeq (g \circ f)^!$.
\end{cor}
\begin{proof}
This is immediate from \Cref{cor:Finite_Cover_Strong_LM} and \Cref{prop:Criterion_Composability_BG}.
\end{proof}

For a fiberwise $\cC$-adjointable map $f\colon A \to B$, the map $\TrMod_{\cC}(f^*) \circ c$ might not necessarily factor through $c\colon \one[A] \to \one[LA]$. The following result, which appears as \cite[Theorem 7.10]{LM}, says that at least it will factor through the map $c\colon \one[A] = \one[B \times_B A] \xrightarrow{c \times_B A} \one[LB \times_B A]$:

\begin{lem}[Lind-Malkiewich {\cite[Theorem 7.10]{LM}}]
Let $f\colon A\to B$ be a fiberwise $\cC$-adjointable map of spaces. Then the following diagram commutes:
\[\begin{tikzcd}
	{\one[LB]} & {\one[LA]} & {\one[LB \times_B A]} \\
	{\one[B]} && {\one[A]}
	\arrow["{\TrMod_{\cC}(f^*)}", from=1-1, to=1-2]
	\arrow["c", from=2-1, to=1-1]
	\arrow["c", from=2-3, to=1-3]
	\arrow["{(Lf,e)}", from=1-2, to=1-3]
	\arrow["{f^!}"', from=2-1, to=2-3]
\end{tikzcd}\]
where the right vertical map is induced by the composite $A = B \times_B A \xrightarrow{c \times_B A} LB\times_B A\to A$.
\end{lem}
\begin{proof}
Consider the following commutative diagram:
\[\begin{tikzcd}
	{\one[LB]} & {\one[LA]} & {\one[LB \times_B A]} \\
	{\one[B]} & {\one[LA \times_{LB} B]} & {\one[A].}
	\arrow["{\TrMod_{\cC}(f^*)}", from=1-1, to=1-2]
	\arrow["{c_B}", from=2-1, to=1-1]
	\arrow["c"', from=2-3, to=1-3]
	\arrow["{(Lf,e)}", from=1-2, to=1-3]
	\arrow["{p_{LA}}"', from=2-2, to=1-2]
	\arrow["e \circ p_{LA}", from=2-2, to=2-3]
	\arrow["\alpha", from=2-1, to=2-2]
\end{tikzcd}\]
Here the left square is the commutative square from \Cref{lem:Assembly_LM}, and the right square commutes as it is induced from a commuting square on the level of spaces. (Note that $LA \times_{LB} B$ is in fact the pullback of the two maps $(Lf,e)\colon LA \to LB \times_{B} A$ and $c\colon A \to LB \times_B A$.) 

Since the map $c \times_B A\colon A \to LB \times_B A$ admits a retraction $p_A\colon LB \times_B A \to A$, it follows from \Cref{thm:LM} that the bottom composite of this diagram must be $f^!\colon \unit[B] \to \unit[A]$, finishing the proof.
\end{proof}

\section{Thom spectra}
\label{sec:ThomSpectra}

Given a presentably symmetric monoidal $\infty$-category $\cC$, a pointed space $A$ and a pointed map $\xi\colon A \to \cC$, the \textit{Thom object} $\Th(\xi) \in \cC$ is, from a modern perspective, simply the colimit of $\xi$. When $A$ is connected, one can think of $\xi$ as encoding an action of the group $\Omega A$ on the monoidal unit $\one \in \cC$, and of $\Th(\xi)$ as the homotopy orbits of this action.
The interest in this construction stems from the fact that if $A = G$ is an $\EE_n$-group\footnote{It is convenient to define an \textit{$\EE_0$-group} to be a pointed \textit{connected} space.} in $\Spc$ (i.e.\ an $n$-fold loop space) and $\xi$ is an $\EE_n$-map with respect to the tensor product on $\cC$, then the object $\Th(\xi)$ inherits a canonical structure of an $\EE_n$-algebra in $\cC$. Several fundamental examples of $\EE_n$-ring spectra, including the cobordism rings $MG$ (for $G = O,\, SO,\, U,\, \mathit{Sp},\, \mathit{Spin},...$),  arise in this way.  

\begin{rem}
    for every $n\ge 0$, an $\EE_n$-map $G \to \cC$ factors uniquely through an $\EE_n$-map $G \to \Pic(\cC)$, where the latter denotes the $\EE_\infty$-group of invertible objects in $\cC$. Hence, it is usually the latter that is taken as the input to the Thom object construction. We shall use both perspectives interchangeably. 
\end{rem}

For $G$ an $\EE_1$-group, we may consider the Hochschild homology of the $\EE_1$-algebra $\Th(\xi)$, 
\[
    \THH_\cC(\Th(\xi)) := 
    \TrMod_{\Cc}(\RMod(\Th(\xi)))
    \qin \cC.
\]
In the case where $\cC = \Sp$ is the $\infty$-category of spectra, Blumberg, Cohen and Schlichtkrull \cite{BCS2010THHofThomSpectra} computed this object: they showed that it is itself the Thom spectrum (i.e.\ colimit) of a certain map $LBG \too \Sp$ associated with $\xi$. The goal of this section is to give an alternative proof of this theorem, generalizing it to an arbitrary presentably symmetric monoidal $\infty$-category $\cC$.

\begin{thm}[cf.\ {\cite[Theorem~1]{BCS2010THHofThomSpectra}}]
\label{thm:BCS_THH_of_Thom}
    Let $G$ be an $\EE_1$-group in $\Spc$, and $\xi\colon G\to \Pic(\cC)$ be an $\EE_1$-group map. Then $\THH_\cC(\Th(\xi)) \in \cC$ is the Thom object of the following composite: 
    \[
        LBG \oto{LB\xi} 
        LB\Pic(\cC) \simeq 
        B\Pic(\cC)\times \Pic(\cC) \oto{\eta + \id} \Pic(\cC),
    \]
    where $\eta \colon B\Pic(\cC) \to \Pic(\cC)$ is the Hopf map.
\end{thm}

\begin{rem}
    The other main results of \cite{BCS2010THHofThomSpectra}, namely Theorems 2 and 3, give a simplified formula for the Thom spectrum $\Th(\xi)$, when $\xi$ is a map of $\EE_2$- and $\EE_3$-groups respectively. These follow easily from \cite[Theorem~1]{BCS2010THHofThomSpectra} by inspecting the decomposition $LBG \simeq BG\times G$, and in particular generalize verbatim to an arbitrary symmetric monoidal $\infty$-category $\cC$.  
\end{rem}   

Our proof differs from the one in  \cite{BCS2010THHofThomSpectra}, which relies on the cyclic bar construction model for Hochschild Homology. Instead, we show in \Cref{thm:Thom_as_colim} that the module category $\RMod_{\Th(\xi)}(\cC)$ is the colimit of the composition
\[
    BG \oto{\:\:B\xi\:\:} 
    B\Pic(\cC) \longhookrightarrow 
    \Mod_\cC,
\]
and apply the formula for the dimension of a colimit from \Cref{cor:Trace_Dimension_Colimit}. In other words, by treating Hochschild Homology as (categorified) monoidal dimension, it becomes amenable to (categorified) character theory. 

There is, however, a mild technical difficulty arising from this categorification process, as it is convenient to work with \textit{presentable} $\infty$-categories throughout, but $\PrL$ and hence $\Mod_\cC(\PrL)$ are not themselves presentable. We overcome this by adopting the solution of \cite[\S 5.3.2]{HA}, replacing `presentable' with `$\kappa$-compactly generated' for a sufficiently large regular cardinal $\kappa$. To make this transparent for the casual reader who does not wish to be bothered by set-theoretical technicalities, we employ the following notational convention:

\begin{convention}\label{Convention}
    For every presentable $\infty$-category $\cC$, there exists an uncountable regular cardinal $\kappa$ such that $\cC$ is $\kappa$-compactly generated. We shall always implicitly choose such $\kappa$ and treat $\cC$ as an object of the $\infty$-category $\PrL_\kappa$ of $\kappa$-compactly generated $\infty$-categories, which is a presentably symmetric monoidal subcategory of $\PrL$ by \cite[Lemma~5.3.2.9 and Lemma~5.3.2.12]{HA}. If needed, we shall allow ourselves to implicitly replace $\kappa$ by some larger $\kappa'$ using the canonical (non-full) inclusion $\PrL_\kappa \into \PrL_{\kappa'}.$
    If $\cC$ is furthermore presentably $\EE_n$-monoidal, we let $\Mod_\cC$ be the presentably  $\EE_{n-1}$-monoidal $\infty$-category $\Mod_\cC(\PrL_\kappa)$, for $\kappa$ as in \cite[Lemma 5.3.2.12]{HA}. 
\end{convention}

Philosophically speaking, while $\PrL$ is not itself presentable, it is a (large) small-filtered colimit of the presentable $\infty$-categories $\PrL_\kappa$ in the $\infty$-category of large small-cocomplete $\infty$-categories and small colimit preserving functors.
Hence, every construction in $\PrL$ (resp.\ property thereof) involving only a small amount of data, can be already carried out (resp.\ witnessed) in $\PrL_\kappa$ for some, and hence all sufficiently large, regular cardinals $\kappa$.

This section is organized as follows: in \Cref{subsec:thomobj}, we define Thom objects with their multiplicative structures, and prove that categories of modules over them can themselves be understood as (categorified) Thom objects; and in \Cref{subsec:HHThom} we combine this with our formula for the trace of colimits to recover and generalize the result of \cite{BCS2010THHofThomSpectra} regarding the topological Hochschild homology of Thom spectra.

\subsection{Thom objects and categorical group algebras}\label{subsec:thomobj}

Let $\cC$ be a presentably symmetric monoidal $\infty$-category. For every $\EE_n$-group $G$ in $\Spc$ (i.e.\ an $n$-fold loop space), the $\cC$-linear $\infty$-category $\cC[G]$ inherits an $\EE_n$-monoidal structure by symmetric monoidality of the functor $\cC[-]\colon \Spc \to \Mod_{\cC}$. We refer to $\cC[G]$ as the \textit{categorical group algebra}. In this preliminary subsection, we collect some facts about $\cC[G]$ and use them to describe the $\EE_n$-algebra structure on the Thom object $\Th(\xi)$ associated with an $\EE_n$-map $G \to \cC$. 

\subsubsection{Thom objects}
Given an $\EE_n$-group $G$, the $\EE_n$-group map $G \to \pt$ induces an adjunction
\[
    G_! \colon \cC[G] \adj \cC \colon G^*,
\]
in which $G_!$ is $\EE_n$-monoidal and hence $G^*$ is lax $\EE_n$-monoidal.

\begin{war}
    While the underlying $\infty$-category of $\cC[G]$ is equivalent to the functor category $\cC^G$, the $\EE_n$-monoidal structure on $\cC[G]$ does not identify with the \textit{pointwise} one on $\cC^G$ (for which it is $G^*\colon \cC \to \cC^G$ which is strong $\EE_n$-monoidal while $G_!\colon \cC^G \to \cC$ is merely oplax $\EE_n$-monoidal). In fact, the $\EE_n$-monoidal structure on $\cC[G]$ corresponds to the \textit{Day convolution} on $\cC^G = \Fun(G,\cC)$, though we shall not use this fact. 
\end{war}

The $\EE_n$-monoidal adjunction $G_! \dashv G^*$ allows for a simple description of the multiplicative structure on Thom objects. 

\begin{defn}\label{def:Th}
     For $n\ge 0$, the symmetric monoidal adjunction between $\Cc[-]\colon \Spc \to \Mod_{\Cc}$ and the forgetful functor induces an adjunction on $\EE_n$-algebras, hence any $\EE_n$-map $\xi\colon G \to \cC$ induces an $\EE_n$-monoidal functor $\xi_\cC \colon \cC[G] \to \cC$. The composite lax $\EE_n$-monoidal functor
    \[
        \cC \oto{\:\:G^*\:\:} 
        \cC[G] \oto{\:\:\xi_\cC\:\:} 
        \cC
    \]
    induces in turn a functor on $\EE_n$-algebras. We define 
    \[
        \Th_\cC(\xi) := \xi_\cC(G^*\one) 
        \qin \Alg_n(\cC).
    \]
    This construction naturally promotes to a functor
    \[
        \Th_\cC \colon
        \Fun_{\EE_n}(G, \cC) \too 
        \Alg_n(\cC).
    \]
\end{defn}
We will drop the subscript $\cC$ and write simply $\Th$ when it is harmless to do so. 

Note that by \Cref{cor:RestrictionInduction}, the underlying object of $\Th_\cC(\xi)$ is indeed simply $\colim_G \xi \in \cC$. Furthermore, the definition of $\Th_\cC(\xi)$ is functorial in $\cC$ in the following sense:

\begin{lem}\label{Th_Functorial}
    Given a map $F\colon \cC \to \cD$ in $\CAlg(\PrL)$. For every $n$-monoidal $\xi \colon G \to \cC$, we have a natural isomorphism 
    \[
        F(\Th_\cC(\xi)) \simeq \Th_\cD(F(\xi)) 
        \qin \Alg_n(\cD).
    \] 
\end{lem}
\begin{proof}
    Consider the diagram of lax $n$-monoidal functors
    \[\begin{tikzcd}
    	\cC && {\cC[G]} && \cC \\
    	\cD && {\cD[G]} && \cD
    	\arrow["F", from=1-5, to=2-5]
    	\arrow["F", from=1-1, to=2-1]
    	\arrow["{F[G]}", from=1-3, to=2-3]
    	\arrow["{G^*}", from=1-1, to=1-3]
    	\arrow["{G^*}", from=2-1, to=2-3]
    	\arrow["{\xi_\cC}", from=1-3, to=1-5]
    	\arrow["{F(\xi)_\cD}", from=2-3, to=2-5]
    \end{tikzcd}\]
    The left square commutes by naturality. For the right square, note that $F(\xi)_{\cD}$ is the unique $\cD$-linear functor whose restriction to $G$ is given by $F(\xi)\colon G \to \cD$. Restricting it to $\Cc[G]$ will thus give a $\Cc$-linear functor whose restriction to $G$ is $F(\xi)$. Since $F \circ \xi_{\cD}$ is another such functor, we see that also the right square commutes. Comparing the two paths in the resulting diagram from the left upper corner to the bottom right corner gives the desired isomorphism.
\end{proof}

It is not immediately clear that our definition of the $\EE_n$-algebra structure on $\Th_\cC(\xi)$ agrees with other definitions in the literature, such as  \cite{ando2018parametrized}, \cite{ACB2018ThomSpectra} and \cite{beardsley2017relative}. We address this point by showing that every definition which is functorial in $\cC$, in the sense of the preceding lemma, is isomorphic to our definition.

\begin{prop}\label{Th_Unique}
    Let $G$ be an $\EE_n$-group and suppose that for every $\cC \in \CAlg(\PrL)$ we have a functor 
    \[
        \Th'_\cC \colon 
        \Fun_{\EE_n}(G, \cC) \too
        \Alg_n(\cC),
    \]
    lifting $\colim_G$ along $\Alg_n(\cC) \to \cC$, and such that for every $n$-monoidal functor $\cC \to \cD$ we have a natural isomorphism
    \[
        F(\Th'_\cC(\xi)) \simeq \Th'_\cD(F(\xi)).
    \]
    Then for every $\cC$ we have an isomorphism of functors $\Th'_\cC \simeq \Th_\cC$.
\end{prop}

\begin{proof}
    We first use the functoriality in $\cC$ to reduce to the universal case. Namely, every $\xi \colon G \to \cC$ factors essentially uniquely as 
    \[
        G \oto{\:u\:} 
        \Spc[G] \oto{\:\xi_\Spc \:}
        \cC,
    \]
    where $u$ is the adjunction unit map. Since $\xi_{\Spc}$ is an $n$-monoidal functor, we have a natural isomorphism 
    \[
        \Th'_\cC(\xi) = 
        \Th'_\cC(\xi_\Spc(u)) \simeq 
        \xi_\Spc(\Th'_{\Spc[G]}(u)).
    \]
    Since by \Cref{Th_Functorial} the same holds for $\Th$, it remains to show that there is an isomorphism 
    \[
        \Th'_{\Spc[G]}(u) \simeq 
        \Th_{\Spc[G]}(u)
        \qin \Alg_n(\Spc[G]).
    \] 
    This in turn follows from the fact that the common underlying object 
    \[
        \colim\nolimits_G u \simeq
        G^*\pt 
        \qin \Spc[G]
    \] 
    is \textit{terminal}, and hence admits an essentially unique $\EE_n$-algebra structure. 
\end{proof}

\begin{rem}
\label{rem:EnhancementThomFunctor}
    One can show that the individual functors $\Th_\cC$ assemble into
    \[
        \Th \colon 
        \CAlg(\PrL)_{G/} \too 
        \mathrm{Pr}^{\mathrm{L},\Alg_n},
    \] 
    where the target is the $\infty$-category of pairs $(\cC,R)$ with $\cC \in \Alg_n(\PrL)$ and $R\in\Alg_n(\cC)$. The global functor $\Th$ assigns to $\xi \colon G \to \cC$ the pair $(\cC, \Th_\cC(\xi))$. The functoriality of this formula is a coherent version of \Cref{Th_Functorial}. Furthermore, the argument of \Cref{Th_Unique} can be adopted to show that $\Th$ is essentially uniquely characterized as a lift of the functor taking $\xi\colon G \to \cC$ to $(\cC, \colim_G \xi)$ along $\mathrm{Pr}{}^{\mathrm{L},\Alg_n} \to \mathrm{Pr}{}^{\mathrm{L},\mathrm{Triv}}$. We shall not prove nor use this.
\end{rem}

\subsubsection{Projection formula}

Considering the underlying $\EE_1$-monoidal structures, the monoidal adjunction $G_! \dashv G^*$ is furthermore $\cC[G]$-linear in the following sense:

\begin{lem}\label{Proj_Formula_Day}
    The monoidal adjunction
    \[
        G_! \colon \cC[G] \adj \cC \colon G^*
    \]
    satisfies the projection formula. Namely, for every $X \in \cC[G]$ and $Y\in \cC$, the canonical map
    \[
        X \otimes G^*Y \too G^*(G_! X \otimes Y)
        \qin \cC[G]
    \]
    is an isomorphism. 
\end{lem}
\begin{proof}
    The $\infty$-category $\cC[G] \simeq \cC^G$ is generated under colimits by the objects of the form $g_! Z$ where $Z\in\cC$ and $g\colon \pt \to G$ (see \cite[Lemma~4.3.8]{HLAmbiKn}). Both the source and target of the projection formula map are colimit preserving and $\cC$-linear in the variable $X$. By $\cC$-linearity of the functor $g_!$, it thus suffices to show that the projection formula map is an isomorphism for $X = g_! \one$ for every $g \in G$. The counit map $G \to \cC[G]^\times$ is a group homomorphism taking $g \in G$ to $g_!\one$. Thus, $g_! \one$ is invertible and hence dualizable. Finally, by \cite[Proposition~3.12]{FHM2003}, the projection formula map is always an isomorphism when $X$ is dualizable. 
\end{proof}

Since $G^* \colon \cC \to \cC[G]$ is lax monoidal, $G^*\one \in \cC[G]$ admits a canonical algebra structure. This is however \textit{not} the unit of $\cC[G]$. In fact, a $G^*\one$-module structure on an object of $\cC[G] \simeq \cC^G$ exhibits it as ``constant'', i.e.\ as lying in the image of the functor $G^*$. More formally,

\begin{lemma}\label{Day_Constant}
    The functor $G^*\colon \cC \to \cC[G]$ induces an equivalence
    \[
        \cC \simeq \RMod_{G^* \unit} (\cC[G])
        \qin \LMod_{\cC[G]},
    \]
    where $\cC$ is a $\Cc[G]$-linear $\infty$-category via the symmetric monoidal left adjoint $G_!\colon \cC[G] \to \cC$.
\end{lemma}
\begin{proof}
    The adjunction 
    \[
        G_!\colon \cC[G] \rightleftarrows 
        \cC\noloc G^*
    \] 
    factors canonically as a composition of adjunctions
    \[
        \cC[G] \rightleftarrows 
        \RMod_{G^*\one}(\cC[G]) \rightleftarrows 
        \cC.
    \]
    We shall show that the right adjunction is in fact an equivalence. The right adjoints in the above adjunctions fit into a commutative triangle:
    \[\begin{tikzcd}
    	{\quad \cC\quad } && {\RMod_{G^*\one}(\cC[G])} \\
    	& {\cC[G].}
    	\arrow["U", from=1-3, to=2-2]
    	\arrow["G^*"', from=1-1, to=2-2]
    	\arrow[from=1-1, to=1-3]
    \end{tikzcd}\]
    
    The functor $U$ is manifestly monadic with left adjoint given by $F(X) =  X\otimes G^*\one $. The functor $G^*$ is conservative and preserves all colimits, as it admits a further right adjoint $G_*$. By \cite[Corollary 4.7.3.16]{HA}, it suffices to check that the induced map of functors 
    \(
        UF \to G^*G_!
    \)
    is an isomorphism. Unwinding the definitions, for every $X \in \cC[G]$, this map is the projection formula map
    \[
        UF(X) = 
        X\otimes G^* \one  \too 
        G^*( G_!X\otimes\one) \simeq
        G^*G_!(X).
    \]
    Thus, we are done by \Cref{Proj_Formula_Day}.
\end{proof}

\subsection{Hochschild Homology of Thom objects}\label{subsec:HHThom}

In this subsection, we give the proof of \Cref{thm:BCS_THH_of_Thom}, the computation of the Hochschild homology of the Thom object $\Th(\xi)$ of an $\EE_1$-group map $\xi\colon G \to \cC$. We shift the perspective slightly in that we start with the pointed connected space $A := BG$ and a pointed map $\zeta := B\xi \colon A \to \Mod_{\cC}$, and identify the Thom object (i.e.\ colimit) $\Th(\zeta) \in \Mod_\cC$ with the $\infty$-category of modules over the Thom object $\Th(\Omega \zeta) \in \Alg(\cC)$.

\subsubsection{Modules over Thom objects}
\label{subsec:ModulesOverThomSpectra}

The first part of our strategy is to show that the module category of a Thom object in $\cC$ admits a description as a certain colimit in $\Mod_\cC$. 
Let us denote by $\cC_{\Omega \zeta}$ the $\infty$-category $\cC$ as an object of $(\Alg_\cC)_{\cC[\Omega A]/}$
via the algebra map $\cC[\Omega A] \to \cC$ associated with $\Omega \zeta$. When $\zeta$ is the trivial map, we shall write simply $\cC$. As an auxiliary step we establish the following formula for $\Th(\zeta)$:

\begin{prop}\label{Th_Zeta_Tensor_Formula}
     For every pointed map $\zeta\colon A \to \Mod_{\cC}$, there is an isomorphism
     \[
        \Th(\zeta) \simeq 
        \cC_{\Omega \zeta} \otimes_{\cC[\Omega A]} \cC
        \qin \Mod_\cC.
     \]
\end{prop}
\begin{proof}
    By definition, the object $\Th(\zeta)  \in \Mod_{\cC}$ is given by the evaluation at $\cC \in \Mod_\cC$ of the composition
    \[
        \Mod_\cC \oto{\:A^*\:} 
        \Mod_\cC[A] \oto{\:\:\zeta_{\Mod_\cC}\:\:} 
        \Mod_\cC.
    \]
   Further, we have the following commutative diagram:
    \[\begin{tikzcd}
    	{\Mod_\cC} &&& {\Mod_\cC[A]} &&& {\Mod_\cC} \\
    	{\Mod_\cC} &&& {\Mod_{\cC[\Omega A]}} &&& {\Mod_\cC.}
    	\arrow[Rightarrow, no head, from=1-1, to=2-1]
    	\arrow[Rightarrow, no head, from=1-7, to=2-7]
    	\arrow["{A^*}", from=1-1, to=1-4]
    	\arrow["\zeta_{\Mod_\cC}", from=1-4, to=1-7]
    	\arrow["{\cC_{\Omega \zeta}\otimes_{\cC[\Omega A]} (-)}", from=2-4, to=2-7]
    	\arrow["{\Res_{(\Omega A)_!}}", from=2-1, to=2-4]
    	\arrow["\wr", from=1-4, to=2-4]
    \end{tikzcd}\]
    The middle vertical equivalence is the equivalence from \Cref{Cat_Grp_Alg_Mod}, applied to the $\infty$-category $\Mod_{\cC}$. The commutativity of the right square follows from \Cref{Aug_Base_Change} applied to the $\infty$-category $\Mod_\cC$ and the unit commutative algebra  $\cC \in \Mod_\cC$. The commutativity of the left square follows, by passing to right adjoints, from the commutativity of the right square in the special case where $\zeta$ is trivial. The result follows by evaluating the composites of the two paths in the outer rectangle at the object $\cC\in \Mod_\cC$.
\end{proof}

\begin{thm}
\label{thm:Thom_as_colim}
    Let $A$ be a pointed connected space and let $\zeta\colon A \to \Mod_\cC$ be a pointed map. There is an equivalence
    \[
        \RMod_{\Th(\Omega \zeta)} \:\simeq\:
        \Th(\zeta) 
        \qin \Mod_\cC.
    \]
\end{thm}

\begin{proof}
    By \Cref{Th_Zeta_Tensor_Formula} we have
    \[
        \Th(\zeta) \simeq 
        \cC_{\Omega \zeta} \otimes_{\cC[\Omega A]} \cC
        \qin \Mod_\cC.
    \]
    Recall that by \Cref{Day_Constant} we have 
    \[
        \cC \simeq 
        \RMod_{(\Omega A)^*\one}(\cC[\Omega A])
        \qin \LMod_{\cC[\Omega A]}.
    \]
    Let 
    $(\Omega \zeta)_\cC \colon \cC[\Omega A] \to \cC$
    be the monoidal functor corresponding to 
    $\Omega \zeta \colon \Omega A \to \cC.$
    By \cite[Theorem~4.8.4.6]{HA}, we get
    \[
        \cC_{\Omega \zeta} \otimes_{\cC[\Omega A]} \RMod_{(\Omega A)^*\one}(\cC[\Omega A]) \simeq
        \RMod_{(\Omega \zeta)_\cC((\Omega A)^*\one)}(\cC).
    \]
    By definition, we have an equality
    $(\Omega \zeta)_\cC((\Omega A)^*\one) = \Th(\Omega \zeta)$ of algebra objects in $\cC$. Altogether, we obtain an equivalence
    \[
        \Th(\zeta) \simeq 
        \RMod_{\Th(\Omega \zeta)}
        \qin \Mod_\cC,
    \]
    finishing the proof.
\end{proof}
\subsubsection{Categorical character formula}
\label{subsec:THHofThomSpectra}

Using the colimit presentation of \Cref{thm:Thom_as_colim}, we have
\[
    \THH_\cC(\Th(\xi)) \simeq
    \TrMod_{\Cc}(\RMod_{\Th(\xi)}) \simeq 
    \TrMod_{\Cc}(\colim\nolimits_A \zeta)
    \qin \cC,
\]
where the first equivalence holds by \Cref{prop:THH_As_Categorical_Trace}. The map $\zeta$ is pointwise dualizable, as its value at each point is the unit $\cC \in \Mod_\cC$. By \Cref{cor:Trace_Dimension_Colimit} (cf.\ also \Cref{ex:ambiPrL} and \Cref{rem:kappaambi}), we can compute the trace (i.e.\ dimension) of its colimit via the composition
\[
    \cC \oto{\Tr_{\Mod_\cC}(A^*)}
    \cC[LA] \oto{\quad \chi_\zeta \quad}
    \cC
    \qin \Mod_\cC.
\]

We begin by analyzing the (categorified) free loop transfer $\Tr_{\Mod_\cC}(A^*) \colon \cC \to \cC[LA]$. 
\begin{prop}
    For every space $A$, we have
    \[
        \Tr_{\Mod_\cC}(A^*) \simeq (LA)^*
        \qin \Fun(\cC, \cC[LA]).
    \]
\end{prop}
\begin{proof}
    We shall make use of the fact that $\Mod_{\Mod_\cC}$ is an $(\infty,3)$-category, and hence the trace functor
    \[
        \Tr_{\Mod_\cC} \colon 
        \Mod_{\Mod_\cC}^\dbl \too
        \Mod_\cC
    \] 
    is an $(\infty,2)$-functor. In particular it preserves the adjunction $A_! \dashv A^*$. By \Cref{thm:THHFreeLoops} we have
    \[
        \Tr_{\Mod_\cC}(A_!) \simeq (LA)_!
        \qin \Fun(\cC, \cC[LA]),
    \]
    and thus the result follows from passing to right adjoints.
\end{proof}

This implies a formula for $\TrMod_{\Cc}(\colim_A \zeta)$ for a general pointwise dualizable $\zeta \colon A\to \Mod_\cC$.

\begin{cor}\label{Cat_Colim_Dim_Formula}
    For every space $A$ and a pointwise dualizable map $\zeta \colon A \to \Mod_\cC$ we get 
    \[
        \TrMod_{\Cc}(\colim\nolimits_A \zeta) \:\simeq\:
        \colim\nolimits_{LA} (\chi_\zeta)
        \qin \cC.
    \]
\end{cor} 

Next, to analyze the (categorified) character map $\chi_\zeta \colon LA \to \cC$, we consider the universal case. 

\begin{prop}\label{Univ_Cat_Char}
    Let $\zeta \colon B\Pic(\cC) \into \Mod_\cC$ be the inclusion of the connected component of $\cC$. The character map $\chi_\zeta \colon LB\Pic(\cC) \to \cC$ identifies with 
    \[
        LB\Pic(\cC) \simeq 
        \Pic(\cC) \times B\Pic(\cC) \oto{1+\eta}
        \Pic(\cC) \sseq 
        \cC,
    \]
    where $\eta$ is the Hopf map. 
\end{prop}
\begin{proof}
    Under the inclusion $B\Pic(\cC) \into \Mod_\cC$, the isomorphism 
    \[
        \Pic(\cC) \times B\Pic(\cC) \iso LB\Pic(\cC)
    \]
    is given by
    \[
        (X, \cD) \longmapsto (\cD \oto{X} \cD).
    \]
    Since $\Cc$-linear trace functor is symmetric monoidal (\Cref{rem:CLinearTraceSymmetricMonoidal}), we have
    \[
        \TrMod_{\Cc}(\cD \oto{X} \cD) \simeq 
        \TrMod_{\Cc}(\cC \oto{X} \cC) \otimes
        \TrMod_{\Cc}(\cD \oto{\mathrm{Id}} \cD) \qin \cC
    \]
    Thus, it suffices to observe that for $X \in \Pic(\cC)$ we have 
    \[
        \TrMod_{\Cc}(\cC \oto{X} \cC) = X
        \qin \Pic(\cC),
    \] 
    and that $\TrMod_{\Cc} \colon B\Pic(\cC) \to \Pic(\cC)$ is given by applying the Hopf map $\eta$ \cite[Proposition~3.20]{CSY2021Cyclotomic}. 
\end{proof}

By combining the above results we obtain the desired description of the Hochschild homology of a Thom object. 

\begin{proof}[Proof of \Cref{thm:BCS_THH_of_Thom}]

By \Cref{thm:Thom_as_colim} and \Cref{Cat_Colim_Dim_Formula}, we have 
\[
    \THH_\cC(\Th(\Omega \zeta)) \simeq
    \TrMod_{\Cc}(\colim\nolimits_{BG} \zeta) \simeq 
    \colim\nolimits_{LBG} (\chi_\zeta)
    \qin \cC.
\]

Now, by \Cref{Univ_Cat_Char}, the map $\chi_\zeta$ identifies with the composition
\[
    LBG \oto{\:L\zeta\:} LB\Pic(\cC) \simeq
    \Pic(\cC) \times B\Pic(\cC) \oto{1+\eta} 
    \Pic(\cC)
\]
and the claim follows.
\end{proof}

\section*{Declarations}

\textbf{Data availability:} Data sharing not applicable to this article as no datasets were generated or analysed during
the current study.

\textbf{Conflict of interest:} On behalf of all authors, the corresponding author states that there is no conflict of
interest.

\bibliographystyle{alpha}
\bibliography{references}

\addcontentsline{toc}{section}{References}

\end{document}